\documentclass[11pt]{amsart}
\usepackage{setspace}
\usepackage{xcolor}
\usepackage{makecell}
\usepackage{amssymb,amscd,amsthm,verbatim,amsmath,color,fancyhdr,mathrsfs,amsfonts,amssymb,commath, graphicx,bbm}
\usepackage[bookmarks=false]{hyperref}
\usepackage[english]{babel}
\usepackage{parskip}
\usepackage{enumitem}

\usepackage[letterpaper,top=2cm,bottom=2cm,left=3cm,right=3cm,marginparwidth=1.75cm]{geometry}

\usepackage{tikz}
\usetikzlibrary{arrows.meta,positioning,calc}

\DeclareMathOperator{\Id}{\mathrm{Id}}

\newcommand{\Was}[1]{\mathbb{W}_{#1}}

\newtheorem{corollary}[]{Corollary}
\newcommand{\Var}{\mathrm{Var}}

\DeclareMathOperator*{\argmin}{argmin}

\newcommand{\cP}{\mathcal{P}}

\newcommand{\cL}{\mathcal{L}}
\newcommand{\cM}{\mathcal{M}}

\newcommand{\Ent}{\mathrm{Ent}}

\newcommand{\Tr}{\text{Tr}}

\newcommand{\Tan}[1]{\text{Tan}(#1)}
\newcommand{\Schro}{\text{Schr\"{o}dinger}}
\newcommand{\Ito}{\text{It\^{o}}}
\newcommand{\Exp}[1]{\mathrm{E}_{#1}}

\newcommand{\vol}{\mathrm{vol}}

\newcommand{\ricci}{\mathrm{Ric}}

\newcommand{\eps}{\varepsilon}

\newtheorem{lemma}{Lemma}
\newtheorem{fact}{Fact}
\newtheorem*{definition}{Definition}

\newtheorem{theorem}{Theorem}
\newtheorem{informaltheorem}{Main Theorem}
\newtheorem*{cor*}{Corollary}
\newtheorem{assumption}{Assumption}
\newtheorem{proposition}{Proposition}

\theoremstyle{definition}
\newtheorem{remark}{Remark}

\newcommand{\commentout}[1]{}

\title[Diffusion Approximations to Schr\"{o}dinger bridges]{Diffusion Approximations to Schr\"{o}dinger Bridges and the Convergence of Entropic Potentials}

\author{Garrett Mulcahy}
\address{Garrett Mulcahy\\ Department of Mathematics \\ University of Washington\\ Seattle WA 98195, USA\\ {Email: gmulcahy@uw.edu}}
\author{Soumik Pal}
\address{Soumik Pal\\ Department of Mathematics \\ University of Washington\\ Seattle WA 98195, USA\\ {Email: soumik@uw.edu}}

	

\thanks{This research is partially supported by the following grants: NSF DMS-2502281, 2133244, 2052239, and
PIMS PRN-01 (Kantorovich Initiative).}

\date{\today}

\begin{document}

\begin{abstract}
Consider the Monge-Kantorovich optimal transport problem between two Euclidean densities $\mu$ and $\nu$ and quadratic cost. The $\varepsilon$-Schr\"{o}dinger bridge is the solution to the entropic regularized problem with regularization parameter $\varepsilon$. By taking the conditional expectation of the second coordinate given the first under this coupling, we obtain the $\varepsilon$-entropic Brenier map. As $\varepsilon$ goes down to zero, it is known that the entropic Brenier map converges to the quadratic cost optimal transport map between the two measures, i.e., the Brenier map. We show that, under some smoothness and log-concavity constraints on the marginals, the difference between the entropic Brenier map and the Brenier map is equal to $\eps$ times one-half of the score function of the first marginal plus an error that is $o(\eps)$ in $\mathbf{L}^2(\mu)$. This expansion holds irrespective of the second marginal $\nu$. The proof relies on an approximation of the Schr\"{o}dinger bridge by a noisy version of the McCann interpolation. Under additional assumptions we show a second approximation to the Schr\"odinger bridge via so-called Mirror Langevin diffusions, which are Langevin diffusions on the Hessian manifold generated by the Brenier map. These two approximations, that appear quite different at first glance, are nonetheless shown to be closely connected. Our proofs utilize a random surface and a novel stochastic operation called the tangent Markov projection.   
\end{abstract}

\maketitle

\section{Introduction}\label{sec:introduction}
Let $\cP_2(\mathbb{R}^{d})$ denote the set of probability measures on $\mathbb{R}^{d}$ with finite second moments, and let $\mu,\nu \in \cP_2(\mathbb{R}^{d})$. We make the standard abuse of notation of using the same symbol to refer to both a measure and its Lebesgue density (when it exists). The 2-Wasserstein distance between $\mu$ and $\nu$, denoted $\Was{2}$, is defined by the following Monge Kantorovich quadratic cost optimal transport problem
\begin{align}\label{defn:was2}
    \Was{2}^{2}(\mu,\nu) := \inf\limits_{\pi \in \Pi(\mu,\nu)} \int_{\mathbb{R}^{d} \times \mathbb{R}^{d}} \norm{x-y}^2 d\pi. 
\end{align}
Here $\Pi(\mu,\nu) := \{\gamma \in \cP(\mathbb{R}^{d} \times \mathbb{R}^{d}): (\pi_{x})_{\#}\gamma = \mu, (\pi_{y})_{\#}\gamma = \nu \}$ is the set of all couplings of $\mu$ and $\nu$. 
When $\mu$ has Lebesgue density, \cite{THEbrenier} established that the optimal transport plan $\pi^*$ in \eqref{defn:was2} is induced by a $\mu$-a.s.\ unique transport map, i.e.\ a measurable function $T: \mathbb{R}^{d} \to \mathbb{R}^{d}$ such that the pushforward of $\mu$ by $T$, denoted $T_{\#}\mu$, is equal to $\nu$. This optimal transport map is given by a gradient of a convex function $\nabla \varphi$. We call this the \textbf{Brenier map}. 

A related object is the $\Schro$ bridge from $\mu$ to $\nu$. Fix $\eps > 0$ and let $H(\cdot|\cdot)$ denote the relative entropy (Kullback-Leibler divergence) between two positive measures on the same measurable space
\begin{align}
    H(\mu_1|\mu_2) := \begin{cases}
        \Exp{\mu_1}\left[\log\left(d\mu_1/d\mu_2\right)\right] &\text{ if $\mu_1 \ll \mu_2$} \\
        +\infty &\text{ otherwise}
    \end{cases}
\end{align}
The notation $\mu \ll \nu$ denotes that $\mu$ is absolutely continuous with respect to $\nu$. We write $\Ent(\mu) := H(\mu|\mathrm{Leb})$. The (static) $\eps$-$\Schro$ bridge from $\mu$ to $\nu$ is the solution to the regularization of \eqref{defn:was2} with relative entropy
\begin{align}\label{defn:intro-static-sb}
    \pi^{\eps} := \argmin\limits_{\pi \in \Pi(\mu,\nu)} \left( \int_{\mathbb{R}^{d}} \frac{1}{2}\norm{x-y}^2d\pi  + \eps H(\pi|\mu\otimes \nu)\right).
\end{align}
Under mild assumptions on $\mu$ and $\nu$, $\pi^{\eps}$ exists and is unique \cite{schroLeonard13}. As $\eps \downarrow 0$, it is known that $\pi^{\eps}$ converges in various senses to a minimizer of \eqref{defn:was2} \cite{leo-sb-to-kp12}. By \cite[Theorem 2.8]{schroLeonard13}, the Radon-Nikodym derivative of the $\Schro$ bridge with respect to its reference measure has a product decomposition. To be precise, let $(r_{t}(\cdot,\cdot), t > 0)$ denote the standard Euclidean heat kernel. Then there exist functions $a^{\eps},b^{\eps}: \mathbb{R}^{d} \to [0,+\infty)$ unique up to multiplication by a positive constant such that 
\begin{align}\label{eq:fg-decomp}
    \frac{d\pi^{\eps}}{dr_{\eps}}(x,z) = a^{\eps}(x)b^{\eps}(z).  
\end{align}
A common change of variables employed is to define $f_{\eps} = \eps \log a^{\eps}$ and $g_{\eps} = \eps \log b^{\eps}$, which are called entropic potentials (also $\Schro$ potentials) in the literature. We note the following connection between these potentials and the entropic Brenier map; this can be seen from \cite[eqn (8)]{chewi2022entropic}, although we use a slightly different convention
\begin{align}\label{defn:entropic-potent}
   \nabla f_{\eps}(x) &:= \eps \nabla \log a^{\eps}(x) = x - \mathrm{E}_{\pi^{\eps}}[Y|X=x] - \eps \nabla f(x). 
\end{align}

We now set $\mu = e^{-f}$ and $\nu = e^{-h}$ to make it clear that we consider probability measures with Lebesgue density. Observe that $\pi^{\eps}$ is a coupling whereas the Brenier map is a measurable function. For this reason, \cite{pooladian2022entropic} introduces the entropic Brenier map, defined as 
\begin{align}
    \mathcal{B}_{\eps}(x) := \Exp{\pi^{\eps}}[Y|X=x],
\end{align}
and propose using it as a biased estimator for the Brenier map, as $\pi^{\eps}$ is more amenable to computation and statistical estimation \cite{cuturi2013sinkhorn,sinkhorn-OG,statOTbook}. As $\eps \downarrow 0$, \cite[Corollary 1]{pooladian2022entropic} establishes the $L^2(\mu)$ convergence of $\mathcal{B}_{\eps}$ to $\nabla \varphi$ as $\eps \downarrow 0$. This convergence is established in more generality (and in slightly different language) in \cite[Theorem 1.1]{chiarini2022gradient}. 

One of our main results identifies the next order term in $\eps$ in the convergence of the entropic Brenier map to the Brenier map in the following setting.
\begin{assumption}\label{assumption:standard-assumptions}
    Let $e^{-f},e^{-h} \in \cP_2(\mathbb{R}^{d})$ 
    and let $\nabla \varphi$ denote the Brenier map from $e^{-f}$ to $e^{-h}$. Assume that
    \begin{itemize}
        \item[(A1)] $f,h \in C^{3}(\mathbb{R}^{d})$ and there exists $\Lambda_1, \Lambda_2 > 0$ such that
        \begin{align}\label{eq:strong-ellipticity}
            \Lambda_1 \Id \leq \nabla^2 f(x), \nabla^2 h(x) \leq \Lambda_2 \Id \text{ for all $x \in \mathbb{R}^{d}$,} 
        \end{align}
        where $\leq$ between matrices denotes the positive semidefinite ordering.
        \item[(A2)] $\varphi \in C^{4}(\mathbb{R}^{d})$ and all third and fourth derivatives of $\varphi$ are globally bounded. By Caffarelli's Theorem \cite{THEcaffarelli}, \eqref{eq:strong-ellipticity} implies that
        \begin{align}
            \sqrt{\Lambda_1/\Lambda_2} \Id \leq \nabla^2 \varphi(x) \leq \sqrt{\Lambda_2/\Lambda_1}\Id \text{ for all $x \in \mathbb{R}^{d}$.}
        \end{align}
    \end{itemize}
\end{assumption}

\begin{informaltheorem}[Theorem \ref{thm:sb-o-eps-score}]\label{informal-theorem-1}
    Under Assumption \ref{assumption:standard-assumptions}, the following limit holds in $L^2(e^{-f})$ irrespective of $h$, 
    \begin{align}\label{eq:intro-score-function-limit}
        \lim\limits_{\eps \downarrow 0}\frac{1}{\eps}\left(\Exp{\pi^{\eps}}\left[Y|X=x\right]-\nabla \varphi(x)\right) &= -\frac{1}{2}\nabla f(x). 
    \end{align}
  Equivalently, for the entropic potentials $(f_{\eps}, \eps > 0)$ defined in \eqref{defn:entropic-potent}, it holds in $L^2(\mu)$ that
    \begin{align}
        \lim\limits_{\eps \downarrow 0} \frac{1}{\eps}(\nabla f_{\eps}(x) - (x-\nabla \varphi(x))) = -\frac{1}{2}\nabla f(x). 
    \end{align}
\end{informaltheorem}
The quantity $-\nabla f$ is called the \textbf{score function} of $e^{-f}$, and it is important in applications such as statistics and machine learning. As far as identifying the limit \eqref{eq:intro-score-function-limit}, there are a few recent attempts that are worth mentioning. First, \cite[Propositions 8, 10]{mordant24selfEOT} presents an informal argument for obtaining the limit \eqref{eq:intro-score-function-limit} based on a WKB ansatz. Second, there is a parallel effort based on geometric arguments that we believe are connected and complementary to our probabilistic arguments here.    
In \cite{legervialard23}, the authors develop a Laplace method involving the Kim-McCann geometry \cite{KimMcCann} that computes integrals that decay exponentially away from the graph of the optimal transport. As the authors note,  ``entropic regularization method leads to such integrals'' and announce that a future separate article will cover the Taylor expansion of the entropic potentials with respect to the regularization parameter. It was shown by \cite{WongYang} that the Kim-McCann geometry is related to the dualistic structure of information geometry, a connection that was utilized in an earlier approximation of the low temperature Schr\"odinger bridge in \cite{pal2019difference}. We believe our current probabilistic construction also has a parallel interpretation in this geometry that we leave for future investigations. See Section \ref{sec:sketch-of-proof} for more or follow the discussion around \cite[eqn. (11)]{MP25}.

We pause to make some remarks about the limit in \eqref{eq:intro-score-function-limit}. First, we note that the limit depends only on the score function of the first marginal. The appearance of only one marginal is perhaps surprising, but it has implications for the stability of (entropic) Brenier maps as studied in \cite{divol2024tightstabilityboundsentropic}. To be precise, fix $\mu, \nu_1,\nu_2 \in \cP_2(\mathbb{R}^{d})$ and let $\mathcal{B}_{\eps}^{\nu_i}$, $\nabla \varphi^{\nu_i}$ denote the entropic Brenier map and Brenier map, respectively, from $\mu$ to $\nu_i$ for $i = 1,2$. By Theorem \ref{thm:sb-o-eps-score}, the following expansions hold in $L^2(\mu)$
\begin{align}
    \mathcal{B}_{\eps}^{\nu_1} = \nabla \varphi^{\nu_1}+\frac{\eps}{2}\nabla \log \mu + o(\eps), \quad \mathcal{B}_{\eps}^{\nu_2} = \nabla \varphi^{\nu_2}+\frac{\eps}{2}\nabla \log \mu + o(\eps).
\end{align}
Thus, $(\mathcal{B}_{\eps}^{\nu_1}-\mathcal{B}_{\eps}^{\nu_2})$ and $(\nabla \varphi^{\nu_1}-\nabla \varphi^{\nu_2})$ differ by a quantity in $L^2(\mu)$ that is $o(\eps)$, which offers an unexpected improvement to the more immediate $o(1)$. This $o(1)$ difference is used to pass stability of entropic Brenier maps to the Brenier maps in \cite[Corollary 3.5]{divol2024tightstabilityboundsentropic}. We note that this result requires compact support, whereas Main Theorem \ref{informal-theorem-1} requires full support.

The crux of our approach in proving \eqref{eq:intro-score-function-limit} is to introduce the dynamic picture. In the notation of Main Theorem \ref{informal-theorem-1}, define for $t \in [0,1]$ the function $\nabla \varphi_{0 \to t}(x) = (1-t)x + t \nabla \varphi(x)$. The \textbf{McCann interpolation}, introduced properly in \eqref{eq:mccann-interp}, is the curve of probability measures $\rho_t^0 = (\nabla \varphi_{0 \to t})_{\#}e^{-f}$ for $t \in [0,1]$. Let $\nabla \varphi_{s\rightarrow t}$ denote the Brenier map transporting $\rho_s^0$ to $\rho_t^0$. We introduce a stochastic process called the \textbf{$\eps$-noisy McCann interpolation}, defined as satisfying the following SDE
\begin{align}\label{eq:noisy-mccann-intro}
    dX_t = \left(\nabla \varphi_{t \to 1}-\nabla \varphi_{t \to 0}+\frac{\eps}{2}\nabla \log \rho_t^{0}\right)(X_t)dt + \sqrt{\eps}dB_t, \quad X_0 \sim e^{-f}. 
\end{align}
This process has the same marginal flow as the McCann interpolation but the joint distribution of $(X_0, X_1)$ is absolutely continuous, while still being a coupling of $e^{-f}$ and $e^{-h}$. We show that $\Exp{}[X_1|X_0 =x]$ has the same $L^2(e^{-f})$ limit as \eqref{eq:intro-score-function-limit} (Proposition \ref{prop:d-dt-matrix-identity}). Let $\widetilde{\ell}_{\eps} = \mathrm{Law}(X_0,X_1)$, then the proof of Main Theorem \ref{informal-theorem-1} follows from establishing that $\widetilde{\ell}_{\eps}$ is an $o(\eps)$ approximation in relative entropy to $\pi^{\eps}$. While the $\eps$-noisy McCann interpolation provides an approximation to $\pi^{\eps}$ that is sufficient for computing \eqref{eq:intro-score-function-limit}, the qualitative properties of $\pi^\eps$ remain unclear. For example, it is not apparent from the SDE what the conditional covariance matrix of $X_1$, given $X_0=x$, is. We expect approximately the same answer as under $\pi^\eps$ \cite{pal2019difference}, $\eps \nabla^2 \varphi(x)$. 

This leads us to the second thrust of this paper. To obtain this qualitative insight into small temperature $\Schro$ bridges, we direct our attention to a diffusion process called the Mirror Langevin diffusion \cite{zhang-mld-20a,ahn2021efficient, MP25, CKP26}.

Let $\nabla \varphi$ denote the Brenier map from $e^{-f}$ to $e^{-h}$. Write $\varphi^*$ to denote the convex conjugate of $\varphi$, and as shorthand, define $x^* := \nabla \varphi(x)$ for $x \in \mathbb{R}^{d}$.  The \text{primal} Mirror Langevin diffusion (MLD) with respect to $e^{-f},e^{-h}$ and $\nabla \varphi$ is defined as the solution to the following SDE, where $Y_t = X_t^*$,
\begin{align}\label{defn:MLD}
    dX_t = -\frac{1}{2}\nabla h(Y_t)dt + \sqrt{\nabla^2 \varphi^*(Y_t)} dB_t.
\end{align}
The process $(Y_t, t \geq 0)$ satisfies a complementary SDE given by
\begin{align}\label{defn:dMLD}
    dY_t = -\frac{1}{2}\nabla f(X_t)dt + \sqrt{\nabla^2 \varphi(X_t)} dB_t.
\end{align}
We say that $(X_t, t \geq 0)$ is the MLD in \textbf{primal} coordinates and $(Y_t,t \geq 0)$ is the MLD in \textbf{dual} coordinates. This nomenclature is due to the fact that the two SDEs are two coordinate representations of a single diffusion on a manifold. Equip $\mathbb{R}^{d}$ with a Riemannian metric given by the Hessian of $\varphi$. This space is called a \textbf{Hessian manifold} \cite{shima2007geometry}, a real-valued analogue of Kähler manifolds \cite[Section 3]{kolesnikov-hessian-metric}. The MLD is then an analogue of the Langevin diffusion on the Hessian manifold with \eqref{defn:MLD} and \eqref{defn:dMLD} being its representations in dual affine coordinate charts. Under assumptions, the first SDE in \eqref{defn:MLD} is reversible with stationary measure equal to $e^{-f}$; for the second SDE the stationary measure is $e^{-h} = (\nabla \varphi)_{\#}e^{-f}$.

The setting considered in our earlier work \cite[Theorem 4]{MP25} is the following. Let $(X_t,t \geq 0)$ satisfy \eqref{defn:MLD} with initial condition $X_0 \sim e^{-f}$. As this is the stationary measure, $X_t \sim e^{-f}$ and $Y_t \sim e^{-h}$ for all $t \geq 0$. It then holds for all $\eps > 0$ that $(X_0,Y_{\eps}) = (X_0,X_{\eps}^*)$ is a coupling of $e^{-f}$ and $e^{-h}$. Set $\ell_{\eps} := \mathrm{Law}(X_0,X_{\eps}^*)$, then it follows from the SDE in \eqref{defn:MLD} that in $L^2(e^{-f})$
\begin{align}\label{eq:barycentric-proj-mld}
    \Exp{\ell_{\eps}}[Y|X=x] = \Exp{}[Y_{\eps}|X_0=x]= \Exp{}[Y_{\eps}|Y_{0}=\nabla \varphi(x)]= \nabla \varphi(x) - \frac{\eps}{2}\nabla f(x) + o(\eps). 
\end{align}
Observe that this matches the expansion of $\Exp{\pi^{\eps}}[Y|X=x]$ in \eqref{eq:intro-score-function-limit}. If $H(\ell_{\eps}|\pi^{\eps})$ were $o(\eps)$, then Theorem \ref{thm:sb-o-eps-score} would follow from the expansion in \eqref{eq:barycentric-proj-mld}. However, in \cite[Theorem 4]{MP25} we show that $H(\ell_{\eps}|\pi^{\eps}) = o(1)$ and we conjecture that it cannot be improved past $O(\eps)$. 

The key innovation in the current work is the following. Take $\ell_{\eps}$ as previously defined and sample $(X_0,X_1) \sim \ell_{\eps}$. Form a stochastic process over the time interval $[0,1]$ such that the conditional distribution $(X_t, t \in [0,1])$, given  $(X_0,X_{1})$, is the $\eps$-Brownian bridge joining $X_0$ to $X_1$. Equivalently, this process is the so-called \textbf{reciprocal process}, given by
\begin{align}\label{defn:intro-stochastic-proc}
    t \in [0,1] \mapsto (1-t)X_0 + tX_1 + \sqrt{\eps}Z_{t},
\end{align}
where $(Z_t, t \in [0,1])$ is the standard Brownian bridge, independent of the coupling $(X_0,X_1)$. 

Reciprocal processes are expected in this study since the dynamic Schr\"odinger bridge (defined later in \eqref{eq:dynam-schro-bridge-defn}) is itself one. 
However, as opposed to the dynamic Schr\"odinger bridge, other reciprocal processes are not Markov processes. In fact, the dynamic $\Schro$ bridge is the unique intersection of the set of reciprocal processes (with fixed bridge distribution and endpoint marginals) and the set of Markov processes, a fact used in the Iterative Markov Fitting (IMF) procedure for computing $\Schro$ bridges proposed by \cite{shi2023diffusion,peluchetti-23}.

Since the reciprocal process with endpoint distribution $\ell_\eps$ is non-Markovian, we perform an operation called the \textbf{tangent Markov projection} to make it Markovian. This operation is a modification of the Markov projection studied in \cite{gyongy-mimicking-86,brunick-mimicking13}, and it is defined properly later in Section \ref{subsec:markov-proj}. 
For now it suffices to say that a Markov projection of a stochastic process is a Markov process (typically a diffusion) that has the same marginal distribution as the original process at each point in time. Hence, both the stochastic process and its Markov projection trace out the same curve of marginal probabilities, but have different joint distributions across time. We introduce the concept of a tangent Markov projection of a stochastic process which, roughly, is a Markov projection diffusion for which the drift is a function of the gradient type. This is consistent with the metric theory of $2$-Wasserstein spaces. 

Denote the law of this new process by $\overline{Q}^{\eps}$, the connection between the Hessian geometry of the MLD and the $\Schro$ bridge is then given in the following result. 
\begin{informaltheorem}[Theorem \ref{thm:mp-mld-approx-of-sb}]\label{informal-theorem-2}
    For marginals $e^{-f}, e^{-h}$, let $\overline{Q}^{\eps}$ denote the law of the tangent Markov projection of the reciprocal process formed from the Mirror Langevin diffusion, see \eqref{defn:intro-stochastic-proc}. Let $P^{\eps}$ denote the dynamic $\Schro$ bridge. Under Assumptions \ref{assumption:standard-assumptions} and \ref{assumption:technical-rd} (see below), it holds that
    \begin{align}
        \lim\limits_{\eps \downarrow 0}\frac{1}{\eps}H(\overline{Q}^{\eps}|P^{\eps}) = 0. 
    \end{align}
\end{informaltheorem}
Let $\widetilde{Q}^{\eps}$ denote the law of the $\eps$-noisy McCann interpolation in \eqref{eq:noisy-mccann-intro}. The reader might be puzzled with the apparent difference between the two approximations for the $\Schro$ bridge by $\overline{Q}^{\eps}$ and $\widetilde{Q}^{\eps}$. In fact, they are close to each other as well. 

\begin{cor*}[Corollary \ref{prop:mp-drift-approx}]
   Under Assumptions \ref{assumption:standard-assumptions} and  \ref{assumption:technical-rd},
$\lim_{\eps \rightarrow 0+}\eps^{-1}H(\overline{Q}^{\eps}|\widetilde{Q}^{\eps}) = 0$.
\end{cor*}

In fact, the generator of the dual Mirror Langevin diffusion shows up explicitly in the proof of Proposition \ref{prop:d-dt-matrix-identity} where it is shown that $\Exp{}[X_1|X_0 =x]$ has the same $L^2(e^{-f})$ limit as \eqref{eq:intro-score-function-limit}. See the key matrix integral identity in Proposition \ref{prop:sqrd-inv-hessian-iden} and the related discussion below \eqref{eq:time-int-generator}.
The same correspondence also shows up in the proof of Lemma \ref{lem:zero-initial-velocity}. As we explain in the next section via the reciprocal surface construction this is due to a geometric ``miracle'' that fits the Hessian geometry transversely to the McCann interpolation. 

\textbf{Technical Assumptions and $\mathbb{T}^{d}$ Setting.} 
The conclusion of Main Theorem \ref{informal-theorem-2} holds under Assumptions \ref{assumption:standard-assumptions} and \ref{assumption:technical-rd}. While Assumption \ref{assumption:standard-assumptions}, assuming strong log-concavity of the marginals and the smoothness of Brenier maps, may be considered standard in this literature, Assumption \ref{assumption:technical-rd} is a more severe smoothness and growth requirement on the marginal densities of the reciprocal processes and their derivatives. Establishing this level of regularity solely under assumptions on $\mu$ and $\nu$ remains a technical challenge. In certain cases one may directly verify that these assumptions hold, such as when the Brenier map from $\mu$ and $\nu$ is an affine map (for example, when $\mu=\nu$) where it follows from \cite{AHMP25,MP25}. We provide some examples for which our assumptions hold: a Gaussian example in Section \ref{subsec:gaussian-computations}, and a non-Gaussian one in Proposition \ref{prop:example-class}. 

To further bolster our case, we also prove an analogous statement to the process approximation given in Main Theorem \ref{informal-theorem-2} on the flat $d$-dimensional torus $\mathbb{T}^{d}$ under the simplified Assumption \ref{assumption:torus}. In Section \ref{sec:torus-case} we state and prove an analogous result to this main theorem on $\mathbb{T}^{d}$. We retain the Euclidean statements with the technical assumptions for two reasons. One, in some Euclidean cases one can directly verify them, and, two, perhaps for some other cases it may follow from veins of existing literature (or future research) that we are not aware of.

\subsection{Related Works}\label{subsec:related-works} 
The most similar predecessors to this paper are previous works by the authors and collaborators in \cite{AHMP25,MP25}. \cite{AHMP25} analyzes the same marginal case, in which $e^{-f} = e^{-h}$, $x^* = \nabla \varphi(x) = x$. In this case, the noisy McCann \eqref{eq:noisy-mccann-intro} as well as both the Primal and the Dual MLD processes \eqref{defn:MLD} all reduce to the same standard Langevin diffusion. In this case, \cite[Theorem 1]{AHMP25} establishes that the symmetric relative entropy between $\pi^{\eps}$ and $\ell_{\eps} = \mathrm{Law}(X_0,X_{\eps}^*)$ is $o(\eps^2)$. For different marginals, \cite[Theorem 4]{MP25} establishes that $H(\ell_{\eps}|\pi^{\eps})$ is $o(1)$ under regularity assumptions on the Brenier map. 

\textbf{Small Temperature $\Schro$ Bridge.} The convergence of the $\Schro$ bridge to the Monge Kantorovich optimal transport problem is a rich area of research dating back to \cite{mikami04} and recounted in \cite{leo-sb-to-kp12,schroLeonard13}. It has been studied in various levels of generality, from $\mathbb{R}^{d}$ to manifolds and metric measure spaces. The motivation for such attention is the use of the small temperature $\Schro$ bridge and its associated quantities as proxies for the unregularized problem in machine learning and data science applications, a perspective initiated by \cite{cuturi2013sinkhorn,GalichonSalanie09}. For a presentation of the statistical aspects of (regularized) optimal transport, consult the recent text \cite{statOTbook}. 

There are many ways to analyze the convergence of the $\Schro$ bridge problem to the unregularized Monge Kantorovich problem in \eqref{defn:was2} as $\eps \downarrow 0$. One approach is to study the convergence of entropic cost (the optimal value of the objective function in \eqref{defn:intro-static-sb}) to $\Was{2}^{2}$. For a slightly more general cost, such an analysis is achieved via the so-called Gaussian approximation in \cite{pal2019difference}. A more abstract analytical argument in \cite{conforti21deriv} studies regularity of $\eps \in [0,+\infty) \mapsto H(\pi^{\eps}|r_{\eps})$ and establishes the first three terms of its Taylor expansion about $\eps = 0$. This Taylor expansion plays a key role in our proofs, see Section \ref{sec:sketch-of-proof}. More recent works have studied this entropic cost convergence for different classes of cost functions \cite{carlierEOTgeneralcost23,malamut-syl25,nutz2026entropicregularizationmongesproblem,aryan2025entropicselectionprinciplemonges}. More general $\Gamma$-convergence and large deviation statements are given in \cite{schroLeonard13,bgn-eot-gld}.

Similarly, the convergence of entropic potentials in \eqref{defn:entropic-potent} to their unregularized counterparts (and their gradients) has received a great deal of attention in the literature. For the convergence of the potentials to the so-called Kantorovich potential, \cite{nutz-weisel-22} establishes this for a large class of cost functions on Polish spaces via a weak compactness argument. The corresponding gradient limit in the Euclidean quadratic cost setting was established in \cite{pooladian2022entropic} as a consequence of an empirical process analysis. A more general gradient limit was established in \cite{chiarini2022gradient} in the manifold setting based on estimates provided from a curvature-dimension condition. Other related works have studied convexity bounds for the entropic potentials and how these persist under vanishing regularization \cite{nutzweisel-stab23,chewi2022entropic,divol2024tightstabilityboundsentropic,gozlan2025globalregularityestimatesoptimal,fathigozlan20}.  
There has been limited progress in determining the rate of convergence for potentials, a challenge that we take up here.

\subsection{AI Disclosure} 
The authors acknowledge the use of ChatGPT Sol 5.6 as well as Claude Opus 5 in the preparation and proofreading of this manuscript. In particular, ChatGPT Sol 5.6 drafted the kinetic energy bound in Theorem \ref{lem:ke-calc}. It revised an earlier Eulerian argument by the authors that had many difficult-to-verify assumptions into a more direct PDE-based argument. It also produced a simpler argument for Step 2 in Proposition \ref{prop:d-dt-matrix-identity} by suggesting an approach based on matrix divergences. It established Propositions \ref{prop:poly-bdd-nabla-log-semigrp} and \ref{prop:neg-sob-norm-poincare} and outlined the proof on the torus in Section \ref{sec:torus-case}. Lastly, it identified and verified the nontrivial class of probability measures in Proposition \ref{prop:example-class} under which the regularity assumptions hold. Claude Opus 5 helped minimize the redundancy in Assumptions \ref{assumption:standard-assumptions} and \ref{assumption:technical-rd} and proofread the paper.
The authors verified all calculations and arguments themselves. All the main theorems and arguments were original ideas of the authors. They also wrote the manuscript completely without any help from AI and retain all responsibility for its correctness.

\section{A sketch of the main ideas in the proofs}\label{sec:sketch-of-proof}
Let us give an overview of our arguments, presenting the intuition behind each step and the main technical innovations we introduce in this paper. 

Fix marginals $e^{-f},e^{-h} \in \cP_2(\mathbb{R}^{d})$ with Brenier map $(\nabla \varphi)_{\#}e^{-f} = e^{-h}$, and let $\pi^{\eps} \in \cP(\mathbb{R}^{d}\times \mathbb{R}^{d})$ be the $\eps$-static $\Schro$ bridge from $e^{-f}$ to $e^{-h}$. Recall that $(r_t(\cdot,\cdot),t > 0)$ is the standard Euclidean heat kernel. Let $P^{\eps} \in \cP(C^{d}[0,1])$ denote the dynamic $\Schro$ bridge, defined later in \eqref{eq:dynam-schro-bridge-defn}, and let $R^{\eps} \in \cP(C^{d}[0,1])$ denote the law of $\eps$-reversible Brownian motion. That is, $R^{\eps}$ is the law of $(B_{\eps t}, t \in [0,1])$ where $B_0 \sim \mathrm{Leb}$.

The starting point for our investigation is the expansion of $\eps \mapsto H(\pi^{\eps}|r_{\eps})$ given by \cite[Theorem 1.6]{conforti21deriv} about $\eps = 0$
\begin{align}\label{eq:ct-exp}
    H(\pi^{\eps}|r_{\eps}) &= \frac{1}{2\eps}\Was{2}^2(e^{-f},e^{-h})+ \frac{1}{2}\left(\Ent(e^{-f})+\Ent(e^{-h})\right)+\frac{\eps}{8}\int_0^{1} I(\rho_t^{0})dt + o(\eps).  
\end{align}
Here, $(\rho_t^0, t \in [0,1])$ denotes the \textbf{McCann interpolation}, defined later in \eqref{eq:mccann-interp}, and $I(\rho) = \Exp{\rho}\norm{\nabla \log \rho}^2$ is the Fisher information. It holds that $H(P^{\eps}|R^{\eps}) = H(\pi^{\eps}|r_{\eps})$, so the same expansion in \eqref{eq:ct-exp} holds in the dynamic setting. 

Our approach towards establishing the main results of this paper is the following ideas.

\textbf{Idea 1:} Guess a family of couplings $(q_{\eps},\eps > 0) \subset \Pi(e^{-f},e^{-h})$ such that $\eps \mapsto H(q_{\eps}|r_{\eps})$ has the same expansion up to $o(\eps)$ as \eqref{eq:ct-exp}. That is, $H(q_{\eps}|r_{\eps})$ is also equal to the RHS of \eqref{eq:ct-exp}, for a different $o(\eps)$ term. Then, prove and use a Pythagorean Theorem for relative entropy to claim that $H(q_{\eps}|\pi^{\eps})=o(\eps)$. 
The error must be of order $o(\eps)$ to prove Main Theorem \ref{informal-theorem-1}.

\textbf{Idea 2: How to guess a good candidate coupling?} Let $(X_t, t \geq 0)$ denote the primal MLD defined in \eqref{defn:MLD} with $X_0 \sim e^{-f}$, and recall the notation $x \mapsto x^* = \nabla \varphi(x)$. Set $\ell_{\eps} = \mathrm{Law}(X_0,X_{\eps}^*)$, then $(\ell_{\eps},\eps > 0) \subset \Pi(e^{-f},e^{-h})$. By path continuity $X_\eps \approx X_0$, and, hence, $X_\eps^* \approx X_0^*$, for $\eps \approx 0$. Hence the coupling is close to the Monge coupling. 
We claim that $\ell_{\eps}$ is a good candidate to approximate $\pi^\eps$. Why is that true?

Let's proceed term by term in \eqref{eq:ct-exp}.  In the expansion of $\eps \mapsto H(\ell_{\eps}|r_{\eps})$, we identify
\begin{itemize}
    \item the leading order $\frac{1}{2\eps}\Was{2}^2(e^{-f},e^{-h})$ term from the pushforward on the second coordinate and Varadhan's large deviation result for the transition density of the MLD \cite{varadhan-diff-ldp67};
    \item the $\frac{1}{2}\left(\Ent(e^{-f})+\Ent(e^{-h})\right)$ term follows from the Gaussian approximation developed in \cite{pal2019difference}.
\end{itemize}

Put another way, if we approximate the conditional density of $X_\eps^*$, given $X_0=x$, by using the small time asymptotics \cite{azencott84,benarous-expan-88} of the transition density of the dual MLD, up to Gaussian fluctuations, one recovers the first two terms of the RHS of \eqref{eq:ct-exp}.

However, this static picture gives an error of $O(\eps)$ and is not sufficient to identify the $\frac{\eps}{8}\int_0^{1} I(\rho_t^0)dt$ term in \eqref{eq:ct-exp}. This is reflected in the fact that $H(\ell_{\eps}|\pi^{\eps}) = o(1)$ in \cite[Theorem 4]{MP25}. Also, see the Conjecture in that paper around equation (12). Since we do not hope to get an $o(\eps)$ error with $\ell_\eps$, we must refine our coupling to a higher order.  

For a purely geometric intuition on why $\ell_\eps$ provides a good (but not good enough) approximation to $\pi^\eps$, see the discussion around \cite[eqn. (11)]{MP25}.

 \textbf{Idea 3: Bring in the dynamics.} To match the term $\frac{\eps}{8}\int_0^{1} I(\rho_t^0)dt$ in \eqref{eq:ct-exp} we introduce a novel modification to $\ell_\eps$ using dynamics. 

Form a reciprocal process with endpoints sampled from $\ell_\eps$ and joined by $\eps$-Brownian bridges as defined in \eqref{defn:intro-stochastic-proc}. Modify it via the \textit{tangent} Markov projection described in Section \ref{subsec:markov-proj}. 
Let $\overline{Q}^{\eps} \in \cP(C^{d}[0,1])$ denote the law of the tangent Markov projection, and let $\overline{\ell}^\eps$ denote the joint density of the two endpoints under $\overline{Q}^\eps$. Since the (tangent) Markov projection has the same marginal distributions as the original reciprocal process, $\overline{\ell}^\eps \in \Pi\left( e^{-f}, e^{-h} \right)$, i.e., it is an admissible coupling. Now we show that this coupling is a sharp enough approximation to the $\Schro$ bridge to capture all three terms on the RHS of \eqref{eq:ct-exp} up to an $o(\eps)$ error.

\textbf{Idea 4: A new coupling emerges.} However, this brings us to a new coupling. Let $\nabla \varphi$ be the Brenier map from $e^{-f}$ to $e^{-h}$, then for each $t \in [0,1]$ define the interpolation $\nabla \varphi_{0 \to t}(x) := (1-t)x + t\nabla \varphi(x)$. The McCann interpolation (see Section \ref{subsec:ot}) is the curve of measures obtained by pushforwarding $e^{-f}$ by $\nabla \varphi_{0 \to t}$, that is,
\begin{align}\label{eq:mccann-interp-intro}
    \rho_t^0 = (\nabla \varphi_{0 \to t})_{\#}e^{-f}, \quad t \in [0,1].
\end{align}
Because each $\nabla \varphi_{0 \to t}$ is the gradient of a convex function, each $\nabla \varphi_{0 \to t}$ is itself the Brenier map from $e^{-f}$ to $\rho_t^0$. In general, we use the notation $\nabla \varphi_{s \to t}$ to denote the Brenier map between $\rho_s^0$ and $\rho_t^0$.
Observe that $\nabla \varphi_{0 \to 0} = \Id$ and $\nabla \varphi_{0 \to 1} = \nabla \varphi$, so that $\rho_0^0 = e^{-f}$ and $\rho_1^0 = e^{-h}$.

 The coupling obtained from the McCann interpolation, $(\nabla \varphi_{0\rightarrow t}(X),\; t\in [0,1])$, $X \sim e^{-f}$, is singular and hence inadmissible for the relative entropy expansion \eqref{eq:ct-exp}. However, one can do a tangent Markov projection of this process with an $\eps$-Brownian noise to get an admissible coupling. This gives us the $\eps$-\textbf{noisy McCann interpolation} process that we introduce in this paper. This is a solution to the SDE
\begin{align}\label{eq:noisy-mcann-intro}
    dZ_t = \left(\nabla \varphi_{t \to 1}-\nabla \varphi_{t \to 0}+\frac{\eps}{2}\nabla \log \rho_t^0\right)(Z_t) dt + \sqrt{\eps}dB_t, \quad Z_0 \sim e^{-f}. 
\end{align}
Let the law of this process be $\widetilde{Q}^\eps$. It is obvious that the time marginal laws of $\widetilde{Q}^\eps$ also follow the McCann interpolation, in spite of the presence of the noise. A quick verification (see \eqref{eq:tan-mark-rel-ent}) yields that 
\[
 H(\widetilde{Q}^{\eps}|R^{\eps}) =  \frac{1}{2\eps}\Was{2}^2(e^{-f},e^{-h}) + \frac{1}{2}\left(\Ent(e^{-f})+\Ent(e^{-h})\right)+\frac{\eps}{8}\int_0^{1} I(\rho_t^0)dt,
\]
which matches the expansion in \eqref{eq:ct-exp} \textit{exactly} up to the $o(\eps)$ error.
We show that $H(\widetilde{Q}^\eps|P^\eps)=o(\eps)$, and that, under $\widetilde{Q}^\eps$, 
\[
   \lim\limits_{\eps \downarrow 0}\frac{1}{\eps}\left(\Exp{\widetilde{Q}^\eps}\left[Z_1|Z_0=z\right]-\nabla \varphi(x)\right) = -\frac{1}{2}\nabla f(z). 
\]
A functional inequality transfers the above limit to $\pi^\eps$ and Theorem \ref{thm:sb-o-eps-score} follows.

\textbf{Idea 5: The remarkable geometry of Schr\"{o}dinger bridges.} 


%
%
%
\providecommand{\FRrecip}{(50)}                 
\providecommand{\FRtmp}{(56)}                   
\providecommand{\FRdsb}{(42)}                   
\providecommand{\FRmccann}{(111)}               
\providecommand{\FRsectmp}{3.4}                 
\providecommand{\FRthmpyth}{Thm.~1}             
\providecommand{\FRthmmain}{Thm.~2}             
\providecommand{\FRthmbrenier}{Thm.~4}          
\providecommand{\FRpropmccann}{Prop.~5}         
\providecommand{\FRpropqq}{Prop.~7}             
\providecommand{\FRproptransfer}{Prop.~6}       

\begin{figure}[t]
\centering
\begin{tikzpicture}[
  x=1cm, y=1cm,
  meas/.style={draw, rounded corners=2pt, inner sep=4pt, align=center,
               minimum height=7.5mm, minimum width=8.5mm, fill=black!2},
  sub/.style={font=\scriptsize, align=center, text=black!65},
  est/.style={-{Stealth[length=2mm]}, thick},
  weak/.style={-{Stealth[length=2mm]}, dashed, black!45},
  none/.style={-{Stealth[length=2mm]}, densely dotted, black!45},
  proj/.style={-{Stealth[length=2.4mm]}, line width=1pt, black!55!black},
  push/.style={-{Stealth[length=1.8mm]}, densely dotted, black!55},
  lab/.style={font=\scriptsize, inner sep=1.5pt, fill=white, align=center},
  cite/.style={font=\scriptsize\itshape, inner sep=1.5pt, text=black!70,
               fill=white, align=center}
]

\node[meas] (Q)  at (0,0)     {$\widetilde{Q}^{\varepsilon}$};
\node[meas] (Qb) at (4.3,0)   {$\overline{Q}^{\varepsilon}$};
\node[meas] (P)  at (8.6,0)   {$P^{\varepsilon}$};
\node[meas] (Qt) at (4.3,2.7) {${Q}^{\varepsilon}$};

\node[sub, text width=3.3cm] at (-0.25,1.45)
     {$\varepsilon$-noisy McCann interpolation};
\node[sub, text width=2.7cm] at (8.9,1.45)
     {dynamic Schr\"odinger\\ bridge};
\node[sub] at (4.3,3.8)
     {reciprocal process, with MLD endpoints $\ell_{\varepsilon}$};
\node[sub] at (4.3,3.7) {
     };

\draw[proj] (Qb) -- (Q)
  node[lab, pos=0.5, above=0.7mm, text=black]
      {$H\bigl(\overline{Q}^{\varepsilon}\,\big|\,\widetilde{Q}^{\varepsilon}\bigr)
        =o(\varepsilon)$}
      ;

\draw[est] (Qb) -- (P)
  node[lab, pos=0.53, above=0.7mm]
      {$H\bigl(\overline{Q}^{\varepsilon}\,\big|\,P^{\varepsilon}\bigr)=o(\varepsilon)$}
  node[cite, pos=0.5, below=0.7mm] {};

\draw[est] (Qt) -- (Qb)
  node[lab, pos=0.42, left=1mm]
      {tangent Markov\\ projection}
  node[cite, pos=0.22, right=1mm] {};

\draw[est] (Q) to[bend right=32]
  node[lab, pos=0.55, below=1mm]
{$H\bigl(\widetilde{Q}^{\varepsilon}\,\big|\,P^{\varepsilon}\bigr)=o(\varepsilon)$} (P);


\node[meas] (l)  at (0,-2.9)   {$\widetilde{\ell}_{\varepsilon}$};
\node[meas] (lb) at (4.3,-2.9) {$\bar{\ell}_{\varepsilon}$};
\node[meas] (pi) at (8.6,-2.9) {$\pi^{\varepsilon}$};

\draw[push] (Q)  -- (l);
\draw[push] (Qb) -- (lb)
  node[lab, pos=0.5, right=1mm] {};
\draw[push] (P)  -- (pi);

\draw[est] (lb) -- (l)
  node[lab, pos=0.5, above=0.7mm]
      {$H\bigl(\bar{\ell}_{\varepsilon}\,\big|\,\widetilde{\ell}_\eps\bigr)=o(\varepsilon)$};

\draw[est] (lb) -- (pi)
  node[lab, pos=0.5, above=0.7mm]
      {$H\bigl(\bar{\ell}_{\varepsilon}\,\big|\,\pi^{\varepsilon}\bigr)=o(\varepsilon)$};

\draw[est] (l) to[bend right=32]
  node[lab, pos=0.55, below=1mm]
{$H\bigl(\widetilde{\ell}_{\varepsilon}\,\big|\,\pi^{\varepsilon}\bigr)=o(\varepsilon)$} (pi);


\end{tikzpicture}
\caption{The four path measures of the argument and the relative entropy
comparisons between them; an arrow $A\to B$ is labeled by $H(A|B)$. The dotted arrows represent a pushforward of the path measure by $\omega\mapsto (\omega_0, \omega_1)$.  
}
\label{fig:processes}
\end{figure}
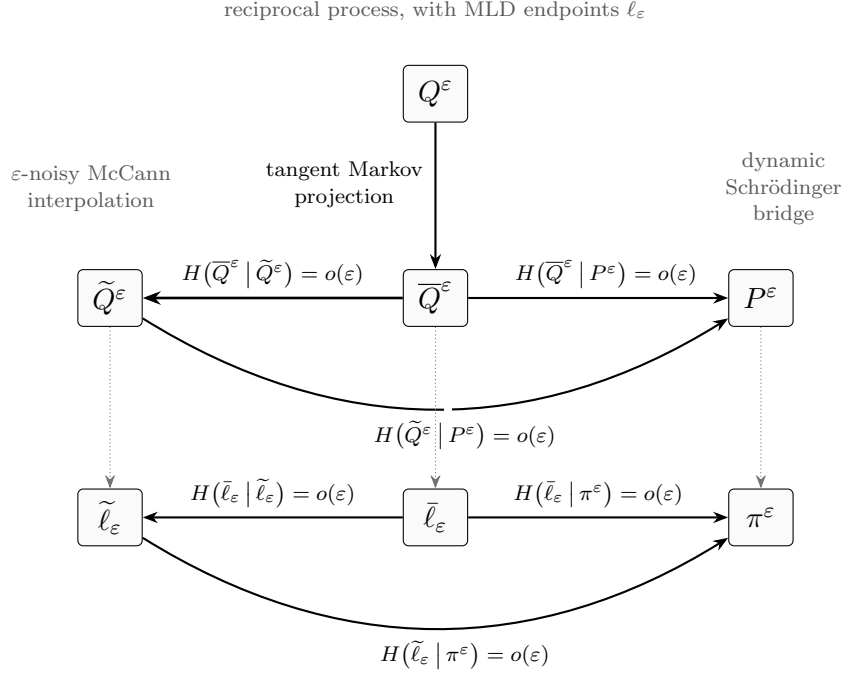


We show in Corollary \ref{prop:mp-drift-approx}, Section \ref{lem:ke-calc}, that $H(\overline{Q}^\eps \mid \widetilde{Q}^\eps)=o(\eps)$. In Theorem \ref{thm:mp-mld-approx-of-sb}, Section \ref{sec:rel-ent-approx}, we show that $H(\overline{Q}^\eps \mid P^\eps)=o(\eps)$. Thus, \textit{both the MLD and the noisy McCann approximate the $\Schro$ bridge}. See Figure \ref{fig:processes} for a diagram showing various approximations. 

Note that, for the noisy McCann, we create the dynamics first, pull out the endpoints $(Z_0, Z_1)$ and show that their law is an approximation for the static $\Schro$ bridge. While for the MLD, we start with the endpoints $(X_0, X_\eps^*)$, join them by Brownian bridge and take a tangent Markov projection to approximate the dynamic $\Schro$ bridge.  Nevertheless, we are now left with an apparent paradox.  \textit{The endpoints of the $\eps$-noisy McCann process $(Z_0, Z_1)$ as well as the MLD coupling $(X_0, X_\eps^*)$ are both approximations to the same static $\Schro$ bridge, but they seem radically different}. For example, the conditional covariance of $Z_1$, given $Z_0=x$, is unclear, whereas it is immediate from the SDE \eqref{defn:dMLD} that the conditional covariance of $X_\eps^*$, given $X_0=x$, is approximately $\eps \nabla^2 \varphi(x)$, which matches with that of the $\Schro$ bridge by \cite{pal2019difference}. More generally, the MLD and the noisy McCann processes have different diffusion matrices. Thus, their small time behaviors  are determined by two different geometries \cite{varadhan-diff-ldp67, azencott84}, Hessian geometry for the former and flat geometry for the latter. If these processes are connected probabilistically, these geometries must be related as well.  The resolution of this apparent paradox is a key insight in the paper carefully explained below.

Our proofs rest on the analysis of a random surface $(U(t,s),\; t \in [0,1],\; s\ge 0)$, called the \textbf{reciprocal surface}. 

\begin{figure}[t!]
    \centering
    \includegraphics[width=1.0\linewidth]{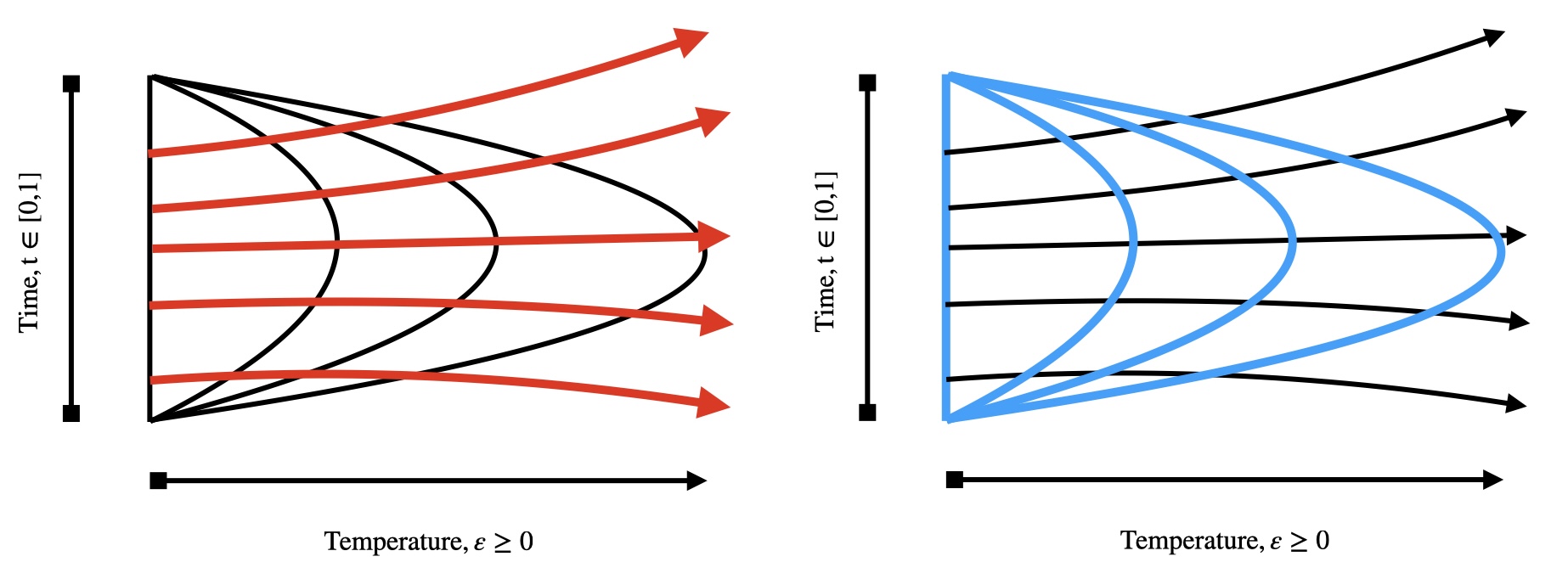}
    \caption{The laws of the reciprocal surface defined in \eqref{eq:defn-ust-intro}. The red curves (with arrows) illustrate $(\rho_t^{s}, s \geq 0)$, and when $s \approx 0$ these curves are approximately the law of the stationary $t$-MLD defined in \eqref{defn:t-mld-interp-defn-intro}. The blue curves (without arrows) illustrate $(\rho_t^{s}, t \in [0,1])$, and when $s \approx 0$ their tangent Markov projections are approximately the $s$-noisy McCann interpolation from \eqref{eq:noisy-mcann-intro}.} 
    \label{fig:sheet-of-curves}
\end{figure}

Let $(X_t, t \geq 0)$ satisfy the primal MLD defined in \eqref{defn:MLD} with initial distribution $X_0 \sim e^{-f}$ and recall the notation $x \mapsto x^* = \nabla \varphi(x)$. Let $(W_s, s \geq 0)$ be a standard Brownian motion independent of the MLD. Define the following family of random variables
\begin{align}\label{eq:defn-ust-intro}
    (t,s) \mapsto U(t,s) := (1-t)X_0 + tX_{s}^* + \sqrt{t(1-t)}W_{s}, \quad \mathrm{Law}(U(t,s))=\rho_t^{s}. 
\end{align}
Thus $U$ is a random surface indexed by $[0,1] \times [0, \infty)$. The variable $s \in [0,+\infty)$ plays the role of the {temperature} parameter $\eps$, and the variable $t \in [0,1]$ plays the role of {time}.  
See Figure \ref{fig:sheet-of-curves}. Note that, when $s=\eps$, the process $(U(t, \eps), \; t\in [0,1])$ has the same marginal laws as the reciprocal process from \eqref{defn:intro-stochastic-proc}. The results of our paper focus on understanding how this family of random variables and their corresponding laws behave in the low temperature regime, i.e., when $s\approx 0$.  We now give a high level description of the phenomena we identify. 

\textbf{Varying low temperature ($s \approx 0$, $t \in [0,1]$ fixed).}
Fix $t \in [0,1]$. We wish to understand the behavior of the family $(U(t,s), s \geq 0)$ defined in \eqref{eq:defn-ust-intro} when $s \approx 0$. This is the left panel in Figure \ref{fig:sheet-of-curves}. Note that setting $s = 0$ in \eqref{eq:defn-ust-intro} gives that $U(t,0) = \nabla \varphi_{0 \to t}(X_0) \sim \rho_{t}^{0}$, i.e.\ the time $t$ marginal on the McCann interpolation.

Recall the SDE definition of the MLD in \eqref{defn:MLD}. Take the same primal MLD $(X_s, s \geq 0)$, but now take the mirror map to be $\nabla \varphi_{0 \to t}$. That is, define $Z_s = \nabla \varphi_{0 \to t}(X_s)$. $(Z_s, s \geq 0)$ is distributed as the dual MLD with stationary distribution $\rho_t^0 = (\nabla \varphi_{0 \to t})_{\#}e^{-f}$. It satisfies the SDE
\begin{align}\label{defn:t-mld-interp-defn-intro}
    dZ_s = -\frac{1}{2}\nabla f(X_s) ds + \sqrt{\nabla ^2\varphi_{0\to t}(X_s)}dB_s, \quad Z_0 \sim \rho_t^0. 
\end{align}
Call this diffusion the \textbf{dual $t$-MLD}. We show that when $s \approx 0$, the joint law of $(U(t,s), U(t,0))$ is approximately the same as that of $(Z_{ts}, Z_0)$ sampled from \eqref{defn:t-mld-interp-defn-intro}. 
See the proof of Lemma \ref{lem:zero-initial-velocity}.

\textbf{Varying time ($s \approx 0$ fixed, $t \in [0,1]$).}
Fix $s \approx 0$ and consider the process $(U(t,s), t \in [0,1])$. This is the right panel in Figure \ref{fig:sheet-of-curves}. Observe, again, that at $s = 0$, the laws of the random variables $(U(t,0), t \in[0,1])$ follow exactly the McCann interpolation. As we make precise in Section \ref{subsec:ot}, the McCann interpolation satisfies the continuity equation with time dependent velocity field $\left(\nabla \varphi_{t \to 1}-\nabla \varphi_{t \to 0}, t \in [0,1]\right)$. We show that when $s \approx 0$, the tangent Markov projection of the stochastic process $(U(t,s), t \in [0,1])$ behaves approximately like the $s$-noisy McCann interpolation. 

These observations, coupled with the matrix identity in Proposition \ref{prop:sqrd-inv-hessian-iden}, connect the MLD with the noisy McCann both probabilistically as well as geometrically. 

\section{Preliminaries}\label{sec:preliminaries}
We now introduce the necessary notation and preliminary results to state and prove our main results. 

\subsection{Optimal Transport and the Wasserstein Space}\label{subsec:ot}
Equipping $\cP_{2}(\mathbb{R}^{d})$ with $\Was{2}$ results in a complete separable metric space called the \textbf{Wasserstein space}. The following notions we introduce may be found in \cite{ambrosio2005gradient}.  
We say a curve of probability measures $(\mu_t, t \in [0,T]) \subset \cP_2(\mathbb{R}^{d})$ is \textbf{absolutely continuous}, abbreviated AC, if there exists a collection of Borel velocity fields, $v_t : \mathbb{R}^{d} \to \mathbb{R}^{d}$, $t \in [0,T]$ such that $\int_0^{T} \norm{v_t}_{L^2(\mu_t)}^2 dt < +\infty$ and the continuity equation, 
\begin{align}\label{defn:continuity-eqn}
    \partial_t \mu_t + \nabla \cdot (v_t \mu_t) &= 0,
\end{align}
is satisfied in the sense of distributions. That is, for all functions $\psi \in C_c^{\infty}([0,T] \times \mathbb{R}^{d})$,
\begin{align}\label{defn:weak-soln}
    \int_{\mathbb{R}^{d}} \psi(T,x)\mu_T(dx) - \int_{\mathbb{R}^{d}} \psi(0,x)\mu_0(dx) = \int_0^{T} \int_{\mathbb{R}^{d}}\left( \partial_t \psi(t,x)+ \langle \nabla \psi(t,x),v_t(x)\rangle \right)\mu_{t}(dx) dt. 
\end{align}
This notion of AC curves is equivalent to the general metric space notion \cite[Theorem 8.3.1]{ambrosio2005gradient}. For an AC curve $(\mu_t, t \in [0,T])$, the \textbf{metric slope} is defined as the following limit that exists for $\mathrm{Leb}$-a.e.\ $t \in [0,T]$, see \cite[Theorem 1.1.2]{ambrosio2005gradient},
\begin{align}\label{defn:metric-slope}
    |(\mu_t)'| := \lim\limits_{h \to 0} \frac{1}{\abs{h}}\Was{2}(\mu_{t+h},\mu_t). 
\end{align}
Given an AC curve $(\mu_t, t \in [0,T])$, there is an infinite family of vector fields $(v_t, t \in [0,T])$ such that \eqref{defn:continuity-eqn} holds. We will be interested in the \textbf{tangent velocity field} obtained by choosing each $v_t$ to have minimum $L^2(\mu_t)$ norm across its equivalence class. Define
\begin{align}\label{defn:tangent-space}
    \Tan{\mu} :=  \overline{\left\{\nabla \varphi: \varphi \in C_c^{\infty}(\mathbb{R}^{d})\right\}}^{L^2(\mu)}.
\end{align}
By \cite[Theorem 8.3.1]{ambrosio2005gradient}, such a $\overline{v}_t$ exists for $\mathrm{Leb}$-a.e.\ $t \in [0,T]$ and satisfies
\begin{align}\label{defn:tangent-velocity}
    \norm{\overline{v}_t}_{L^2(\mu_t)} = |(\mu_t)'|, \quad \overline{v}_t \in \Tan{\mu_t}.
\end{align}
As a consequence, we note that all velocity fields that provide a solution to \eqref{defn:continuity-eqn} can be obtained from the tangent velocities by the addition of a \textbf{divergence-free vector field}, i.e.\ a vector field $w_t \in L^2(\mu_t)$ such that $\nabla \cdot (w_t \mu_t) = 0$, see \cite[Proposition 8.4.3]{ambrosio2005gradient}. As we see in Section \ref{subsec:markov-proj}, this fact is the key distinction between the Markov projection and the tangent Markov projection we introduce. 

Geodesics in the Wasserstein space will also be essential in our analysis. Let $\mu, \nu \in \cP_2(\mathbb{R}^{d})$. There is a corresponding dynamic picture to the definition of $\Was{2}$ in \eqref{defn:was2}, developed in \cite{benamou2000computational}, that casts the Monge Kantorovich problem as a constrained kinetic energy minimization problem: 
\begin{align}\label{eq:bb}
    \Was{2}^2(\mu,\nu) &= \inf\limits_{(\rho_t,v_t, t \in [0,1])}\left\{\int_0^{1}\norm{v_t}^2_{L^2(\rho_t)} dt: \partial_t \rho_t + \nabla \cdot (v_t \rho_t) = 0, \rho_0 = \mu,\ \rho_1 = \nu\right\}. 
\end{align}
The minimizer of \eqref{eq:bb} is the unit time geodesic from $\mu$ to $\nu$. In the case that $\mu$ has a Lebesgue density, the minimizer of \eqref{eq:bb} is given in terms of the Brenier map. Let $\nabla \varphi$ denote the Brenier map, and define for $t \in [0,1]$ the map $\nabla \varphi_{0 \to t} = (1-t) \Id + t\nabla \varphi$. Note that $\nabla \varphi_{0 \to t}$ is also the gradient of a convex function. The curve of measures
\begin{align}\label{eq:mccann-interp}
    (\rho_t, t \in [0,1]), \quad \rho_t := (\nabla \varphi_{0 \to t})_{\#}\mu \text{ with } \nabla \varphi_{0 \to t} = (1-t) \Id + t \nabla \varphi
\end{align}
is called the \textbf{McCann interpolation} \cite{mccann-interp97}. We will use extensively the following notation for transport along the McCann interpolation:
\begin{align}\label{eq:mccann-interp-notation}
    \nabla \varphi_{s \to t} \text{ is the Brenier map from $\rho_s$ to $\rho_t$.}
\end{align}
Indeed, when the McCann interpolation exists we can explicitly compute its corresponding tangent velocities via the Hamilton-Jacobi equation. Define $\psi_{0} = \varphi-\frac{1}{2}\norm{x}^2$. In the optimal transport literature, $(-\psi_0)$ is called a Kantorovich potential. For $t \in (0,1]$ define
\begin{align*}
    \psi_{t}(y) := \inf\limits_{x \in \mathbb{R}^{d}}\limits\left\{\frac{1}{2t}\norm{x-y}^2 + \psi_{0}(x)\right\}.
\end{align*}
Under sufficient regularity, the function $(\psi_t, t \in [0,1])$ is a classical solution to the following Hamilton-Jacobi PDE
\begin{align}\label{eq:hj-pde}
    \partial_{t}\psi_t + \frac{1}{2}\norm{\nabla \psi_{t}}^2 &= 0. 
\end{align}
Then it holds that
\begin{align}\label{eq:cont-eqn-mccann}
    \partial_t \rho_t + \nabla \cdot(\overline{v}_t \rho_t) &= 0 \text{ with }\overline{v}_t = \nabla \psi_t = \nabla \varphi_{t \to 1}-\nabla \varphi_{t \to 0}. 
\end{align}
In particular, from the definition of a weak solution to the continuity equation we obtain the following key identity
\begin{align}\label{eq:hjb-ce-identity}
    \frac{1}{2}\Was{2}^2(\mu,\nu) &= \int_{\mathbb{R}^{d}} \psi_1 d\rho_1 - \int_{\mathbb{R}^{d}} \psi_0 d\rho_0 = \int_0^1 \int_{\mathbb{R}^{d}} \left(\partial_t \psi_t + \langle \nabla \psi_t , \overline{v}_t \rangle \right)d\rho_t dt = \int_0^{1} \frac{1}{2}\norm{\overline{v}_t}_{L^2(\rho_t)}^2 dt,
\end{align}
where the last step follows from velocity field in \eqref{eq:cont-eqn-mccann} and the identity furnished by \eqref{eq:hj-pde}. 

We make the following note about the correspondence between solutions to the continuity equation and another PDE called the Fokker-Planck PDE. Let $(\mu_t, v_t, t \in [0,1])$ satisfy \eqref{defn:continuity-eqn}. Under enough regularity, we remark that for any fixed $\eps >0$, one can add and subtract $\frac{\eps}{2}\nabla \log \mu_t$ to $v_t$ and obtain that
\begin{align}\label{eq:cont-fp-corres}
    \partial_{t} \mu_t + \nabla \cdot(v_t\mu_t) = 0 \Leftrightarrow \partial_{t} \mu_t + \nabla \cdot \left(\left(v_t + \frac{\eps}{2}\nabla \log \mu_t\right)\mu_t\right) = \frac{\eps}{2}\Delta \mu_t. 
\end{align}
The PDE on the right hand side of \eqref{eq:cont-fp-corres} is called the Fokker-Planck PDE, and its solution gives the time marginals of the following SDE
\begin{align}\label{eq:sde-fp-correspondence}
    dX_t = \left(v_t + \frac{\eps}{2}\nabla \log \mu_t\right)(X_t)dt + \sqrt{\eps} dB_t, \quad X_0 \sim \mu_0.
\end{align}
In particular, we now see that the $\eps$-noisy McCann interpolation as defined in \eqref{eq:noisy-mcann-intro} in the Introduction has time marginal laws equal to the McCann interpolation.

Lastly, as the relative entropy between path measures will be an essential computation throughout this paper we record the following formula.

\begin{fact}[Girsanov under Finite Entropy (Theorem 2.3, \cite{leo-gis12})]\label{fact:girs-formula}
    For $i = 1,2$, let $P^{i} \in \cP(C^{d}[0,1])$ denote the laws of solutions to the following SDEs
    \begin{align}
        dX_t = \gamma_t^i(X_{[0,t]}) dt + \sqrt{\eps}dB_t, \quad X_0 \sim \mu_0^{(i)},
    \end{align}
    where $X_{[0,t]}$ means that $\gamma_t^i$ can depend on the whole path over the interval $[0,t]$. If $P^{1}$ has uniqueness in law as defined in \cite[Definition 1.3]{revuz2004continuous}, then when $H(P^2|P^1) < +\infty$ the following formula holds
    \begin{align*}
        H(P^2|P^1) = H(\mu_0^{(2)}|\mu_0^{(1)}) + \frac{1}{2\eps}\int_0^{1} \Exp{P^2}\left[\norm{\gamma_t^1(X_{[0,t]})-\gamma_t^2(X_{[0,t]})}^2 \right]dt.
    \end{align*}
\end{fact}

\subsection{Weighted Sobolev Spaces}
In our proofs, we will make use of weighted Sobolev spaces. We quickly recall the basic definitions here. Fix $\mu \in \cP(\mathbb{R}^{d})$. Let $f \in H^{1}(\mu)$, i.e.\ $f$ and its weak gradient are both in $L^2(\mu)$. We define the $\dot{H}^{1}(\mu)$ norm as
\begin{align}\label{defn:neg-sob-norm}
    \norm{f}_{\dot{H}^{1}(\mu)}^{2} := \int_{\mathbb{R}^{d}} \norm{\nabla f}^2 d\mu. 
\end{align}
The dual norm $\dot{H}^{-1}(\mu)$ is defined for a distribution $\nu$ as follows, where $\langle \cdot,\cdot\rangle$ denotes the usual dual pairing between distributions and $C_c^{\infty}(\mathbb{R}^{d})$,
\begin{align}\label{defn:neg-sob-norm}
    \norm{\nu}_{\dot{H}^{-1}(\mu)} &= \sup \left\{\abs{\langle f,\nu\rangle}: f \in H^{1}(\mu) \text{, } \norm{f}_{\dot{H}^{1}(\mu)} = 1\right\}.
\end{align}
For an AC curve $(\mu_t, t \in [0,1]) \subset \cP_2(\mathbb{R}^{d})$, we can write the metric slope defined in \eqref{defn:metric-slope} in terms of this negative Sobolev norm, interpreting $\partial_t \mu_t$ as a distribution,
\begin{align*}
    \int_{0}^{1} |(\mu_{t})'|^2 dt &= \int_0^{1} \norm{\partial_{t} \mu_{t}}^2_{\dot{H}^{-1}(\mu_t)}dt. 
\end{align*}

Lastly, recall that a distribution $\sigma$ satisfies $\sigma + \nabla \cdot (v \mu) = 0$ for some Borel $v: \mathbb{R}^{d} \to \mathbb{R}^{d}$ means that,
\begin{align}\label{defn:div-weak-sense}
    \langle \sigma, \psi \rangle &:= \int_{\mathbb{R}^{d}} \langle v, \nabla \psi \rangle d\mu, \quad \text{for all $\psi \in \dot{H}^{1}(\mu)$}.
\end{align}

\subsection{$\Schro$ Bridges and Reciprocal Processes}\label{subsec:sb-rp}
For a detailed survey on $\Schro$ bridges, we refer readers to \cite{schroLeonard13}. Recall that $(r_t(\cdot,\cdot), t > 0)$ denotes the standard Euclidean heat kernel. The $\Schro$ bridge $\pi^{\eps}$ between two measures $\mu,\nu \in \cP(\mathbb{R}^{d})$ defined in \eqref{defn:intro-static-sb} is equivalently the solution to the following relative entropy minimization problem
\begin{align}\label{defn:eot}
    \pi^{\eps} := \argmin \limits_{\pi \in \Pi(\mu,\nu)} H(\pi|r_{\eps}).
\end{align}
Let $\eps > 0$ and let $R^{\eps}$ denote the law of reversible Wiener measure on $C^{d}[0,1]$, i.e.\ the law of $(B_{\eps t}, t \in [0,1])$ with $B_0 \sim \mathrm{Leb}$. The $\eps$-dynamic $\Schro$ bridge from $\mu$ to $\nu$ is defined as
\begin{align}\label{eq:dynam-schro-bridge-defn}
    P^{\eps} &:= \argmin \left\{H(P|R^{\eps}): P \in \cP(C^{d}[0,1]) \text{ such that } (\omega_0)_{\#}P=\mu, (\omega_1)_{\#}P = \nu\right\}.
\end{align}
Additionally, with $b^{\eps}$ from \eqref{eq:fg-decomp}, $P^{\eps}$ is the law of $(X_t, t \in [0,1])$ satisfying the following SDE
\begin{align}\label{eq:sb-sde}
    dX_t &= \eps \nabla \log R^{\eps}_{1-t}b^{\eps}(X_t) dt + \sqrt{\eps}dB_t, \quad X_0 \sim \mu. 
\end{align} 
Let $(R^{\eps}_{xy}, x,y \in \mathbb{R}^{d})$ denote the laws of the bridges of $R^{\eps}$ on $C^{d}([0,1])$. In other words, this is the disintegration of $R^{\eps}$ with respect to the events $\{\omega_0 = x,\omega_1 = y\}$. Equivalently, each $R_{xy}^{\eps}$ is the law of the Brownian bridge from $x$ to $y$ with diffusion matrix $\sqrt{\eps}\Id$. For shorthand, we will refer to this collection as \textbf{$\eps$-Brownian bridges}. The static and dynamic $\Schro$ bridges are related by
\begin{align}
    P^{\eps} = \int_{\mathbb{R}^{d} \times \mathbb{R}^{d}} R^{\eps}_{xy}\pi^{\eps}(dxdy). 
\end{align}
Analogous to the static $\Schro$ bridge, in the notation of \eqref{eq:fg-decomp} it holds that $dP^{\eps}/dR^{\eps} = a^{\eps}(\omega_0)b^{\eps}(\omega_1)$. 

We will be interested in other reciprocal processes throughout this paper. In particular, we will closely consider a path measure $Q^{\eps}$ obtained by mixing a joint density $q_{\eps} \in \cP(\mathbb{R}^{d}\times \mathbb{R}^{d})$ with the bridges of $R^{\eps}$ in the following manner
\begin{align}\label{eq:intro-rp}
    \frac{dQ^{\eps}}{dR^{\eps}} = \frac{q_{\eps}(\omega_0,\omega_1)}{r_{\eps}(\omega_0,\omega_1)} \text{ or equivalently, } Q^{\eps} = \int_{\mathbb{R}^{d} \times \mathbb{R}^{d}} R^{\eps}_{xy}q_{\eps}(dxdy). 
\end{align}
While we use measure theoretic notation in \eqref{eq:intro-rp}, we note that this path measure is the law of the stochastic process written in \eqref{defn:intro-stochastic-proc} in the Introduction.  

Reciprocal measures, introduced in \cite{jamison-rp-74}, are well-studied objects. For a modern survey and proper definition, see \cite{leo-reciprocal14}. They have been used and studied in the context of entropy minimization \cite{follmergantert97} as well as diffusions and their bridge processes \cite{vonrenesse-conf18,conforti-recip-characteristics18}. Recently in the generative modeling literature, they have gained attention as the first ``bridge fitting'' step in the Iterative Markov Fitting (IMF) procedure for computing $\Schro$ bridges proposed by \cite{shi2023diffusion,peluchetti-23}. Importantly, $P^{\eps}$ is distinguished among all other reciprocal measures on $C^{d}[0,1]$ from $\mu$ to $\nu$ with bridges $(R^{\eps}_{xy},x,y \in \mathbb{R}^{d})$ in that it is also a Markov measure. This fact is due to the product structure in \eqref{eq:fg-decomp}, see \cite[Section 3]{leo-reciprocal14}. 

Lastly, we consider the absolutely continuous curve in $P_2(\mathbb{R}^{d})$ given by the marginal flow of the dynamic $\Schro$ bridge, $((\omega_t)_{\#}P^{\eps}, t \in [0,1]) \subset \cP_2(\mathbb{R}^{d})$, called the $\eps$-\textbf{entropic interpolation} from $\mu$ to $\nu$. This curve is a solution to the following entropic analogue of \eqref{eq:bb}. We first define the following action functional on AC curves $(\mu_t, v_t, t \in [0,1])$
\begin{align}\label{eq:action-functional}
    A_{\eps}((\mu_t,v_t)) = \frac{1}{2}\int_0^{1} \norm{v_t}_{L^{2}(\mu_t)}^2dt + \frac{\eps^2}{8}\int_0^{1} I(\mu_t) dt. 
\end{align}
Recall that the Fisher information and relative Fisher information are defined respectively as
\begin{align}\label{defn:fi}
    I(\mu) := 
    \begin{cases}
        \Exp{\mu} \norm{\nabla \log \mu}^2, & \text{ if $\mu \ll \mathrm{Leb}$} \\
        +\infty, &\text{otherwise}
    \end{cases}, \quad
    I(\mu|\nu) := 
    \begin{cases}
        \Exp{\mu} \norm{\nabla \log \left(d\mu/d\nu\right)}^2, & \text{ if $\mu \ll \nu$} \\
        +\infty, &\text{otherwise}
    \end{cases}
\end{align}
\begin{fact}[Entropic Benamou-Brenier, Corollary 5.8 in \cite{gentil2017analogy}]\label{prop:entropic-bb}
    Under mild assumptions, the $\eps$-entropic interpolation from $\mu$ to $\nu$ minimizes $A_{\eps}$, defined in \eqref{eq:action-functional}, over a class of absolutely continuous curves from $\mu$ to $\nu$. Let $(\mu_t^{\eps},v_t^{\eps}, t \in [0,1])$ denote the entropic interpolation with $v_t^{\eps} \in \Tan{\mu_t^{\eps}}$, then the entropic cost can be written as
    \begin{align}\label{eq:ent-cost-ent-interp}
        \eps H(\pi^{\eps}|r_{\eps}) &= \frac{\eps}{2}(\Ent(\mu)+\Ent(\nu))+\frac{1}{2}\int_0^{1} \norm{v_t^{\eps}}_{L^{2}(\mu_t^{\eps})}^2dt + \frac{\eps^2}{8}\int_0^{1} I(\mu_t^{\eps})dt. 
    \end{align}
\end{fact}

\subsection{(Tangent) Markov Projection}\label{subsec:markov-proj}
In this subsection, we introduce three stochastic processes, all possessing the same one-dimensional flow of marginals. We keep the exposition in this subsection \textbf{general}, as these constructions apply to more general settings than the one considered later in this paper.

Let $\ell_{\eps} \in \cP(\mathbb{R}^{d} \times \mathbb{R}^{d})$ denote a joint density satisfying some assumptions we delineate shortly. Now, let $Q^{\eps} \in \cP(C^{d}[0,1])$ denote the reciprocal measure formed by mixing the endpoint distribution $\ell_{\eps}$ with $\eps$-Brownian bridges,
\begin{align}\label{eq:recip-measure-varphi}
    \frac{dQ^{\eps}}{dR^{\eps}} &= \frac{\ell_{\eps}(\omega_0,\omega_1)}{r_{\eps}(\omega_0,\omega_1)} =:\varphi_{0\eps}(\omega_0,\omega_1).
\end{align}
Equivalently, $Q^{\eps}$ is the law of a stochastic process $(X_t, t \in [0,1])$ as in \eqref{defn:intro-stochastic-proc}. It is known that the process satisfies the following SDE
\begin{align}\label{eq:recip-process-sde}
    dX_t &= \beta_t^{\eps}(X_0,X_t) dt + \sqrt{\eps}dB_t, \quad X_0 \sim e^{-f},
\end{align}
where the drift is given by \cite[Theorem 5.4]{leonard2011stochastic} as, with $\varphi_{0\eps}$ defined in \eqref{eq:recip-measure-varphi},
\begin{align}\label{eq:recip-process-drift}
    \beta_t^{\eps}(x_0,x_t) &= \eps \nabla \log R^{\eps}_{1-t} \left(\varphi_{0\eps}(x_0,\cdot)\right)(x_t). 
\end{align}

Now fix some notation. Let $(\mu_t, \overline{v}_t, t \in [0,1])$ be the AC curve of measures governing the marginal flow of $Q^{\eps}$, i.e.\ $\mu_t = (\omega_t)_{\#}Q^{\eps}$, $\overline{v}_t \in \mathrm{Tan}(\mu_t)$, and the continuity equation is satisfied. The reciprocal process is not Markov, as can be seen from the dependence on $X_0$ at all times $t \in (0,1]$ in the drift \eqref{eq:recip-process-drift}. However, we can perform the \textbf{Markov projection} and obtain a Markov process that has the same one-dimensional flow \cite{gyongy-mimicking-86,brunick-mimicking13}. Define 
\begin{align}\label{eq:mp-drift-defn}
    \hat{\beta}_t^{\eps}(x) &= \Exp{Q^{\eps}}\left[\beta_t^{\eps}(X_0,X_t)|X_t = x\right].
\end{align}
Let $\hat{Q}^{\eps}$ denote the law of a process $(X_t, t \in [0,1])$ satisfying
\begin{align}\label{eq:mp-sde}
    dX_t &= \hat{\beta}^{\eps}_t(X_t)dt + \sqrt{\eps}dB_t, \quad X_0 \sim e^{-f}.
\end{align}
Then $\hat{Q}^{\eps}$ is called the Markov projection of $Q^{\eps}$.
We emphasize again that an important fact of this construction is that $(\mu_t, \overline{v}_t, t \in [0,1])$ still governs the marginal flow of $\hat{Q}^{\eps}$.

The Markov projection has seen a recent surge in interest in the machine learning literature. In theoretical analyses of flow-based generative models, the Markov projection has received attention as one is typically only interested in the one-dimensional time marginal evolution  from a source measure to a target \cite{gentiloni2024theoretical}. Moreover, in the aforementioned IMF scheme, the Markov projection is the key second step of the scheme. The Markov projection satisfies a certain Pythagorean Theorem of relative entropy \cite[Lemma 6]{shi2023diffusion} similar to that of the $\Schro$ bridge \cite{csiszar75idiv}, and more quantitative analysis of the Markov projection via functional inequalities is given in \cite[Theorem 4]{SCD-dsb-25}.

Lastly, we define the \textbf{tangent Markov projection}. This process is obtained by projecting the drift in \eqref{eq:recip-process-drift} onto $\Tan{\mu_t}$, the closed subspace of $L^2(\mu_t)$ defined in \eqref{defn:tangent-space}. For each $t \in [0,1]$ define $\overline{\beta}_{t}^{\eps}$ as 
\begin{align}\label{eq:relaxed-l2-proj}
    \overline{\beta}_t^{\eps} &:= \argmin\limits_{\beta_t \in \Tan{\mu_t}} \Exp{Q^{\eps}}\norm{\beta_{t}^{\eps}(X_0,X_t)-\beta_t(X_t)}^2
    = \argmin\limits_{\beta_t \in \Tan{\mu_t}} \norm{\hat{\beta}_t^{\eps}-\beta_t}_{L^2(\mu_t)}^2.
\end{align}
We note that the two definitions of  $\overline{\beta}_t^{\eps}$ are equivalent by a composition of projections. Indeed, let $\gamma \in L^2(\mu_t)$, then by the definition of conditional expectation it holds that
\begin{align*}
    \Exp{Q^{\eps}}\norm{\beta_t^{\eps}(X_0,X_t)-\gamma(X_t)}^2 &= \Exp{Q^{\eps}}\norm{\beta_t^{\eps}(X_0,X_t)-\hat{\beta}_t^{\eps}(X_t)}^2+\Exp{Q^{\eps}}\norm{\hat{\beta}_t^{\eps}(X_t)-\gamma(X_t)}^2.
\end{align*}
Observe that the first term on the right hand side has no dependence on $\gamma$, and recall that under $Q^{\eps}$, $X_{t} \sim \mu_t$. Thus, the argmin of the left and right hand sides over $\gamma \in \Tan{\mu_t}$ must be the same.

Let $\overline{Q}^{\eps}$ denote the law of a process $(X_t, t \in [0,1])$ that satisfies 
\begin{align}\label{eq:tan-mp-sde}
    dX_t &= \overline{\beta}_t^{\eps}(X_t)dt + \sqrt{\eps}dB_t, \quad X_0 \sim e^{-f}.
\end{align}
As $\overline{\beta}_t^{\eps}$ is obtained from $\hat{\beta}_t^{\eps}$ by subtracting a zero-divergence (in the $L^2(\mu_t)$ sense) vector field, the marginal flow of $\overline{Q}^{\eps}$ is still given by $(\mu_t,\overline{v}_t, t \in [0,1])$. The motivation for this projection is to put the stochastic definition \eqref{eq:tan-mp-sde} in alignment with the metric picture-- a point elucidated in Proposition \ref{prop:relaxed-drift-calculation}.

\begin{remark}
    Note that the tangent Markov projection defined above in \eqref{eq:tan-mp-sde} makes an implicit choice of $\eps$ and Brownian noise. In fact, we can generalize the construction to any AC curve of probabilities. Given an AC curve of probability measures and its corresponding tangent velocities, by changing $\eps$ and the diffusion matrix in \eqref{eq:tan-mp-sde}, one obtains a collection of Markov processes whose time marginals are all identical. For example, if we apply this operation to the McCann interpolation we recover the $\eps$-noisy McCann interpolation process \eqref{eq:noisy-mcann-intro}.
\end{remark}    

We summarize conditions for the existence of these processes in the following proposition.
\begin{proposition}[Existence of Tangent Markov Projection]\label{prop:proj-processes-exist}
Let $Q^{\eps}, \hat{Q}^{\eps}, \overline{Q}^{\eps}$ denote the reciprocal process, its Markov projection, and its tangent Markov projection, respectively, defined in \eqref{eq:recip-process-sde}, \eqref{eq:mp-sde}, and \eqref{eq:tan-mp-sde}. Recall $\left(\beta_{t}^{\eps}(\cdot,\cdot), t \in [0,1]\right)$ defined in \eqref{eq:recip-process-drift}. Assume that $\Ent(\mu_0)$ is finite. When $H(\ell_{\eps}|r_{\eps})$ is finite, the SDEs defining $\hat{Q}^{\eps}$ and $\overline{Q}^{\eps}$ have weak solutions. Additionally, the relative entropies with respect to the Wiener measure $R^{\eps}$ are given by the formulas 
    \begin{align}\label{eq:rel-ent-formulas}
        H(\overline{Q}^{\eps}|R^{\eps}) &= \Ent(\mu_0)+ \frac{1}{2\eps}\int_{0}^{1}\norm{\overline{\beta}_t^{\eps}}_{L^2(\mu_t)}^2 dt, \; H(\hat{Q}^{\eps}|R^{\eps}) = \Ent(\mu_0)+\frac{1}{2\eps}\int_{0}^{1}  \norm{\hat{\beta}_t^{\eps}}_{L^2(\mu_t)}^2 dt.
    \end{align}
\end{proposition}

\begin{proof}
The existence of the Markov projection $\hat{Q}^{\eps}$ is established in \cite[Appendix A, Theorem 4]{gentiloni2024theoretical}. As a similar argument establishes the existence of $\overline{Q}^{\eps}$, we recount the argument for both processes. Note that the existence of the Markov projection is considered in much greater generality in \cite{brunick-mimicking13}. 

We first make the following observation. For a path measure $P \in \cP(C^{d}[0,1])$ and $x \in \mathbb{R}^{d}$, let $P_x$ denote the path measure obtained by conditioning on the event $\{\omega_0 = x\}$. By the factorization of relative entropy and the stopping time argument of \cite[Theorem 5.6, Step 1]{DGW}, we compute that
\begin{align}
    H(Q^{\eps}|R^{\eps}) &= \Ent(\mu_0) + \Exp{\mu_0} \left[H(Q^{\eps}_{x}|R^{\eps}_{x})\right] = \Ent(\mu_0)+\frac{1}{2\eps}\int_0^{1} \Exp{Q^{\eps}}\norm{\beta_t^{\eps}(X_0,X_{t})}^2 dt.
\end{align}
As $Q^{\eps}$ and $R^{\eps}$ have the same bridge distribution, it follows again from the factorization of relative entropy that $H(Q^{\eps}|R^{\eps}) = H(\ell_{\eps}|r_{\eps})$. As $\Ent(\mu_0)$ is finite, it follows that
\begin{align}\label{eq:l2-dom-recip-drift}
    H(\ell_{\eps}|r_{\eps}) \text{ is finite} \Leftrightarrow \int_0^{1} \Exp{Q^{\eps}}\norm{\beta_t^{\eps}(X_0,X_{t})}^2 dt< +\infty.
\end{align}
We now proceed with the argument. First, from the definition \eqref{eq:tan-mp-sde} and applying Jensen's inequality to \eqref{eq:mp-drift-defn}, we obtain the following ordering
\begin{align}\label{eq:l2-ordering-of-drifts}
    \int_0^{1} \norm{\overline{\beta}_t^{\eps}}_{L^2(\mu_t)}^2 dt \leq \int_0^{1} \norm{\hat{\beta}_t^{\eps}}_{L^2(\mu_t)}^2 dt \leq \int_0^{1} \Exp{Q^{\eps}}\norm{\beta_t^{\eps}(X_0,X_{t})}^2 dt.
\end{align}
From \eqref{eq:l2-dom-recip-drift}, there is a unique weak solution to the Fokker-Planck equation corresponding to the generators of $\overline{Q}^{\eps}$ and $\hat{Q}^{\eps}$ by \cite[Theorem 9.4.3, Remark 9.4.4]{fpk-bogachev}. The superposition principle implies the existence of path measures solving the corresponding martingale problems \cite[Theorem 2.5]{trevisan16}. 

Now, it remains to justify the relative entropy expression. Indeed, \eqref{eq:l2-ordering-of-drifts} and \eqref{eq:l2-dom-recip-drift} imply by the Kadota-Shepp condition \cite[Theorem 1']{kadota70} that $\hat{Q}^{\eps}$ and $\overline{Q}^{\eps}$ are absolutely continuous with respect to $R^{\eps}$. From the domination \eqref{eq:l2-ordering-of-drifts}, by the stopping time argument in \cite[Theorem 5.6, Step 1]{DGW} and the lsc of relative entropy with respect to weak convergence \cite[Lemma 9.4.3]{ambrosio2005gradient}, the quantities $H(\hat{Q}^{\eps}|R^{\eps})$ and $H(\overline{Q}^{\eps}|R^{\eps})$ are finite by a routine application of Girsanov's Theorem \cite[Theorem 1.4]{revuz2004continuous}. Thus, the formulas in \eqref{eq:rel-ent-formulas} follow from \cite[Theorem 2.3]{leo-gis12} as presented in Fact \ref{fact:girs-formula}. 
\end{proof}

We now establish some quick properties of the tangent Markov projection. 
First, we show that the drift of the tangent Markov projection is determined entirely by metric quantities, i.e.\ its density and corresponding tangent velocities in \eqref{defn:continuity-eqn}. 
\begin{proposition}\label{prop:relaxed-drift-calculation}
    Recall the setting and notation of Proposition \ref{prop:proj-processes-exist}. Let $\overline{\beta}_t^{\eps}$ be as defined in \eqref{eq:relaxed-l2-proj}. If $\int_0^{1} I(\mu_t) dt < +\infty$, then it holds for $\mathrm{Leb}$-a.e.\ $t \in [0,1]$ that
    \begin{align*}
        \overline{\beta}_t^{\eps} &= \overline{v}_t+ \frac{\eps}{2}\nabla \log \mu_t,
    \end{align*}
    where we recall that $(\mu_t,\overline{v}_t, t \in [0,1])$ is the marginal flow of $Q^{\eps}$ defined in \eqref{eq:recip-process-sde} with each $\overline{v}_t \in \mathrm{Tan}(\mu_t)$. 
\end{proposition}
\begin{proof}
    Recall $(\hat{\beta}_t^{\eps}, t \in [0,1])$ from \eqref{eq:mp-drift-defn}. Define the following velocity field 
    \begin{align*}
        \hat{v}_t^{\eps} := \hat{\beta}_t^{\eps} - \frac{\eps}{2}\nabla \log \mu_t.
    \end{align*}
    It holds that $(\mu_t, \hat{v}_t^{\eps}, t \in [0,1])$ is also a solution to the continuity equation, and by definition $\overline{v}_{t}$ is the projection of $\hat{v}_t^{\eps}$ onto $\Tan{\mu_t}$ for each $t \in [0,1]$.
    
    We now reparameterize the optimization problem in \eqref{eq:relaxed-l2-proj} to obtain for $\mathrm{Leb}$-a.e.\ $t \in [0,1]$
    \begin{align*}
        \overline{\beta}_t^{\eps} &= \argmin\limits_{\beta_t \in \Tan{\mu_t}} \Exp{\mu_t}\norm{\hat{\beta}_t^{\eps}-\beta_t}^2 
        = \argmin\limits_{\beta_t \in \Tan{\mu_t}} \Exp{\mu_t}\norm{\hat{v}_t^{\eps}-\left(\beta_t-\frac{\eps}{2}\nabla \log \mu_t\right)}^2 \\
        &= \frac{\eps}{2}\nabla \log \mu_t + \argmin\limits_{w_t \in \Tan{\mu_t}} \Exp{\mu_t}\norm{\hat{v}_t^{\eps}-w_t}^2 
        = \frac{\eps}{2}\nabla \log \mu_t + \overline{v}_t. 
    \end{align*}
    Note that the third equality follows from the fact that $\Tan{\mu_t} \subset L^2(\mu_t)$ is a closed subspace and $\nabla \log \mu_t \in \Tan{\mu_t}$ as $I(\mu_t) < +\infty$ for $\mathrm{Leb}$-a.e.\ $t \in [0,1]$. 
\end{proof}

The following is similar to the Pythagorean Theorems in relative entropy satisfied by Markov processes \cite[Lemma 6]{shi2023diffusion} and of the $\Schro$ bridge \cite{csiszar75idiv}. However, we must constrain the collection of stochastic processes we consider. 
\begin{proposition}\label{prop:tangent-markov-pythag}
Recall the setting and notation of Proposition \ref{prop:proj-processes-exist}. Let $Q^{\eps}, \overline{Q}^{\eps}$ denote the reciprocal process and its tangent Markov projection, respectively, defined in \eqref{eq:recip-process-sde} and \eqref{eq:tan-mp-sde}. Let $(\mu_t, t \in [0,1])$ denote the marginal flow of $Q^{\eps}$. 

Let $(\gamma_t, t \in [0,1])$ be such that $\gamma_t \in \Tan{\mu_t}$, and suppose that the following SDE has a weak solution
    \begin{align*}
        dX_t &= \gamma_t(X_t)dt + \sqrt{\eps} dB_t, \quad X_0 \sim \mu_0. 
    \end{align*}
    Let $N \in \cP(C^{d}[0,1])$ denote the law of $(X_t, t \in [0,1])$. Then when all the quantities below are finite, it holds that
    \begin{align*}
        H(Q^{\eps}|N) &= H(Q^{\eps}|\overline{Q}^{\eps})+H(\overline{Q}^{\eps}|N).
    \end{align*}
\end{proposition}
\begin{proof}
    The key point is the following observation. Fix $t \in [0,1]$, then by conditioning on $X_{t}$ the Tower property gives
    \begin{align*}
        \Exp{Q^{\eps}}\left[\langle \beta_t^{\eps}(X_0,X_t)- \overline{\beta}_t^{\eps}(X_t),\overline{\beta}_{t}^{\eps}(X_t)-\gamma_t(X_t)\rangle\right] &= \Exp{\mu_t} \left[\langle \hat{\beta}_t^{\eps}-\overline{\beta}_t^{\eps},\overline{\beta}_t^{\eps}-\gamma_t\rangle\right]=0,
    \end{align*}
    the last equality follows from the fact that $(\hat{\beta}_t^{\eps}-\overline{\beta}_t^{\eps}) \in \Tan{\mu_t}^{\perp}$ by  \eqref{eq:relaxed-l2-proj} and $(\overline{\beta}_t^{\eps}-\gamma_t) \in \Tan{\mu_t}$. Thus,
    \begin{align*}
        2\eps H(Q^{\eps}|N) &= \Exp{Q^{\eps}} \int_0^{1} \norm{\beta_t^{\eps}(X_0,X_t) - \gamma_t(X_t)}^2 dt \\
        &= \Exp{Q^{\eps}} \int_0^{1} \norm{\beta_t^{\eps}(X_0,X_t) - \overline{\beta}_t^{\eps}(X_t)}^2 dt +  \int_0^{1} \Exp{\mu_t}\norm{\overline{\beta}_t^{\eps}- \gamma_t}^2 dt \\
        &+ \int_0^{1} \Exp{Q^{\eps}}\left[\langle \beta_t^{\eps}(X_0,X_t)- \overline{\beta}_t^{\eps}(X_t),\overline{\beta}_{t}^{\eps}(X_t)-\gamma_t(X_t)\rangle\right] dt \\
        &= 2\eps \left(H(Q^{\eps}|\overline{Q}^{\eps})+H(\overline{Q}^{\eps}|N)\right).
    \end{align*}
    The first and last steps follow from Fact \ref{fact:girs-formula}. 
\end{proof}

We now conclude with our first theorem, giving an expansion of relative entropy for the tangent Markov projection that is reminiscent of the entropic Benamou-Brenier in Fact \ref{prop:entropic-bb}. 
\begin{theorem}\label{thm:rel-ent-tangent-markov-proj}
Recall the setting and notation of Proposition \ref{prop:proj-processes-exist}. Let $Q^{\eps}, \overline{Q}^{\eps}$ denote the reciprocal process and its tangent Markov projection, respectively, defined in \eqref{eq:recip-process-sde} and \eqref{eq:tan-mp-sde}. Let $(\mu_t, t \in [0,1])$ denote the marginal flow of $Q^{\eps}$. Then it holds that
    \begin{align}\label{eq:pythag-recip-tan-mp}
        H(Q^{\eps}|R^{\eps}) &= H(Q^{\eps}|\overline{Q}^{\eps})+H(\overline{Q}^{\eps}|R^{\eps}).
    \end{align}
    Moreover, whenever $\int_0^{1} I(\mu_t) dt < +\infty$ it holds that
    \begin{align}\label{eq:rel-ent-decomp-mp}
        H(\overline{Q}^{\eps}|R^{\eps}) &= \frac{1}{2\eps}\int_{0}^{1} |(\mu_t)'|^2 dt +\frac{1}{2}\left(\Ent(\mu_0)+\Ent(\mu_1)\right)+\frac{\eps}{8}\int_0^{1} I(\mu_t)dt.
    \end{align}
\end{theorem}
\begin{proof}[Proof of Theorem \ref{thm:rel-ent-tangent-markov-proj}]
The first equality in the theorem follows by setting $\gamma_{t} = 0$ for all $t \in [0,1]$ in Proposition \ref{prop:tangent-markov-pythag} with the initial condition $X_0 \sim \mu_0$, call this path measure $R^{\eps}_{\mu_{0}}$. Then, observe that $H(P|R^{\eps}_{\mu_{0}}) = H(P|R^{\eps})-\Ent(\mu_0)$ for all $P \in \cP(C^{d}[0,1])$. Next, by the factorization of relative entropy and Fubini observe that
\begin{align*}
    H(\overline{Q}^{\eps}|R^{\eps}) &= \Ent(\mu_0) + \Exp{\mu_0}\left[H(\overline{Q}^{\eps}_{x}|R^{\eps}_{x})\right] \\
    &= \Ent(\mu_0) + \Exp{\mu_0}\left[\frac{1}{2\eps}\int_{0}^{1} \Exp{\overline{Q}^{\eps}_{x}}\norm{\overline{\beta}_t^{\eps}(X_t)}^2dt\right] 
    = \Ent(\mu_0) + \frac{1}{2\eps}\int_0^{1} \Exp{\mu_t}\norm{\overline{\beta}_t^{\eps}}^2 dt. 
\end{align*}
From Proposition \ref{prop:relaxed-drift-calculation} then,
\begin{align*}
    \frac{1}{2\eps}\int_0^{1} \Exp{\mu_t}\norm{\overline{\beta}_t^{\eps}}^2dt &= \frac{1}{2\eps}\int_0^{1} \Exp{\mu_t}\norm{\overline{v}_t+\frac{\eps}{2}\nabla \log \mu_t}^2dt \\
    &= \frac{1}{2\eps}\int_0^{1} \norm{\overline{v}_t}^2_{L^2(\mu_t)}dt + \frac{1}{2}\int_0^{1} \Exp{\mu_t}\left[\langle \overline{v}_t,\nabla \log \mu_t \rangle\right]dt + \frac{\eps}{8}\int_0^{1} I(\mu_t)dt. 
\end{align*}
The stated claim then follows from noting that $\norm{\overline{v}_t}^2_{L^2(\mu_t)} = |(\mu_t)'|^2$, a.s., and
\begin{align*}
    \int_0^{1} \Exp{\mu_t}\left[\langle \overline{v}_t,\nabla \log \mu_t\rangle\right]dt &= \int_0^{1} \frac{d}{dt} \Ent(\mu_t)dt = \Ent(\mu_1)-\Ent(\mu_0),
\end{align*}
by \cite[Section 10.1.2(E)]{ambrosio2005gradient}.
\end{proof}

\section{Entropic Potential Expansion}\label{sec:schro-potent-exp}
In this section we prove the entropic Brenier map limit stated in Main Theorem \ref{informal-theorem-1}. This will follow by comparing the $\eps$-$\Schro$ bridge with the \textbf{$\eps$-noisy McCann interpolation} introduced in the introduction in \eqref{eq:noisy-mcann-intro}. We summarize the main result of this section in the following statement.
\begin{theorem}\label{thm:sb-o-eps-score}
    Let $P^{\eps}$ denote the $\eps$-dynamic $\Schro$ bridge, and let $\widetilde{Q}^{\eps}$ denote the law of the $\eps$-noisy McCann interpolation in \eqref{eq:noisy-mcann-intro}. Under Assumption \ref{assumption:standard-assumptions},
    \begin{align}\label{eq:rel-ent-approx-sb-mccann}
        H(\widetilde{Q}^{\eps}|P^{\eps}) = o(\eps). 
    \end{align}
    As a consequence, the following limit holds in $L^2(e^{-f})$,
    \begin{align}\label{eq:thm-statement-sb-score}
        \lim\limits_{\eps \downarrow 0}\frac{1}{\eps}\left(\Exp{\pi^{\eps}}[Y|X=x] - \nabla \varphi(x)\right) = -\frac{1}{2}\nabla f(x).
    \end{align}
    Equivalently, with $(f_{\eps}, \eps > 0)$ being the entropic potentials defined in \eqref{defn:entropic-potent}, in $L^2(e^{-f})$
    \begin{align}
        \lim\limits_{\eps \downarrow 0}\frac{1}{\eps}(\nabla f_{\eps}(x) - (x-\nabla \varphi(x))) = -\frac{1}{2}\nabla f(x). 
    \end{align}
\end{theorem}

Recall that the $\eps$-noisy McCann interpolation is defined as the solution to the following SDE
\begin{align}\label{defn:sec5-mccann-interp}
    dY_t = \alpha_t^{\eps}(Y_t) dt + \sqrt{\eps}dB_t, \quad Y_0 \sim e^{-f},
\end{align}
where 
\begin{align}
    \alpha_t^{\eps}(x) = v_t^0(x) + \frac{\eps}{2}\nabla \log \rho_t^0(x) = \nabla \varphi_{t \to 1}(x)-\nabla \varphi_{t \to 0}(x)+\frac{\eps}{2}\nabla \log \rho_t^0(x).
\end{align}
As argued in Section \ref{subsec:ot}, $Y_t \sim \rho_t^0$ for all $t \in [0,1]$. Moreover, by setting $\eps = 0$ in \eqref{defn:sec5-mccann-interp} we recover the standard McCann interpolation. Note that by the Hessian assumptions on $f$ and $h$ as well as the bounded higher derivatives on the Brenier map in Assumption \ref{assumption:standard-assumptions}, each $\alpha_t^{\eps}$ is Lipschitz with a constant that can be picked to be uniform over $t \in [0,1]$. Thus, the SDE in \eqref{defn:sec5-mccann-interp} has a strong solution. 

The main result we will prove about the $\eps$-noisy McCann interpolation is the following limit analogous to \eqref{eq:thm-statement-sb-score}. 
\begin{proposition}\label{prop:d-dt-matrix-identity}
    Under Assumption \ref{assumption:standard-assumptions}, fix $\eps > 0$ and let $(Y_t, t \geq 0)$ denote the $\eps$-noisy McCann interpolation defined in \eqref{defn:sec5-mccann-interp}. Denote the law of $(Y_t, t \in [0,1])$ by $\widetilde{Q}^{\eps}$. It holds in $L^2(e^{-f})$ that
    \begin{align}\label{eq:cond-exp-mccann}
        \lim\limits_{\eps \downarrow 0} \frac{1}{\eps}\left(\Exp{\widetilde{Q}^{\eps}}[Y_1|Y_0 = x] - \nabla \varphi(x)\right) = -\frac{1}{2}\nabla f(x). 
    \end{align}
\end{proposition}

We now establish how Theorem \ref{thm:sb-o-eps-score} follows from Proposition \ref{prop:d-dt-matrix-identity}.
\begin{proof}[Proof of Theorem \ref{thm:sb-o-eps-score}]
The proof of this theorem follows from two steps. 

\textbf{Step 1: Establish Relative Entropy Comparison.}
This follows from Theorem \ref{thm:rel-ent-tangent-markov-proj}. First, as argued in \cite[Theorem 4]{MP25}, since $\widetilde{Q}^{\eps}, P^{\eps}$ are path measures with endpoint marginals $e^{-f}$ and $e^{-h}$ and $dP^{\eps}/dR^{\eps}$ has the product structure in \eqref{eq:fg-decomp},
\begin{align}\label{eq:split-1}
    H(\widetilde{Q}^{\eps}|P^{\eps})= H(\widetilde{Q}^{\eps}|R^{\eps})-H(P^{\eps}|R^{\eps}).
\end{align}
By \cite[Theorem 1.6]{conforti21deriv} (extended to the current noncompact setting in \cite[Proposition 15]{MP25}), it holds as $\eps \downarrow 0$ that
\begin{align}\label{eq:ct-exp-again}
    H(P^{\eps}|R^{\eps}) &= \frac{1}{2\eps}\Was{2}^2(e^{-f},e^{-h}) + \frac{1}{2}\left(\Ent(e^{-f})+\Ent(e^{-h})\right)+\frac{\eps}{8}\int_0^{1} I(\rho_t^0)dt + o(\eps).
\end{align}
Recall that $\widetilde{Q}^{\eps}$ has marginal flow given by the McCann interpolation \eqref{eq:mccann-interp}. As $\widetilde{Q}^{\eps}$ is its own tangent Markov projection, Theorem \ref{thm:rel-ent-tangent-markov-proj} gives
\begin{align}\label{eq:tan-mark-rel-ent}
    H(\widetilde{Q}^{\eps}|R^{\eps}) =  \frac{1}{2\eps}\Was{2}^2(e^{-f},e^{-h}) + \frac{1}{2}\left(\Ent(e^{-f})+\Ent(e^{-h})\right)+\frac{\eps}{8}\int_0^{1} I(\rho_t^0)dt.
\end{align}
It follows from \eqref{eq:split-1}, \eqref{eq:ct-exp-again}, and \eqref{eq:tan-mark-rel-ent} that $H(\widetilde{Q}^{\eps}|P^{\eps}) = o(\eps)$.

\textbf{Step 2: Pass limit to the $\Schro$ bridge.} 
Set $\widetilde{\ell}_{\eps} := (\omega_0,\omega_1)_{\#}\widetilde{Q}^{\eps}$, and observe that $\Exp{\widetilde{\ell}_{\eps}}[Y|X=x] = \Exp{\widetilde{Q}^{\eps}}[X_1|X_0 = x]$. By the information processing inequality, $H(\widetilde{\ell}_{\eps}|\pi^{\eps}) \leq H(\widetilde{Q}^{\eps}|P^{\eps}) = o(\eps)$. We now pass the limit from Proposition \ref{prop:d-dt-matrix-identity} to $\pi^{\eps}$ via a functional inequality.

By the Hessian bounds on the log densities of the marginals in Assumption \ref{assumption:standard-assumptions}, it follows from \cite[Lemma 6, Theorem 8]{chewi2022entropic} that for all $\eps > 0$ small enough there exists a constant $\kappa > 0$ such that $-\nabla^2 \log \pi^{\eps}_{x}(y) \geq \eps^{-1}\kappa \Id$ for all $y \in \mathbb{R}^{d}$. Note that we use $\pi^{\eps}_{x}$ to denote the conditional distribution $\pi^{\eps}(\cdot|x)$. Thus, $\pi^{\eps}_{x}$ satisfies a Talagrand inequality with constant $\eps/\kappa$ \cite[Corollary 9.3.2]{bgl-markov}. By repeating the gluing coupling construction from \cite[Theorem 2]{MP25}, we obtain a triplet $(X,Y,Z)$ such that $(X,Y) \sim \pi^{\eps}$, $(X,Z) \sim \widetilde{\ell}_{\eps}$, and the joint distribution of $Y|X=x$ and $Z|X=x$ is equal to the $\Was{2}$ optimal coupling for $e^{-f}$-a.s.\ $x \in \mathbb{R}^{d}$. Altogether then,
    \begin{align*}
        \norm{\Exp{\pi^{\eps}}[Y|X=x]-\Exp{\widetilde{\ell}_{\eps}}[Z|X=x]}^2 &\leq \Exp{}\left[\norm{Y-Z}|X=x\right]^2 \\
    &\leq  \Was{2}^2(\pi^{\eps}_{x},\widetilde{\ell}_{\eps}(\cdot|x)) 
    \leq  \frac{2\eps}{\kappa} H(\widetilde{\ell}_{\eps}(\cdot|x)|\pi^{\eps}_{x}).
    \end{align*}
    Integrating with respect to $e^{-f}$ and rearranging terms, 
     \begin{align*}
        \frac{1}{\eps}\norm{\Exp{\pi^{\eps}}[Y|X=x]-\Exp{\widetilde{\ell}_{\eps}}[Y|X=x]}_{L^2(e^{-f})} &\leq \sqrt{\frac{1}{\eps^2}\frac{2\eps}{\kappa}H(\widetilde{\ell}_{\eps}|\pi^{\eps})},
    \end{align*}
    which vanishes as $\eps \downarrow 0$ by \eqref{eq:rel-ent-approx-sb-mccann} established in Step 1. Note that the same argument will work for any Lipschitz function of the $Y$ coordinate. 
\end{proof}

\begin{remark}
    Step 2 in the above proof is similar to the proofs of \cite[Theorem 2]{AHMP25}, \cite[Theorem 2]{MP25}. There is one key difference: Step 1 gives a bound on $H(\widetilde{\ell}_{\eps}|\pi^{\eps})$, whereas \cite{MP25,AHMP25} work with a bound on the relative entropy with the arguments reversed. In the current case, we must use functional inequalities that hold on the conditional densities of the $\Schro$ bridge. In fact, this turns out to be for our benefit as we do not have to produce a functional inequality for the $\eps$-noisy McCann interpolation. 
\end{remark}

Hence, the rest of this section is now devoted to proving Proposition \ref{prop:d-dt-matrix-identity} below in Section \ref{subsec:noisy-mccann-proof}.

\subsection{Proof of Proposition \ref{prop:d-dt-matrix-identity}}\label{subsec:noisy-mccann-proof}
For convenience, we recall that the $\eps$-noisy McCann interpolation is the solution to the following SDE
\begin{align}\label{defn:sec51-mccann-interp}
    dY_t = \alpha_t^{\eps}(Y_t) dt + \sqrt{\eps}dB_t, \quad Y_0 \sim e^{-f},
\end{align}
where 
\begin{align}\label{eq:alpha-t-defn}
    \alpha_t^{\eps}(x) = v_t^0(x) + \frac{\eps}{2}\nabla \log \rho_t^0(x) = \nabla \varphi_{t \to 1}(x)-\nabla \varphi_{t \to 0}(x)+\frac{\eps}{2}\nabla \log \rho_t^0(x).
\end{align} 
We use $\widetilde{Q}^{\eps}$ to denote the law of $(Y_t, t \in [0,1])$ satisfying \eqref{defn:sec51-mccann-interp}.

In the course of the proof, we will use a couple of important technical calculations that are presented later in the Appendix. In particular, we advise readers to consult the matrix divergence identity in Proposition \ref{prop:matrix-divergence-identity} as well as the matrix calculations involving the McCann interpolation in Proposition \ref{prop:sqrd-inv-hessian-iden}. 
\begin{proof}[Proof of Proposition \ref{prop:d-dt-matrix-identity}]
    Fix $\eps > 0$. The goal of this calculation is to identify the mean deviation of $Y_{1}|Y_0 =x$ about $\nabla \varphi(x)$. To facilitate this calculation, we perform a change of variables to ``center'' the process. In particular, we observe that $(\nabla \varphi_{t \to 1})_{\#}\rho_t^0 = e^{-h}$ for all $t \in [0,1]$. That is, under this time-inhomogeneous transformation, the process in \eqref{defn:sec51-mccann-interp} becomes a time-inhomogeneous diffusion with a stationary measure equal to $e^{-h}$. We will thus take the map $(t,x) \mapsto \nabla \varphi_{t \to 1}(x)$ as our centering transformation. 

    Define $Z_t = \nabla \varphi_{t \to 1}(Y_t)$. To simplify some bookkeeping, we define the following quantities
    \begin{align}
        a_t(z) := \left(\nabla^2 \varphi_{1 \to t}(z)\right)^{-1}, \quad A_{t}(z) := \left(\nabla^2 \varphi_{1 \to t}(z)\right)^{-2} = a_t^2(z). 
    \end{align}

    By $\Ito$'s formula, the SDE for $(Z_t, t \geq 0)$ is
    \begin{align}\label{eq:pushforward-mccann-interp}
        dZ_t = \gamma_{t}^{\eps}(Z_t)dt+ \sqrt{\eps} a_{t}(Z_t) dB_t, \quad Z_0 \sim e^{-h},
    \end{align}
    where the drift is given by
    \begin{align}\label{eq:pushforward-drift-mccann}
        \gamma_{t}^{\eps}(z) &= \left(\nabla^2 \varphi_{1 \to t}(z)\right)^{-1}\alpha_{t}^{\eps}(\nabla \varphi_{1 \to t}(z))+\partial_{t}[\nabla \varphi_{t \to 1}](\nabla \varphi_{1 \to t}(z))+\frac{\eps}{2}\Delta \nabla \varphi_{t \to 1}(\nabla \varphi_{1 \to t}(z)).
    \end{align}
    with $\alpha_t^\eps$ given in \eqref{eq:alpha-t-defn}.

    \textbf{Step 1:} The drift in \eqref{eq:pushforward-drift-mccann} can be written in the following divergence form
    \begin{align}\label{eq:simple-transformed-drift}
        \gamma_t^{\eps}(z) = -\frac{\eps}{2}\left(A_{t}(z)\nabla h(z) - \nabla_z \cdot A_{t}(z)\right).
    \end{align}
    This is a straightforward calculation we relegate to Proposition \ref{prop:noisy-mccann-drift-into-divergence} in the Appendix.

    \textbf{Step 2:} Next, we show that integrating $\gamma_t^{\eps}$ over $t \in [0,1]$ gives
    \begin{align}\label{eq:gamma-t-time-iden}
        \int_0^{1} \gamma_t^{\eps}(z) dt = -\frac{\eps}{2}\nabla f(\nabla \varphi^*(z)). 
    \end{align}
    For each time $t \in [0,1]$, recall that the (forward/backward) \textbf{stochastic derivative} at $t$, defined in \cite[Definition 2.6]{leonard2011stochastic}, is the differential operator $\mathcal{L}_{t}$ defined for $u \in C_c^{\infty}(\mathbb{R}^{d})$ by 
    \begin{align}
        \mathcal{L}_t u(z) := \lim\limits_{s \to t} \frac{1}{s-t} \Exp{\widetilde{Q}^{\eps}}[u(Z_{s})-u(Z_t)|Z_t = z].
    \end{align}
   
    From \eqref{eq:pushforward-mccann-interp} and \eqref{eq:simple-transformed-drift}, the stochastic derivative of the process $(Z_t, t \geq 0)$ is equal to  
    \begin{align}\label{eq:stoch-deriv-1}
        \mathcal{L}_t u(z) = -\frac{\eps}{2}\langle A_t(z)\nabla h(z) - \nabla_z \cdot A_t(z), \nabla u(z)\rangle+\frac{\eps}{2}A_{t}(z): \nabla^2 u(z). 
    \end{align}
    An equivalent way to write \eqref{eq:stoch-deriv-1} is to put it in divergence form. For $t \in [0,1]$ it holds that
    \begin{align}\label{eq:stoch-deriv-2}
        \mathcal{L}_{t} u(z) = \frac{\eps}{2e^{-h(z)}} \nabla_{z} \cdot \left(e^{-h(z)} A_{t}(z) \nabla u(z)\right)
    \end{align}
    We now pick the function $u(z) = z$. Observe that this is smooth and square integrable with respect to all $(\rho_t^0, t \in [0,1])$. Plugging in this choice of $u$ into \eqref{eq:stoch-deriv-1} and \eqref{eq:stoch-deriv-2} gives the identity
    \begin{align}
        \gamma_{t}^{\eps}(z) &= -\frac{\eps}{2}\left(A_{t}(z)\nabla h(z) - \nabla_z \cdot A_{t}(z)\right) = \frac{\eps}{2e^{-h(z)}}\nabla \cdot \left(e^{-h(z)}A_t(z)\right).
    \end{align}
    Integrate the above identity in $t \in [0,1]$. From the identity in Proposition \ref{prop:sqrd-inv-hessian-iden}, it holds that
    \begin{align}\label{eq:time-int-generator}
        \int_0^{1} \gamma_t^{\eps}(z)dt &= \frac{\eps}{2e^{-h(z)}}\nabla_{z} \cdot \left(e^{-h(z)}\left(\int_0^{1} A_t(z) dt\right)\right) = \frac{\eps}{2e^{-h(z)}} \nabla_{z} \cdot \left(e^{-h(z)} \left(\nabla^2 \varphi_{1 \to 0}(z)\right)^{-1}\right).
    \end{align}
    Recall that $\varphi_{1 \to 0} = \varphi^*$. Let $\overline{\cL}$ denote the generator of the dual Mirror Langevin diffusion defined in \eqref{defn:dMLD}, written in divergence form in \eqref{eq:div-form-mld-dual}. Observe that the right hand side of \eqref{eq:time-int-generator} is equal to $\eps \overline{\cL}u$, where $u(z) = z$, which is equal to the drift of the dual MLD. Hence, 
    \begin{align}
        \int_{0}^{1} \gamma_t^{\eps}(z) dt = -\frac{\eps}{2}\nabla f(\nabla \varphi^{*}(z)). 
    \end{align}

    \textbf{Step 3:} We now use  Step 2 to establish \eqref{eq:cond-exp-mccann}. From Step 1, define the quantity
    \begin{align}\label{eq:theta-leading-order}
        \theta_t(z) = A_{t}(z)\nabla h(z) - \nabla_z \cdot A_{t}(z),
    \end{align}
    so that $\gamma_t^{\eps}(z) = -\frac{\eps}{2}\theta_t(z)$. By the SDE for $(Z_t, t \in [0,1])$ in \eqref{eq:pushforward-mccann-interp}, it holds that
    \begin{align}\label{eq:cond-exp-push}
        \Exp{\widetilde{Q}^{\eps}}[Z_1|Z_0 = \nabla \varphi(x)] - \nabla \varphi(x) = - \frac{\eps}{2}\int_0^{1} \Exp{\widetilde{Q}^{\eps}}[\theta_t(Z_t)|Z_0= \nabla \varphi(x)]dt.
    \end{align}
    By \eqref{eq:gamma-t-time-iden},
        $-\frac{\eps}{2}\int_0^{1} \theta_t(\nabla \varphi(x)) dt =\int_0^{1} \gamma_t^{\eps}(\nabla \varphi(x))dt = -\frac{\eps}{2}\nabla f(x)$.  
    Combining with \eqref{eq:cond-exp-push}  gives  
    \begin{align}\label{eq:remainder-identification}
        \frac{1}{\eps}\left(\Exp{\widetilde{Q}^{\eps}}[Z_1|Z_0 = \nabla \varphi(x)] - \nabla \varphi(x)\right)+\frac{1}{2}\nabla f(x) = \frac{1}{2}\int_0^{1}\Exp{\widetilde{Q}^{\eps}}\left[\theta_t(Z_t)-\theta_t(\nabla \varphi(x))|Z_0=\nabla \varphi(x)\right]dt. 
    \end{align}
    Let us argue that the right hand side of \eqref{eq:remainder-identification} vanishes in $L^2(e^{-f})$, as $\eps \downarrow 0$. By the global derivative bounds in Assumption \ref{assumption:standard-assumptions}, each $\theta_t$ is Lipschitz with some constant $L>0$ that is uniform in $t \in [0,1]$. By Jensen's inequality and Fubini,
    \begin{align*}
    &\int_{\mathbb{R}^{d}}\norm{\int_0^{1}\Exp{\widetilde{Q}^{\eps}}\left[\theta_t(Z_t)-\theta_t(\nabla \varphi(x))|Z_0=\nabla \varphi(x)\right]dt}^2 e^{-f(x)}dx \\
        &\leq L^2 \int_0^{1} \int_{\mathbb{R}^{d}} \Exp{\widetilde{Q}^{\eps}}\left[\norm{\theta_t(Z_t)-\theta_t(\nabla \varphi(x))}^2 |Z_0 = \nabla \varphi(x) \right] e^{-f(x)}dx dt \leq L^2 \int_0^{1} \Exp{\widetilde{Q}^{\eps}}\left[\norm{Z_t - Z_0}^2\right]dt. 
    \end{align*}
Finally, the fact that $\int_0^{1} \Exp{\widetilde{Q}^{\eps}}\left[\norm{Z_t-Z_0}^2\right]dt = O(\eps)$, follows from the SDE
\begin{align}
    Z_t - Z_0 &= \frac{\eps}{2}\int_0^{t} \theta_{s}(Z_s)ds + \sqrt{\eps}\int_0^{t} \left(\nabla^2 \varphi_{1 \to s}(Z_s)\right)^{-1}dB_s,
\end{align}
where $\theta_s$ has a linear bound that is uniform in $s$ and $\nabla^2 \varphi_{1 \to s}$ is uniformly bounded. 

To summarize, from \eqref{eq:remainder-identification} it holds in $L^2(e^{-f})$ that
\begin{align}\label{eq:zt-limit}
    \lim\limits_{\eps \downarrow 0}\frac{1}{\eps}\left(\Exp{\widetilde{Q}^{\eps}}[Z_1|Z_0 = \nabla \varphi(x)] - \nabla \varphi(x)\right)= -\frac{1}{2}\nabla f(x).
\end{align}
Recall that $Z_t = \nabla \varphi_{t \to 1}(Y_t)$. As $Z_1 = \nabla \varphi_{1 \to 1}(Y_1) = Y_1$ and $Z_0 = \nabla \varphi(Y_0)$, it holds that $\Exp{}[Z_1|Z_0 = \nabla \varphi(x)] = \Exp{}[Y_1|Y_0 = x]$, and thus \eqref{eq:zt-limit} also gives in $L^2(e^{-f})$ that
\begin{align}
    \lim\limits_{\eps \downarrow 0} \frac{1}{\eps}\left(\Exp{\widetilde{Q}^{\eps}}[Y_1|Y_0 = x] - \nabla \varphi(x)\right)= -\frac{1}{2}\nabla f(x).
\end{align}
\end{proof}

\section{Main Process Approximation}\label{sec:rel-ent-approx}
The main result in this section is Theorem \ref{thm:mp-mld-approx-of-sb}. In particular, we will take attention in this section to illustrate our assumptions and point out how they are used in our argument.

To begin, we establish notation and assumptions for this section and Section \ref{sec:kinetic-energy-bdds}. Recall Assumption \ref{assumption:standard-assumptions} from the Introduction. We will show in Proposition \ref{prop:uniform-poincare-constant} below that all the objects we work with in this section are well-defined. 

We now fix notation. Let $\ell_{\eps} = \mathrm{Law}(X_0,X_{\eps}^*)$, where $(X_t, t \geq 0)$ is the stationary primal MLD defined in \eqref{defn:MLD} with initial distribution $X_0 \sim e^{-f}$. Recall the notation $x \mapsto x^* = \nabla \varphi(x)$. Recall that $R^{\eps} \in \cM_{+}(C^{d}[0,1])$ is the law of $\eps$-reversible Brownian motion, and $(R^{\eps}_{xy},x,y\in\mathbb{R}^{d})$ denotes its bridges. We also use $(r_{t}(\cdot,\cdot),t > 0)$ to denote the standard Euclidean heat kernel. As before, let
\begin{align}\label{eq:tan-mp-mld-rp}
    \overline{Q}^{\eps}\; \text{be the tangent Markov projection of } Q^{\eps} = \int_{\mathbb{R}^{d} \times \mathbb{R}^{d}} R^{\eps}_{xy}\ell_{\eps}(dxdy). 
\end{align}
We denote the time marginal flow of $\overline{Q}^{\eps}$ and $Q^{\eps}$ (recall that they are the same) and its corresponding tangent velocities by
\begin{align}\label{eq:recip-marginal-flow}
    (\rho_t^{\eps}, v_t^{\eps}, t \in [0,1]), \quad v_t^{\eps} \in \Tan{\rho_t^{\eps}}. 
\end{align}
By definition, $(\rho_t^0, v_t^{0}, t \in [0,1])$ is the McCann interpolation as defined in \eqref{eq:cont-eqn-mccann}.

We now impose a more technical assumption related to the joint smoothness in $(t,\eps)$ of the family $(\rho_t^{\eps},v_t^{\eps}, t \in [0,1], \eps \geq 0)$ defined in \eqref{eq:recip-marginal-flow}. In particular, we require a certain amount of regularity jointly in the $t$ and $\eps$ variables that holds \text{up to and including the boundary} values $\eps = 0$ and $t = 0,1$. Moreover, we require the necessary domination so that the quantities defined below are in the appropriate weighted $L^2$ spaces. 
\begin{assumption}[Smoothness in $(t, \eps)$ at $\eps=0$]\label{assumption:technical-rd}
     
     Let $(\rho_t^{\eps}, v_t^{\eps}, t \in [0,1])$ as in \eqref{eq:recip-marginal-flow}. 
     Assume that the map $(t,s,x) \mapsto \rho_t^{s}(x)$ is twice continuously differentiable for $(t,s,x) \in [0,1] \times [0,+\infty) \times \mathbb{R}^{d}$, up to and including the boundary. Moreover, there exists $s_0 > 0$ and $C > 0$, $k \geq 1$ such that
     \begin{align}\label{eq:polynomial-dom-to-zero}
        \abs{\partial_s \partial_t \log \rho_t^{s}(x)} + \norm{\nabla \partial_s \log \rho_t^s(x)} \leq C\left(1+\norm{x}^{k}\right), \quad \text{for all $(t,s,x) \in [0,1]\times [0,s_0] \times \mathbb{R}^{d}$.}
     \end{align}
\end{assumption}

We pause to make some remarks on the above assumptions.

\begin{remark}
    We emphasize that Assumption \ref{assumption:technical-rd} should be interpreted as providing sufficient joint regularity in $(t,\eps)$ up to and including their boundary values. Indeed, by parabolic regularity it holds that $(t,\eps,x) \mapsto \rho_t^{\eps}(x)$ has mixed derivatives of many orders over $(t,\eps,x) \in (0,1) \times (0,+\infty) \times \mathbb{R}^{d}$. The uniform polynomial bound in \eqref{eq:polynomial-dom-to-zero} along with the uniform sub-Gaussian tail control established below in Proposition \ref{prop:uniform-poincare-constant} will provide the necessary domination as $\eps \downarrow 0$.  
    
    On the torus, compactness allows us to avoid the bound in \eqref{eq:polynomial-dom-to-zero}. 
    Indeed, Assumption \ref{assumption:torus} just requires that the map $(t,s,x) \mapsto \rho_t^{s}(x)$ is in $C^{\infty}([0,1] \times [0,+\infty) \times \mathbb{T}^{d})$. 

    We also remark that Assumption \ref{assumption:technical-rd} holds for Gaussian marginals (see the univariate calculations in Section \ref{subsec:gaussian-computations}) as well as a nontrivial class of marginals developed in the Appendix in Proposition \ref{prop:example-class}. 
\end{remark}

One of our main theorems is as follows. Once again, recall that $\ell_{\eps} = \mathrm{Law}(X_0,X_{\eps}^*)$, where $(X_t, t \geq 0)$ is the stationary primal MLD \eqref{defn:MLD}, 
$Q^\eps$ is the law of the reciprocal process obtained by joining $X_0$ and $X_\eps^*$ by an $\eps$-Brownian bridge, and $\overline{Q}^\eps$ is its tangent Markov projection, as in \eqref{eq:tan-mp-mld-rp}. Let $\overline{\ell}_{\eps} := (\omega_0,\omega_1)_{\#}\overline{Q}^{\eps}$ be the joint distribution of the endpoints under $\overline{Q}^\eps$.

\begin{theorem}\label{thm:mp-mld-approx-of-sb}
    Let $P^{\eps}$ denote the dynamic $\Schro$ bridge from $e^{-f}$ to $e^{-h}$. Under Assumptions \ref{assumption:standard-assumptions} and \ref{assumption:technical-rd},
    \begin{align}
        \lim\limits_{\eps \downarrow 0}\frac{1}{\eps}H(\overline{Q}^{\eps}|P^{\eps}) = 0 \text{ and } \lim\limits_{\eps \downarrow 0} \frac{1}{\eps}H(\overline{\ell}_{\eps}|\pi^{\eps}) = 0. 
    \end{align}
\end{theorem}

Before passing to the proof of Theorem \ref{thm:mp-mld-approx-of-sb}, we point out some technical results that hold in our setting under Assumptions \ref{assumption:standard-assumptions} and \ref{assumption:technical-rd}. Importantly, we establish a uniform choice of Poincaré inequality constant across the curve of measures defined in \eqref{eq:recip-marginal-flow}. Recall that a measure $\mu \in \cP(\mathbb{R}^{d})$ satisfies a Poincaré inequality with constant $C > 0$ if for all $f: \mathbb{R}^{d} \to \mathbb{R}$ in the domain of the Dirichlet form associated with $\mu$, see \cite[Section 3.4.1 (D4)]{bgl-markov}, it holds that
\begin{align}\label{defn:poincare-inequality}
    \int_{\mathbb{R}^{d}} \abs{f-\int_{\mathbb{R}^{d}}fd\mu}^2 d\mu \leq C\int_{\mathbb{R}^{d}} \norm{\nabla f}^2 d\mu.
\end{align}
We define $C_{P}(\mu) > 0$ to be the smallest constant for which \eqref{defn:poincare-inequality} holds. The proof of the following proposition is relegated to the Appendix.
\begin{proposition}\label{prop:uniform-poincare-constant}
The following results hold under Assumption \ref{assumption:standard-assumptions}.
\begin{itemize}
    \item[(1)] The MLD as defined in \eqref{defn:MLD} has a unique (weak) solution with a.s.\ infinite explosion time. Setting $\ell_{\eps}:=   \mathrm{Law}(X_0,X_{\eps}^*)$, it holds that $H(\ell_{\eps}|r_{\eps}) < +\infty$ for each $\eps > 0$. Moreover, for all $\eps_0 > 0$ it holds that
    \begin{align}\label{eq:uniform-poincare-constant}
        \sup\limits_{\eps \in (0,\eps_0], t \in [0,1]} C_{P}(\rho_t^{\eps}) < +\infty.
    \end{align}
    \item[(2)] It holds that $\Ent(e^{-f}), \Ent(e^{-h})$ are both finite. Thus, the $\Schro$ bridge from $e^{-f}$ to $e^{-h}$ exists at each $\eps > 0$. Additionally, each of $I(e^{-f})$, $I(e^{-h})$, and $\int_0^{1} I(\rho_t^0) dt$ are finite. 
    \item[(3)] The $(\rho_t^{\eps}, t \in [0,1], \eps \in [0,\eps_0])$ are uniformly sub-Gaussian for all $\eps_0$, meaning that the sub-Gaussian norm $\norm{\cdot}_{\psi_{2}}$ defined in \cite[Definition 2.5.6]{vershynin-hdp} has a uniform bound. For all $p \in [1,+\infty)$, it holds that $\rho_t^{\eps} \to \rho_t^0$ in $\Was{p}$ as $\eps \downarrow 0$ for all $t \in [0,1]$.
    \item[(4)] Additionally, under Assumption \ref{assumption:technical-rd}, it holds that
    \begin{align}\label{eq:fi-limit}
        \lim\limits_{\eps \downarrow 0} \int_0^{1} I(\rho_t^\eps) dt  = \int_0^{1} I(\rho_t^0)dt, \quad \lim\limits_{\eps \downarrow 0}\int_0^{1} I(\rho_t^{\eps}|\rho_t^0) dt = 0.   
    \end{align}
\end{itemize}
\end{proposition}

We are now able to present the proof of Theorem \ref{thm:mp-mld-approx-of-sb} contingent on the kinetic energy bound proven below in Theorem \ref{lem:ke-calc} in Section \ref{sec:kinetic-energy-bdds}.
\begin{proof}[Proof of Theorem \ref{thm:mp-mld-approx-of-sb}]
This follows from Theorem \ref{thm:rel-ent-tangent-markov-proj} and Proposition \ref{prop:tangent-markov-pythag}. First, by Assumption 1 it follows from Proposition \ref{prop:sb-in-tangent-space} in the Appendix that $P^{\eps}$ is a permissible choice for $N$ in Proposition \ref{prop:tangent-markov-pythag}. That is,
\begin{align}\label{eq:decomp-1}
    H(\overline{Q}^{\eps}|P^{\eps}) &= H(Q^{\eps}|P^{\eps}) -H(Q^{\eps}|\overline{Q}^{\eps}).
\end{align}
Next, recall that as argued in \cite[Theorem 4]{MP25}, due to the product structure of $\pi^{\eps}$ and the fact that $Q^{\eps}$, $P^{\eps}$, and $R^{\eps}$ have the same bridges,
\begin{align}\label{eq:decomp-2}
    H(Q^{\eps}|P^{\eps}) = H(\ell_{\eps}|\pi_{\eps})= H(\ell_{\eps}|r_{\eps})-H(\pi_{\eps}|r_{\eps}) = H(Q^{\eps}|R^{\eps})-H(P^{\eps}|R^{\eps}).
\end{align}
From \eqref{eq:decomp-1}, \eqref{eq:decomp-2}, and \eqref{eq:pythag-recip-tan-mp} in Theorem \ref{thm:rel-ent-tangent-markov-proj},
\begin{align}\label{eq:key-decomp}
    H(\overline{Q}^{\eps}|P^{\eps}) = \left(H(Q^{\eps}|R^{\eps})-H(Q^{\eps}|\overline{Q}^{\eps})\right)-H(P^{\eps}|R^{\eps}) = H(\overline{Q}^{\eps}|R^{\eps})-H(P^{\eps}|R^{\eps}).
\end{align}
By \cite[Theorem 1.6]{conforti21deriv} (extended to the current noncompact setting in \cite[Proposition 15]{MP25}), it holds as $\eps \downarrow 0$ that
\begin{align}\label{eq:sb-cost-expansion}
    H(P^{\eps}|R^{\eps}) &= \frac{1}{2\eps}\Was{2}^2(e^{-f},e^{-h}) + \frac{1}{2}\left(\Ent(e^{-f})+\Ent(e^{-h})\right)+\frac{\eps}{8}\int_0^{1} I(\rho_t^0)dt + o(\eps).
\end{align}
Recall that $\overline{Q}^{\eps}$ has marginal flow given by \eqref{eq:recip-marginal-flow}. We now look at the expression of $\eps \mapsto H(\overline{Q}^{\eps}|R^{\eps})$ given by \eqref{eq:rel-ent-decomp-mp} of Theorem \ref{thm:rel-ent-tangent-markov-proj}. With Assumption \ref{assumption:technical-rd}, Theorem \ref{lem:ke-calc} below in Section \ref{sec:kinetic-energy-bdds} establishes that
\begin{align}\label{eq:asympt-terms-in-rel-ent}
    \frac{1}{2\eps}\int_0^{1} |(\rho_t^{\eps})'|^2 dt &= \frac{1}{2\eps}\Was{2}^2(e^{-f},e^{-h})+o(\eps).
\end{align}
From \eqref{eq:fi-limit}, $\lim\limits_{\eps \downarrow 0} \int_0^{1} I(\rho_t^{\eps})dt = \int_0^{1} I(\rho_t^0)dt$. 
Plugging everything into the expansion of $H(\overline{Q}^{\eps}|R^{\eps})$ given in \eqref{eq:rel-ent-decomp-mp}, it follows from \eqref{eq:key-decomp} and \eqref{eq:sb-cost-expansion} that $H(\overline{\ell}_{\eps}|\pi_{\eps})\leq H(\overline{Q}^{\eps}|P^{\eps}) = o(\eps)$,
where the inequality follows from the information processing inequality.
\end{proof}

\section{Kinetic Energy Bound}\label{sec:kinetic-energy-bdds}
The main result of this section is the following kinetic energy estimate that gets used in \eqref{eq:asympt-terms-in-rel-ent} as a part of the proof of Theorem \ref{thm:mp-mld-approx-of-sb}. Continue with the same setting and notation introduced in Section \ref{sec:rel-ent-approx}. In particular,  $\overline{Q}^{\eps}$ is the tangent Markov projection of the reciprocal process formed from the MLD defined in \eqref{eq:tan-mp-mld-rp}, and has the marginal flow and tangent velocities  $(\rho_t^{\eps},v_t^{\eps},t \in [0,1])$ as in \eqref{eq:recip-marginal-flow}. 

\begin{theorem}\label{lem:ke-calc}
    Under Assumptions \ref{assumption:standard-assumptions} and \ref{assumption:technical-rd},
\begin{align}
    \frac{1}{2}\int_{0}^{1} |(\rho_t^{\eps})'|^{2}dt = \frac{1}{2}\Was{2}^{2}(e^{-f},e^{-h})+o(\eps^2). 
\end{align}
As a consequence, 
\begin{align}\label{eq:tan-vel-comp}
    \int_{0}^{1} \norm{v_t^{\eps}-v_{t}^0}_{L^2(\rho_t^{\eps})}^2 dt = o(\eps^2). 
\end{align}
\end{theorem}
From Theorem \ref{lem:ke-calc}, we obtain the following relative entropy approximation result between the $\eps$-noisy McCann interpolation $\widetilde{Q}^{\eps}$ in \eqref{eq:noisy-mcann-intro} and $\overline{Q}^{\eps}$. 
\begin{corollary}\label{prop:mp-drift-approx}
   Under Assumptions \ref{assumption:standard-assumptions} and  \ref{assumption:technical-rd},
    \begin{align}\label{eq:mp-drift-approx}
        \lim\limits_{\eps \downarrow 0}\frac{1}{\eps^2}\int_0^{1} \int_{\mathbb{R}^{d}}\norm{\overline{\beta}_{t}^{\eps}-\left(\nabla \varphi_{t \to 1}-\nabla \varphi_{t \to 0}+\frac{\eps}{2}\nabla \log \rho_{t}^0\right)}^2 \rho_t^{\eps}(dx)dt &= 0. 
    \end{align}
    In other words, $H(\overline{Q}^{\eps}|\widetilde{Q}^{\eps}) = o(\eps)$.
\end{corollary}
\begin{proof}[Proof of Corollary \ref{prop:mp-drift-approx}]
Recall the notation from \eqref{eq:recip-marginal-flow}. By Proposition \ref{prop:relaxed-drift-calculation}, it then holds for all $\eps > 0$ and Lebesgue-a.e.\ $t \in [0,1]$ that the drift of $\overline{Q}^{\eps}$ is
    $\overline{\beta}_{t}^{\eps}(x) = v_{t}^{\eps}(x) + \frac{\eps}{2}\nabla \log \rho_t^{\eps}(x)$. 
By Theorem \ref{lem:ke-calc},
\begin{align*}
    \int_{0}^{1} \norm{v_{t}^{\eps}-(\nabla \varphi_{t \to 1}-\nabla \varphi_{t \to 0})}^2_{L^2(\rho_t^{\eps})}dt = o(\eps^2). 
\end{align*}
Now, by \eqref{eq:fi-limit} in Proposition \ref{prop:uniform-poincare-constant}, 
\begin{align*}
    \int_0^{1} \int_{\mathbb{R}^{d}} \norm{\frac{\eps}{2}\nabla \log \rho_t^{\eps}-\frac{\eps}{2}\nabla \log \rho_t^0}^2 \rho_t^{\eps}(dx)dt = \frac{\eps^2}{4}\int_0^{1} I(\rho_t^{\eps}|\rho_t^0)dt = o(\eps^2).
\end{align*}
Let $(\alpha_t^{\eps}, t \in [0,1])$ denote the drift of $\widetilde{Q}^{\eps}$. Since both $\overline{Q}^{\eps}$ and $\widetilde{Q}^{\eps}$ have the same initial distribution, $e^{-f}$, from Fact \ref{fact:girs-formula},
$H(\overline{Q}^{\eps}|\widetilde{Q}^{\eps}) = H(e^{-f}|e^{-f})+\frac{1}{2\eps} \int_{0}^{1} \Exp{\rho_t^{\eps}}\norm{\overline{\beta}_t^{\eps}-\alpha_t^{\eps}}^2dt = o(\eps)$. 
\end{proof}

The proof of Theorem \ref{lem:ke-calc} is based on the following lemma that requires the reciprocal surface construction outlined in Section \ref{sec:sketch-of-proof}.
Consider the random surface $(U(t,s),\; t \in [0,1],\; s \geq 0)$ from \eqref{eq:defn-ust-intro}:
\begin{align}\label{eq:defn-ust}
     U(t,s) := (1-t)X_0 + tX_{s}^* + \sqrt{t(1-t)}W_{s}, \quad \mathrm{Law}(U(t,s))=\rho_t^{s}. 
\end{align}

\begin{lemma}\label{lem:zero-initial-velocity}
    Retain the setting of Theorem \ref{lem:ke-calc}. For each $t \in [0,1]$, it holds that 
    \begin{align}\label{eq:zero-in-dist}
        \left.\partial_{s}\rho_{t}^{s}\right|_{s = 0} = 0. 
    \end{align}
\end{lemma}

Let us explain the intuition behind this lemma. Note that $U(t,0) = \nabla \varphi_{0 \to t}(X_0) \sim \rho_{t}^{0}$, the time $t$ marginal on the McCann interpolation. 
To better interpret the curve $s \in [0,+\infty) \mapsto \rho_t^{s}$, for each $t \in [0,1]$, observe that
    $U(t,s) := \nabla \varphi_{0 \to t}(X_0) + t(X_s^*-X_0^*) + \sqrt{t(1-t)}W_s$.
Reasoning informally for $s \approx 0$, using the SDE \eqref{defn:MLD},  
\begin{align}\label{defn:ust-approx}
    U(t,s) - \nabla \varphi_{0 \to t}(X_0) &\approx -\frac{t}{2}\nabla f(X_0)ds + t \cdot \sqrt{\nabla^2 \varphi(X_0)} dB_s + \sqrt{t(1-t)} dW_s \\
    &\approx t\left(-\frac{1}{2}\nabla f(X_0) ds + \sqrt{\nabla^2 \varphi_{0 \to t}(X_0)} d\widetilde{B}_s \right),
\end{align}
where $(\widetilde{B}_t, t \geq 0)$ is another standard BM. 
Now, recall that the $t$-dual MLD defined in \eqref{defn:t-mld-interp-defn-intro} satisfies the following SDE and has stationary measure equal to $\rho_t^0 = (\nabla \varphi_{0 \to t})_{\#}e^{-f}$:
\begin{align}\label{defn:t-mld-interp}
    dZ_s = -\frac{1}{2}\nabla f(X_s) ds + \sqrt{\nabla ^2\varphi_{0\to t}(X_s)}dB_s. 
\end{align}
Observe that the difference in \eqref{defn:ust-approx} is approximately equal to $t \in [0,1]$ times the differential in  \eqref{defn:t-mld-interp}. 

Hence, one may interpret the family of random variables $(U(t,s), s \geq 0)$ for $s \approx 0$, for fixed $t$, as follows. Each $(U(t,s),s \geq 0)$ behaves approximately like the $t$-dual MLD, with initial distribution equal to its stationary distribution $\rho_t^0$, where time is scaled by its index $t \in [0,1]$. Since the curve $(\rho_t^{s}, s \geq 0)$ is at $s \approx 0$ the time marginal laws of a stationary stochastic process, the corresponding velocity in $s$ is zero, giving us \eqref{eq:zero-in-dist}.

\begin{proof}[Proof of Lemma \ref{lem:zero-initial-velocity}]
    The result is immediate when $t = 0$ as $U(0,s) = X_0$ for all $s \geq 0$ and thus $\rho_0^{s} = e^{-f}$ for all $s \geq 0$. Similarly, at $t = 1$ it holds that $U(1,s) = X_s^*$. By stationarity, it holds that $\rho_1^{s} = e^{-h}$ for all $s \geq 0$. 

    The nontrivial case is when $t \neq 0,1$. Fix $t \in (0,1)$ and let $\xi \in C_{c}^{\infty}(\mathbb{R}^{d})$. We introduce the following notation for matrices $A,B \in \mathbb{R}^{d \times d}$
    \begin{align}\label{eq:a-semi-b}
        A:B = \Tr(A^{T}B) = \sum_{i,j} A_{i,j}B_{i,j}.
    \end{align}
    Condition on $X_0$, then by $\Ito$'s formula (in the $s$ variable) and the SDE for $(X_u^{*}, u \geq 0)$ in \eqref{defn:MLD}, 
    \begin{align}\label{eqn:ito-t-s}
        \xi(U(t,s))-\xi(\nabla \varphi_{0 \to t}(X_0)) &= \int_0^{s} \langle \nabla \xi(U(t,u)), tdX_u^* \rangle + \frac{t^2}{2}\int_{0}^{s} \nabla^2 \xi(U(t,u)):\nabla^2 \varphi(X_u)  du \\
        &+ \int_0^{s} \langle \nabla \xi(U(t,u)),\sqrt{t(1-t)}dW_u \rangle + \frac{1}{2}\int_0^{s} \Delta \xi(U(t,u)) t(1-t)du.  
    \end{align}
    Recall from \eqref{eq:mccann-interp} that $\nabla \varphi_{0 \to t}(x) = (1-t)x + t\nabla \varphi(x)$. Computing the Hessians gives $\nabla^2 \varphi_{0 \to t} = (1-t)\Id + t \nabla^2 \varphi$. The two deterministic integrals in \eqref{eqn:ito-t-s} combine to give
    \begin{align}
        \frac{t^2}{2}\int_{0}^{s} \nabla^2 \xi(U(t,u)):\nabla^2 \varphi(X_u)  du  &+ \frac{1}{2}\int_0^{s} \Delta \xi(U(t,u)) t(1-t)du\\
        &= \frac{t}{2}\int_0^{s} \nabla^2 \xi(U(t,u)):\nabla^2 \varphi_{0 \to t}(X_u) du. 
    \end{align}
    Altogether then, \eqref{eqn:ito-t-s} becomes
    \begin{align}\label{eqn:ito-t-s-2}
        \xi(U(t,s))-\xi(\nabla \varphi_{0 \to t}(X_0)) &= \int_0^{s} \langle \nabla \xi(U(t,u)), tdX_u^* \rangle +  \frac{t}{2}\int_0^{s} \nabla^2 \xi(U(t,u)):\nabla^2 \varphi_{0 \to t}(X_u) du \\
        &+ \int_0^{s} \langle \nabla \xi(U(t,u)),\sqrt{t(1-t)}dW_u \rangle.
    \end{align}
    In \eqref{eqn:ito-t-s-2}, take $\Exp{}[\cdot |X_0]$. As $\xi \in C_c^{\infty}(\mathbb{R}^{d})$, the stochastic integrals are bounded and thus martingales and have zero expectation. Next, take full expectation with respect to $X_0 \sim e^{-f}$. Finally then, take the derivative with respect to $s$ to obtain
    \begin{align}\label{eq:d-ds-before-cond}
        \frac{d}{ds}\Exp{}[\xi(U(t,s))] &= \Exp{}\left[\langle\nabla \xi(U(t,s)),-\frac{t}{2}\nabla f(X_s) \rangle + \frac{t}{2} \nabla^2 \xi(U(t,s)):\nabla^2 \varphi_{0 \to t}(X_s) \right].
    \end{align}
    Next, send $s \downarrow 0$ in \eqref{eq:d-ds-before-cond}. This is justified on the right hand side by the path continuity of $s \mapsto X_s$ and $s \mapsto U(t,s)$ and the boundedness of $\xi$ and all its derivatives. Recall that $U(t,0) = \nabla \varphi_{0 \to t}(X_0)$, so it holds that
    \begin{align}\label{eq:initial-velocity}
        \left.\frac{d}{ds}\right|_{s=0}\Exp{}[\xi(U(t,s))] &= \Exp{}\left[\langle\nabla \xi(\nabla \varphi_{0 \to t}(X_0)),-\frac{t}{2}\nabla f(X_0) \rangle + \frac{t}{2} \nabla^2 \xi(\nabla \varphi_{0 \to t}(X_0)):\nabla^2 \varphi_{0 \to t}(X_0) \right].
    \end{align}
    Importantly: observe that this right hand side is $t$ times the action of the generator of the dual $t$-MLD in \eqref{defn:t-mld-interp} on $\xi$. 
    
    We will now perform an integration by parts of the second term on the right hand side of \eqref{eq:initial-velocity}. First, observe that
    \begin{align}\label{eq:ip-ident}
        \sum_{i=1}\partial_i \left[(\partial_{i} \xi)(\nabla \varphi_{0 \to t}(x))\right] = \sum_{i,j} \partial_{ij}\xi(\nabla \varphi_{0 \to t}(x))\partial_{ji}\varphi_{0 \to t}(x) = \nabla^2 \xi(\nabla \varphi_{0 \to t}(x)):\nabla^2 \varphi_{0 \to t}(x).  
    \end{align}
    From the identity in \eqref{eq:ip-ident}, integration by parts gives
    \begin{align*}
        \Exp{}&\left[\nabla^2 \xi(\nabla \varphi_{0 \to t}(X_0)):\nabla^2 \varphi_{0 \to t}(X_0)\right] = \int_{\mathbb{R}^{d}} \sum_{i=1}\partial_i \left[(\partial_{i} \xi)(\nabla \varphi_{0 \to t}(x)))\right] e^{-f(x)}dx \\
        &= \int_{\mathbb{R}^{d}} \langle \nabla f(x),\nabla \xi(\nabla \varphi_{0 \to t}(x))\rangle e^{-f(x)}dx = \Exp{}\left[\langle \nabla \xi(\nabla \varphi_{0 \to t}(X_0)),\nabla f(X_0)\rangle\right]. 
    \end{align*}
    With this calculation, \eqref{eq:d-ds-before-cond} becomes
        $\left.\frac{d}{ds}\int_{\mathbb{R}^{d}} h(x) \rho_t^s(dx)\right|_{s=0} = \left.\frac{d}{ds}\Exp{}[\xi(U(t,s))]\right|_{s=0} = 0$. 
    Hence, $\left.\partial_s \rho_t^{s}\right|_{s = 0} = 0$ in the sense of distributions. Under Assumption \ref{assumption:technical-rd}, $\left.\partial_s \rho_t^{s}\right|_{s = 0}$ classically exists and is a continuously differentiable function. Hence, $\left.\partial_s \rho_t^{s}\right|_{s = 0}$ identically vanishes on $\mathbb{R}^{d}$.
\end{proof} 

\begin{proof}[Proof of Theorem \ref{lem:ke-calc}]
We break this argument down into three steps. The first two are more routine manipulations, whereas the last step is technical where we use the full force of Assumption \ref{assumption:technical-rd} and its uniform polynomial bound \eqref{eq:polynomial-dom-to-zero}. 

\textbf{Step 1: Identify deviation from $\Was{2}$.} First, we show
\begin{align}\label{eq:deviation-velocity}
    \int_0^1 |(\rho_t^{\eps})'|^2 dt &= \Was{2}^2(e^{-f},e^{-h})+ \int_{0}^{1} \norm{v_t^\eps - v_t^0}_{L^2(\rho_t^{\eps})}^2 dt.
\end{align}
Fix $s > 0$. By definition, $(\rho_t^s, v_{t}^{s}, t \in [0,1])$ is a weak solution to the continuity equation  \eqref{defn:weak-soln}. For the test function to this continuity equation, use the Hamilton-Jacobi solution $(\psi_{t}, t \in [0,1])$ defined in \eqref{eq:hj-pde}. Note that we can pass from smooth, compactly supported test functions to $(\psi_t, t \in [0,1])$ due to the integrability established in Proposition \ref{prop:sb-in-tangent-space} in the Appendix. It holds that
\begin{align*}
    &\int_{\mathbb{R}^{d}} \psi_1 d\rho_1 - \int_{\mathbb{R}^{d}} \psi_0 d\rho_0 = \int_0^1 \int_{\mathbb{R}^{d}} \left(\partial_t \psi_t + \langle \nabla \psi_t , v_{t}^{s} \rangle \right)d\rho_t^s dt \\
    &= \int_0^{1} \int_{\mathbb{R}^{d}}\left(-\frac{1}{2}\norm{v_t^{0}}^2 + \langle v_t^{0}, v_{t}^{s}\rangle\right) d\rho_t^s dt = -\frac{1}{2}\int_{0}^{1} \int_{\mathbb{R}^{d}} \norm{v_t^s - v_t^0}^2 d\rho_{t}^{s}dt + \frac{1}{2}\int_{0}^{1} \norm{v_t^{s}}_{L^{2}(\rho_t^{s})}^{2} dt. 
\end{align*}
By \eqref{eq:hjb-ce-identity}, the LHS is equal to $\frac{1}{2}\Was{2}^2(e^{-f},e^{-h})$. As $\norm{v_t^{s}}_{L^{2}(\rho_t^{s})}^{2} = |(\rho_t^{s})'|^2$ a.s., \eqref{eq:deviation-velocity} holds. 

\textbf{Step 2: Pass to Negative Sobolev Norm.} Rather than working explicitly with the tangent velocity fields on the right hand side of \eqref{eq:deviation-velocity}, we will work with an equivalent distribution and corresponding negative Sobolev norm. This will allow us to apply the Poincaré inequality with uniform constant established in Proposition \ref{prop:uniform-poincare-constant}.

Recall the weak definition of $\sigma + \nabla \cdot (v\mu) = 0$ for a mean zero distribution $\sigma$ given in \eqref{defn:div-weak-sense}. For each $s \geq 0$ and $t \in [0,1]$ define the distribution $r_t^{s}$ as
\begin{align}\label{defn:r-s-t-dist}
    r_t^{s} &:= - \nabla \cdot (\rho_t^{s}(v_t^s - v_t^0)). 
\end{align}
Additionally, define the following quantities
\begin{align}\label{defn:q-s-t-diff}
    q_{t}^s := \log \left(\rho_t^{s}/\rho_t^{0}\right), \quad D_{t} := \partial_t + v_t^0 \cdot \nabla. 
\end{align}
The main inequality of this step is as follows. There exists $C > 0$ such that for all $s > 0$ small enough and $t \in [0,1]$
\begin{align}\label{eq:pass-to-dist}
    \norm{v_t^s - v_t^0}_{L^2(\rho^s_t)}^2 = \norm{r_{t}^{s}}_{\dot{H}^{-1}(\rho_t^{s})}^2 \leq C\norm{D_{t}q_{t}^{s}}_{L^2(\rho_t^s)}^2.
\end{align}
First, note that at $t = 0,1$ we have $v_t^{s} = v_t^0 = 0$ for all $s \geq 0$ because $\rho_1^{s} = e^{-h}$ and $\rho_0^{s} = e^{-f}$ for all $s \geq 0$. Since the LHS is zero, the inequality in \eqref{eq:pass-to-dist} holds immediately.

Fix $t \in (0,1)$. The first equality in \eqref{eq:pass-to-dist} follows from the definition of $\norm{\cdot}_{\dot{H}^{-1}(\mu)}$, which holds because $(v_t^s - v_t^0) \in \Tan{\rho_t^s}$ by Proposition \ref{prop:sb-in-tangent-space} in the Appendix. To obtain the inequality in \eqref{eq:pass-to-dist}, we argue as follows. Recall the definitions in \eqref{defn:q-s-t-diff}. Then, for all $t \in (0,1)$ and $s > 0$,
\begin{align*}
    r_{t}^{s} &= \partial_t (\exp(q_{t}^{s})\rho_t^0)+\nabla \cdot(\exp(q_{t}^{s})\rho_t^0 v_t^0) = \rho_t^0 \left(\partial_t + v_t^0 \cdot \nabla \right)\exp(q_{t}^{s}) \\
    &= \rho_t^{0} D_{t}\exp(q_t^{s})
    = \left(\rho_{t}^{0}\exp(q_{t}^{s})\right) \cdot \exp(-q_{t}^{s}) D_{t}\exp(q_{t}^{s}) 
    = \rho_t^{s} D_{t}q_{t}^{s}. 
\end{align*}
By Assumption \ref{assumption:standard-assumptions}, the chain rule application in the last step holds because these quantities classically exist for $(s,t) \in (0,+\infty) \times (0,1)$. By the definition \eqref{defn:r-s-t-dist}, $r_{t}^{s}$ is a mean zero distribution. Using the identity $r_{t}^{s} = \rho_t^{s} D_{t}q_{t}^{s}$ that we just derived, we can apply Proposition \ref{prop:neg-sob-norm-poincare} in the Appendix to obtain the inequality in \eqref{eq:pass-to-dist} with $C> 0$ being a common Poincaré constant across $(\rho_t^{s}, s \in (0,s_0), t \in (0,1))$, guaranteed to exist by Proposition \ref{prop:uniform-poincare-constant}.

\textbf{Step 3: Bound on Negative Sobolev Norm. } Now we have shown that the remainder quantity $\int_0^{1} \norm{v_t^{s}-v_t^0}_{L^2(\rho_t^{s})}^2dt$ in Step 1 can be bounded above by a constant multiple of $\int_0^{1} \norm{D_t q_t^{s}}^2_{L^2(\rho_t^s)}dt$. It remains to show that this upper bound decays with the desired rate $o(s^2)$. Hence, in this last step we will establish that under Assumption \ref{assumption:technical-rd},
\begin{align}\label{eq:bdd-on-sob-norm}
    \int_0^{1} \norm{D_{t}q_{t}^{s}}_{L^{2}(\rho_t^{s})}^2 dt  = o(s^2). 
\end{align}
To do this, we take a derivative with respect to $s$, apply the Fundamental Theorem of Calculus, and interchange the derivatives in $D_t = (\partial_t + v_t^0 \cdot \nabla)$ with $\partial_u$. These maneuvers are justified by the regularity in Assumption \ref{assumption:technical-rd}. Altogether then,
\begin{align}
    D_{t}q_{t}^{s}-D_{t}q_{t}^{0} &=  \int_0^{s} \partial_{u} (D_{t} q_{t}^{u}) du = \int_0^{s} D_{t}(\partial_{u}q_{t}^{u}) du.  
\end{align}
By definition $q_t^{0} = 0$ for all $t \in [0,1]$. By Jensen's inequality and Fubini
\begin{align}\label{eq:first-upper-bdd}
    \int_0^{1}\norm{D_{t}q_{t}^{s}}_{L^{2}(\rho_t^{s})}^2 dt &= \int_0^{1}\norm{\int_0^{s} D_{t}(\partial_{u}q_{t}^{u}) du}_{L^{2}(\rho_t^{s})}^2 dt \leq s^2  \cdot \left(\frac{1}{s}\int_0^{s}\int_0^{1}\norm{D_{t}(\partial_{u}q_{t}^{u})}^{2}_{L^2(\rho_t^{s})} dt du\right).
\end{align}
Let $s_0>0$ be given from Assumption \ref{assumption:technical-rd}. To control the right hand side of \eqref{eq:first-upper-bdd}, define the function
\begin{align}\label{eq:upper-bdd-fnc}
    U(u) := \sup\limits_{s \in [0,s_0]} \int_0^{1}\int_{\mathbb{R}^{d}} \norm{D_t\partial_u \log \rho_t^u}^2 \rho_t^{s}(dx) dt. 
\end{align}
Note that $\partial_u q_t^u = \partial_u \log \rho_t^u$. It then follows from \eqref{eq:first-upper-bdd} that for all $s \leq s_0$
\begin{align}\label{eq:second-upper-bdd}
    \int_0^{1}\norm{D_{t}q_{t}^{s}}_{L^{2}(\rho_t^{s})}^2 dt \leq s^2 \cdot \frac{1}{s}\int_0^{s} U(u)du. 
\end{align}
It now remains to show that $U(u) \to 0$ as $u \downarrow 0$. Recall the polynomial bound \eqref{eq:polynomial-dom-to-zero} from Assumption \ref{assumption:technical-rd}. In the Appendix, Proposition \ref{prop:sb-in-tangent-space} furnishes a $C > 0$ such that $\norm{v_t^0(x)} \leq C(1+\norm{x})$ for all $x \in \mathbb{R}^{d}$ and $t \in [0,1]$. By Cauchy–Schwarz, these facts combine to give a $K > 0$ and $k \geq 2$ such that
\begin{align}\label{eq:poly-upper-bdd}
    F(t,u,x):= \norm{D_{t}\partial_{u}\log\rho_{t}^{u}(x)}^2 \leq K(1+\norm{x}^k), \quad \text{for all $(t,u,x) \in [0,1]\times[0,s_0]\times \mathbb{R}^{d}$.}
\end{align}
Fix $R > 0$, and observe that
\begin{align}\label{eq:u-split}
    U(u) \leq \sup\limits_{t \in [0,1],\norm{x} \leq R} F(t,u,x) + K \sup\limits_{s \in [0,s_0]} \int_{0}^{1} \int_{\norm{x} \geq R} (1+\norm{x}^{k}) \rho_t^{s}(x)dxdt. 
\end{align}
From Lemma \ref{lem:zero-initial-velocity}, it holds for all $(t,x) \in[0,1]\times \mathbb{R}^{d}$ that $D_t\partial_u \log\rho_t^u(x) \to 0$ as $u \to 0$. Thus, by the joint continuity of $F(t,u,x)$ from Assumption \ref{assumption:technical-rd} and the compactness of the region $(t,u,x) \in [0,1] \times [0,s_0] \times \{\norm{x} \leq R\}$, it holds that
\begin{align}\label{eq:f-t-u-x-limit}
    \lim\limits_{u \downarrow 0} \left(\sup\limits_{t \in [0,1], \norm{x} \leq R} F(t,u,x)\right) = 0.
\end{align}
For the other term in \eqref{eq:u-split}, recall from Proposition \ref{prop:uniform-poincare-constant} that the $(\rho_t^{s},t \in [0,1], u \in [0,s])$ are uniformly sub-Gaussian. There then exists a $K_1 > 0$ such that $\Exp{\rho_t^{s}} \left[(1+\norm{X}^k)^{2} \right] \leq K_1^2$ for all $(t,s) \in [0,1] \times [0,s_0]$. 
Hence, by Cauchy–Schwarz
\begin{align}\label{eq:junk-term}
    \sup\limits_{s \in [0,s_0]} \int_{0}^{1} \int_{\norm{x} \geq R} (1+\norm{x}^{k}) \rho_t^{s}(x)dxdt \leq K_1  \cdot \sup\limits_{(t,s) \in [0,1]\times [0,s_0]}\left(\rho_t^{s}\left(\norm{x}\geq R\right)\right)^{1/2}.
\end{align}
Thus, in \eqref{eq:u-split} fix $R > 0$ and send $u \downarrow 0$. From \eqref{eq:f-t-u-x-limit} and \eqref{eq:junk-term} it holds that
\begin{align}\label{eq:almost-lim-sup}
    \limsup\limits_{u \downarrow 0} U(u) \leq KK_{1}\cdot \sup\limits_{(t,s) \in [0,1]\times [0,s_0]}\left(\rho_t^{s}\left(\norm{x}\geq R\right)\right)^{1/2}.
\end{align}
By the uniform sub-Gaussianity of the family of probability measures, the right hand side of \eqref{eq:almost-lim-sup} vanishes as $R \to +\infty$. 
Thus, from \eqref{eq:almost-lim-sup} it follows that $U(u) \to 0$ as $u \downarrow 0$. Returning to \eqref{eq:second-upper-bdd}, it now holds that $\int_0^{1}\norm{D_{t}q_{t}^{s}}_{L^{2}(\rho_t^{s})}^2dt$ is $o(s^2)$. This completes Step 3.

Combining the three steps, it altogether follows that
\begin{align*}
     \int_0^1 |(\rho_t^{\eps})'|^2 dt &= \Was{2}^2(e^{-f},e^{-h})+ \int_{0}^{1} \norm{v_t^\eps - v_t^0}_{L^2(\rho_t^{\eps})}^2 dt \\
     &\leq \Was{2}^2(e^{-f},e^{-h}) + C\int_0^{1} \norm{D_{t}q_{t}^{\eps}}_{L^2(\rho_t^\eps)}^2 dt = \Was{2}^2(e^{-f},e^{-h}) + o(\eps^2).
\end{align*}
\end{proof}

\section{The Torus Setting}\label{sec:torus-case}
Let $\mathbb{T}^{d} := \mathbb{R}^{d}/\mathbb{Z}^{d}$ denote the flat torus. In this compact setting, the argument presented in Sections \ref{sec:rel-ent-approx} and \ref{sec:kinetic-energy-bdds} can be established under more simply stated hypotheses. The key difference is that we must (1) define an analogue of the MLD defined in \eqref{defn:MLD} on the torus, and (2) establish the analogous approximation result to Theorem \ref{thm:mp-mld-approx-of-sb}.

\subsection{MLD Analogue on $\mathbb{T}^{d}$.}
We will fix the following setting. Let $e^{-f},e^{-h} \in \cP(\mathbb{T}^{d})$ denote the marginals and let $T: \mathbb{T}^{d} \to \mathbb{T}^{d}$ denote the quadratic cost optimal transport map from $e^{-f}$ to $e^{-h}$, i.e.\ the argmin of \eqref{defn:was2} with cost function given by $d_{\mathbb{T}^{d}}^2(\cdot,\cdot)$, the distance function on the torus. The existence of an optimal transport map in this setting follows from \cite{mccann-ot-maifold01}. Additionally, we will assume that $T$ is obtained by passing an equivariant map from $\mathbb{R}^{d} \to \mathbb{R}^{d}$, i.e.\ a function $\Phi: \mathbb{R}^{d} \to \mathbb{R}^{d}$ such that $\Phi(x+k) = \Phi(x)+k$ for all $x \in \mathbb{R}^{d}$ and $k \in \mathbb{Z}^{d}$, through the quotient. The exact construction we use is as follows. Let $\psi: \mathbb{R}^{d} \to \mathbb{R}$ be a smooth, $\mathbb{Z}^{d}$-periodic function. We set $\varphi(x) = \frac{1}{2}\norm{x}^2 + \psi(x)$, then assume that there are $m,M > 0$ such that
\begin{align}\label{eq:transport-map-reg}
    m \Id \leq \nabla^2 \varphi(x) \leq M \Id, \quad \text{ for all $x \in \mathbb{R}^{d}$.}
\end{align}
The map $T$ is then obtained by passing $\nabla \varphi$ through the quotient. 

We form the Mirror Langevin diffusion on $\mathbb{T}^{d}$ by first lifting the relevant quantities to $\mathbb{R}^{d}$ and then taking the quotient over $\mathbb{Z}^{d}$. Let $\widetilde{h}$ denote the periodic lift of $h$ to $\mathbb{R}^{d}$ and define the following SDE on $\mathbb{R}^{d}$
\begin{align}\label{eq:lifted-mld-torus}
    d\widetilde{X}_t &= -\frac{1}{2}\nabla \widetilde{h}(\nabla \varphi(\widetilde{X}_t))dt + \sqrt{\nabla^2 \varphi^*(\nabla \varphi(\widetilde{X}_t))}dB_t,
\end{align}
where $(B_t, t \geq 0)$ is Euclidean Brownian motion. Define now
\begin{align}
    X_t = \widetilde{X}_t \mod \mathbb{Z}^{d}
\end{align}
for all $t \geq 0$. For each test function $\xi \in C^{\infty}(\mathbb{T}^{d})$, take its periodic lift to $C^{\infty}(\mathbb{R}^{d})$ and apply $\Ito$'s formula to $(\widetilde{X}_t, t \geq 0)$. In this way, it holds that $(X_t, t \geq 0)$ is a $\mathbb{T}^{d}$-valued diffusion that is a solution to the martingale problem with the following generator, written in coordinates and defined for $u \in C^{\infty}(\mathbb{T}^{d})$ by
\begin{align}\label{defn:td-mld}\tag{$\mathbb{T}^{d}$-MLD}
    Lu = \frac{1}{2} \sum_{i,j=1}^{d} (\nabla^2 \varphi)^{-1}_{ij}\partial^2_{ij}u - \frac{1}{2}\langle \nabla h \circ \nabla \varphi, \nabla u\rangle.
\end{align}
We will call this process $(X_t, t \geq 0)$ the MLD on $\mathbb{T}^{d}$.

\subsection{Analogous Construction and Result}
We now construct the tangent Markov projection of the reciprocal process formed from the MLD on the torus. First, we define a flow of marginals analogous to \eqref{eq:recip-marginal-flow}. Let $X_0 \sim e^{-f}$, and let $\widetilde{X}_0$ denote a lift of $X_0$ to $\mathbb{R}^{d}$. Then, let $(\widetilde{X}_t, t \geq 0)$ denote the solution to \eqref{eq:lifted-mld-torus} with this initial data. Let $(W_t, t \geq 0)$ be an independent Euclidean Brownian motion and define, as in \eqref{eq:defn-ust},
\begin{align}\label{eq:sheet-on-torus}
    U(t,s) := (1-t)\widetilde{X}_0 + t\nabla \varphi(\widetilde{X}_{s}) + \sqrt{t(1-t)}W_s \mod \mathbb{Z}^{d}, \quad \rho_t^{s} = \mathrm{Law}(U(t,s)). 
\end{align}
We will also use $\rho_t^{s}$ to denote the density with respect to the volume measure on the torus, $\vol_{\mathbb{T}^{d}}$. A notion of tangent space for measures on manifolds exists. For a reference on optimal transport on manifolds and AC curves of probability measures on manifolds, see \cite{gigli-ot-manifold}. Define for $\mu \in \cP_2(\mathbb{T}^{d})$,
    $\Tan{\mu}:= \overline{\left\{\nabla \varphi: \varphi \in C^{\infty}(\mathbb{T}^{d})\right\}}^{L^2(\mu)}$,
and for each $\eps \geq 0$
\begin{align}\label{eq:recip-marginal-flow-td}
    (\rho_t^{\eps},v_t^{\eps}, t \in [0,1]), \quad v_t^{\eps} \in \Tan{\rho_t^{\eps}}.
\end{align}
Observe again that at $\eps = 0$ the curve $(\rho_t^0,v_t^0, t \in [0,1])$ is the McCann interpolation.

We now state the assumptions in our current setting.
\begin{assumption}[Torus Setting]\label{assumption:torus}
    Assume that $f,h \in C^{\infty}(\mathbb{T}^{d})$. The optimal transport map $T$ is then a diffeomorphism. Let $\psi: \mathbb{R}^{d} \to \mathbb{R}$ be a smooth $\mathbb{Z}^{d}$-periodic function, and let $T$ be formed by passing the gradient of the map $\varphi = \frac{1}{2}\norm{x}^2+\psi(x)$ through the quotient. Assume that there exists $m,M > 0$ such that
    \begin{align}
        m \Id \leq \nabla^2 \varphi(x) \leq M \Id, \quad \text{ for all $x \in \mathbb{R}^{d}$.}
    \end{align}
    Additionally, we insist that $T(x)$ is not in the cut locus of $x$ (\cite[page 308]{riemannian-manifolds-lee}) for each $x\in \mathbb{T}^{d}$. For each $t \in [0,1]$, define $T_t : \mathbb{T}^{d} \to \mathbb{T}^{d}$ by passing through the quotient the map $x \mapsto (1-t)x+t\nabla \varphi(x)$. 
    
    Lastly, assume that
    \begin{align}\label{eq:parabolic-regularity-td}
        (t,s,x) \mapsto \rho_t^{s}(x), \text{ defined in \eqref{eq:sheet-on-torus} is in } C^{\infty}([0,1] \times [0,+\infty) \times \mathbb{T}^{d}). 
    \end{align}
\end{assumption}
Observe that $\min(1,m)\Id \leq DT_t \leq \max(1,M)\Id$ for all $t \in [0,1]$, where $D$ denotes the Jacobian.

\begin{remark}\label{remark:td-regul}
    We remark that the key point in \eqref{eq:parabolic-regularity-td} is the regularity up to and including the boundary, i.e.\ at $t = 0,1$ and $s = 0$. Indeed, by standard parabolic regularity, smoothness is immediate over $(t,s,x) \in (0,1)\times (0,+\infty) \times \mathbb{T}^{d}$. 
\end{remark} 

\begin{remark}\label{remark:poincare-torus}
Under Assumption \ref{assumption:torus}, there are global positive upper and lower bounds on the densities $(\rho_t^{s}, t \in [0,1], s \in [0,s_0])$ for all $s_0 > 0$. The volume measure on $\mathbb{T}^{d}$ satisfies a Poincaré inequality because $\mathbb{T}^{d}$ is compact and thus $-\frac{1}{2}\Delta_{\mathbb{T}^{d}}$ has a discrete spectrum \cite[Theorem 10.13]{heat-kernel-manifold-grig09}.
By directly comparing the measures in this family with the volume measure, it holds that there is a uniform upper bound on the Poincaré constant of the above family. This provides the torus analogue of Proposition \ref{prop:uniform-poincare-constant}.
\end{remark}

We can now finally define the tangent Markov projection in our setting. Based on Proposition \ref{prop:relaxed-drift-calculation} and with the quantities in \eqref{eq:recip-marginal-flow-td}, set $\overline{\beta}_t^{\eps} := v_t^{\eps} + \frac{\eps}{2}\nabla \log \rho_t^{\eps}$. Denote by $\overline{Q}^{\eps} \in \cP(C([0,1];\mathbb{T}^{d}))$ the time-inhomogeneous Markov process with initial condition $e^{-f}$ and generator
\begin{align}
    u \in C^{\infty}([0,1] \times \mathbb{T}^{d}) \mapsto \langle \overline{\beta}_t^{\eps}(\cdot), \nabla u(t,\cdot)\rangle + \frac{1}{2}\Delta_{\mathbb{T}^{d}}u(t,\cdot).
\end{align}
The definition and results regarding the tangent Markov projection developed in Section \ref{subsec:markov-proj} remain unchanged in this manifold setting once the analogous definitions on the manifold are supplied. 

We now follow the exact same argument outlined in Section \ref{sec:sketch-of-proof}. Let $P^{\eps} \in \cP(C([0,1];\mathbb{T}^{d}))$ denote the $\eps$-dynamic $\Schro$ bridge from $e^{-f}$ to $e^{-h}$ with reference measure $R^{\eps}$. Here, $R^{\eps}$ denotes the law of temperature $\eps$ Wiener measure on $\mathbb{T}^{d}$. Importantly, $P^{\eps}$ still has a description as a time-inhomogeneous Markov process, with drift in gradient form as \eqref{eq:sb-sde}-- the semigroup is just now that of the torus Brownian motion. Thus, arguing as in Theorem \ref{thm:mp-mld-approx-of-sb},
\begin{align}\label{eq:pythag-on-torus}
    H(\overline{Q}^{\eps}|P^{\eps}) = H(\overline{Q}^{\eps}|R^{\eps})-H(P^{\eps}|R^{\eps}).
\end{align}
The same expansion about $\eps = 0$ for $\eps \mapsto H(P^{\eps}|R^{\eps})$ in \cite[Theorem 1.6]{conforti21deriv} holds on $\mathbb{T}^{d}$:
\begin{align}\label{eq:dynam-sb-exp-torus}
    H(P^{\eps}|R^{\eps}) &= \frac{1}{2\eps}\Was{2}^2(e^{-f},e^{-h})+\frac{1}{2}\left(H(e^{-f}|\vol_{\mathbb{T}^{d}})+H(e^{-h}|\vol_{\mathbb{T}^{d}})\right)+\frac{\eps}{8}\int_0^{1} I(\rho_t^0|\vol_{\mathbb{T}^{d}})dt + o(\eps). 
\end{align}
It then remains to obtain a matching expansion from the analogue of Theorem \ref{thm:rel-ent-tangent-markov-proj}: 
\begin{align}
    H(\overline{Q}^{\eps}|R^{\eps}) &= \frac{1}{2\eps} \int_0^{1} |(\rho_t^{\eps})'|^2 dt + \frac{1}{2}\left(H(e^{-f}|\vol_{\mathbb{T}^{d}})+H(e^{-h}|\vol_{\mathbb{T}^{d}})\right)+\frac{\eps}{8}\int_0^{1} I(\rho_t^\eps|\vol_{\mathbb{T}^{d}})dt.
\end{align}

This leads us to establish the following proposition. 
\begin{proposition}[Analogue to Theorems \ref{thm:mp-mld-approx-of-sb} and \ref{lem:ke-calc} on $\mathbb{T}^{d}$]\label{prop:td-approx-analogue}
    Under Assumption \ref{assumption:torus}, let $(\rho_t^{\eps}, v_t^{\eps}, t \in [0,1])$ be as defined in \eqref{eq:recip-marginal-flow-td}, i.e.\ the marginal flow of the tangent Markov projection $\overline{Q}^{\eps}$. It holds that
    \begin{align}\label{eq:mld-approx-torus}
         \int_0^{1} |(\rho_t^{\eps})'|^2 dt = \Was{2}^2(e^{-f},e^{-h}) + O(\eps^4),\quad \int_0^{1} I(\rho_t^\eps|\vol_{\mathbb{T}^{d}})dt &= \int_0^{1} I(\rho_t^0|\vol_{\mathbb{T}^{d}})dt + O(\eps).
    \end{align}
    As a result of \eqref{eq:pythag-on-torus} and \eqref{eq:dynam-sb-exp-torus}, it holds that $H(\overline{Q}^{\eps}|P^{\eps}) = O(\eps^2)$.
\end{proposition}
Notice the improved error bound in each of the approximations compared to the Euclidean case. 

\begin{proof}
    As a consequence of \eqref{eq:parabolic-regularity-td}, it follows from the Bounded Convergence Theorem that
\begin{align}
    \lim\limits_{\eps \downarrow 0}\int_{0}^{1} I(\rho_t^{\eps}|\vol_{\mathbb{T}^{d}}) dt = \int_{0}^{1}I(\rho_t^{0}|\vol_{\mathbb{T}^{d}})dt.
\end{align}

It now remains to establish the kinetic energy estimate analogous to Theorem \ref{lem:ke-calc}. Recall that the first step in this argument was the $\Ito$ calculus calculation in Lemma \ref{lem:zero-initial-velocity} that established $\left.\partial_{s} \rho_t^{s}\right|_{s=0} = 0$ for all $t \in [0,1]$ in the sense of distributions. This same result holds on the torus by nearly the same argument. Fix $\xi \in C^{\infty}(\mathbb{T}^{d})$, and let $\widetilde{\xi}$ denote its periodic lift. Apply $\Ito$'s formula on the $\mathbb{R}^{d}$-process in \eqref{eq:sheet-on-torus} to $\widetilde{\xi}$, then by repeating the same argument in Lemma \ref{lem:zero-initial-velocity}, we obtain the analogous formula to \eqref{eq:initial-velocity}:
\begin{align}\label{eq:torus-vel-0-calc}
        \left.\frac{d}{ds}\right|_{s=0}\Exp{}[\xi(U(t,s))] &= \Exp{}\left[\langle\nabla \xi(T_t(X_0)),-\frac{t}{2}\nabla f(X_0) \rangle + \frac{t}{2} \nabla^2 \xi(T_t(X_0)):DT_t(X_0) \right],
    \end{align}
where we have again used the identity $t^2 DT(x) + t(1-t)\Id = t DT_t(x)$. Next, observe that
\begin{align}\label{eq:div-form}
   DT_t(x) : \nabla^2 \xi(T_t(x)) = \nabla_{x} \cdot \left(\nabla \xi(T_{t}(x))\right).
\end{align}
By the identity in \eqref{eq:div-form}, integrate by parts the second expression in \eqref{eq:torus-vel-0-calc} to obtain that (recall $\mathbb{T}^{d}$ has no boundary) 
\begin{align*}
    \Exp{}\left[DT_t(X_0) : \nabla^2 \xi(T_t(X_0))\right] &= \int_{\mathbb{T}^{d}} \nabla_{x} \cdot \left(\nabla \xi(T_{t}(x))\right) e^{-f(x)}\vol_{\mathbb{T}^{d}}(dx)\\
    &=\Exp{}\left[\langle \nabla f(X_0), \nabla \xi(T_t(X_0))\rangle\right].
\end{align*}
Altogether then, \eqref{eq:torus-vel-0-calc} is equal to zero. This establishes the analogue of Lemma \ref{lem:zero-initial-velocity} on the torus. 

Next, we use this calculation to bound the kinetic energy as in Theorem \ref{lem:ke-calc}. Fortunately, Steps 1 and 2 in the proof of Theorem \ref{lem:ke-calc} follow identically. Set $q_{t}^{s} = \log(\rho_t^{s}/\rho_t^0)$ and $D_t = \partial_t + \langle v_t^0, \nabla (\cdot) \rangle$ as in \eqref{defn:q-s-t-diff}.  Let $C > 0$ denote the uniform upper bound on the Poincaré constants developed in Remark \ref{remark:poincare-torus}. It holds for all $\eps > 0$ small enough that
\begin{align}\label{eq:ke-norm-torus-1}
    \int_{0}^{1} \norm{\partial_t \rho_t^{\eps}}_{\dot{H}^{-1}(\rho_t^{\eps})}^2 dt &\leq \Was{2}^{2}(e^{-f},e^{-h}) + C \int_{0}^{1}\norm{D_{t}q_t^{\eps}}^2_{L^2(\rho_t^{\eps})}dt. 
\end{align}
From \eqref{eq:parabolic-regularity-td}, $D_{t}q_{t}^{\eps}$ exists classically and is equal to $0$ at $\eps = 0$ as $q_t^{0}$ identically vanishes. As in Remark \ref{remark:poincare-torus}, for all small enough $\eps$ the families $(\rho_t^{\eps}, t \geq 0)$ possess a uniform positive lower bound. This implies that $\left.\partial_{s}q_{t}^{s}\right|_{s = 0}$ identically vanishes. Applying the Fundamental Theorem of Calculus and using \eqref{eq:parabolic-regularity-td} to justify the interchanging of derivatives gives for each $t \in [0,1]$
\begin{align}\label{eq:fotc-torus}
    D_t q_t^\eps = D_t q_t^\eps - D_t q_t^{0} &= \int_{0}^{\eps} D_t \partial_{s}q_{t}^{s}ds = \int_0^{\eps}\int_0^{s}  D_t \partial_{u}^2q_t^{u}duds.
\end{align}
Fix some $\eps_0 > 0$, by \eqref{eq:parabolic-regularity-td} and compactness there is some $K > 0$ such that $\abs{D_t \partial_{u}^2q_t^{u}(x)} \leq K$ for all $(t,u,x) \in [0,1] \times [0,\eps_0] \times \mathbb{T}^{d}$. From \eqref{eq:ke-norm-torus-1} and \eqref{eq:fotc-torus} it then follows that
\begin{align}
    \int_0^{1} |(\rho_t^{\eps})'|^2dt = \int_{0}^{1} \norm{\partial_t \rho_t^{\eps}}^2_{\dot{H}^{-1}(\rho_t^{\eps})}dt &\leq \Was{2}^2(e^{-f},e^{-h}) + CK^2 \cdot \frac{\eps^4}{4}.
\end{align}
This establishes \eqref{eq:mld-approx-torus}. 
\end{proof}

\section{Appendix}\label{sec:appendix}
Here we collect calculations involving the Mirror Langevin diffusion and McCann interpolation, recall the notation in \eqref{eq:mccann-interp-notation}. 

First, we define our notation for the divergence of a matrix.  For a matrix $A(x) = (A_{ij}(x))$, we write the $i$th entry of the vector $\nabla_x \cdot A(x)$ as 
\begin{align}\label{defn:matrix-divergence}
    \left(\nabla_{x} \cdot A(x)\right)_{i} = \sum_{j} \partial_{j}A_{ij}(x).
\end{align}
This notion of matrix divergence is helpful because it will facilitate several later integration by parts calculations. Importantly, we note that the generator of the dual Mirror Langevin diffusion defined in \eqref{defn:dMLD} has the following divergence form 
\begin{align}\label{eq:div-form-mld-dual}
    u \in C^2_c(\mathbb{R}^{d}) \mapsto \frac{1}{2e^{-h(y)}}\nabla_{y} \cdot \left(e^{-h(y)}\left(\nabla^2 \varphi^*(y)\right)^{-1}\nabla u(y)\right)
\end{align}

\begin{proposition}[Matrix Divergence Identity]\label{prop:matrix-divergence-identity}
    Let $\varphi \in C^{3}(\mathbb{R}^{d})$ be a strictly convex function. This implies that $\nabla \varphi^* = (\nabla \varphi)^{-1}$ and thus $\nabla^2 \varphi(x) = \left(\nabla^2 \varphi^*(\nabla \varphi(x))\right)^{-1}$. It holds that
    \begin{align}\label{eq:matrix-div-identity}
        \nabla_{x} \cdot \left(\nabla^2 \varphi(x)\right)^{-2} &= \left(\nabla^2 \varphi(x)\right)^{-2}\nabla \log \det \nabla^2 \varphi^*(\nabla \varphi(x)) + \Delta \nabla \varphi^* (\nabla \varphi(x)).
    \end{align} 

    As a consequence, if $e^{-\alpha}, e^{-\beta} \in \cP_2(\mathbb{R}^{d})$ are such that $\alpha,\beta \in C^1(\mathbb{R}^{d})$ and $(\nabla \varphi)_{\#}e^{-\alpha} = e^{-\beta}$, it holds that
    \begin{align}\label{eq:annoying-cov-identity}
        \left(\nabla^2 \varphi(x) \right)^{-1}\nabla \beta(\nabla \varphi(x)) &= \left(\nabla^2 \varphi(x)\right)^{-2}\nabla \alpha(x)-\nabla_x \cdot \left(\nabla^2 \varphi(x)\right)^{-2} + \Delta \nabla \varphi^*(\nabla \varphi(x)). 
    \end{align}
\end{proposition}
\begin{proof}
    We first expand the quantity on the left hand size of \eqref{eq:matrix-div-identity}. The $i$th entry is equal to
    \begin{align*}
        &\left(\nabla_{x} \cdot \left(\nabla^2 \varphi(x)\right)^{-2}\right)_{i} = \sum_{j} \partial_{j} \left(\sum_{k} \partial^2_{ik}\varphi^*(\nabla \varphi(x))\partial^2_{kj}\varphi^*(\nabla \varphi(x))\right) \\
        &= \sum_{jk\ell} \partial^{3}_{ik\ell}\varphi^*(\nabla \varphi(x))\partial^2_{\ell j}\varphi(x)\partial^2_{kj}\varphi^*(\nabla\varphi(x))+ \sum_{jk\ell} \partial^2_{ik}\varphi^*(\nabla \varphi(x))\partial^{3}_{kj\ell}\varphi^*(\nabla \varphi(x))\partial^2_{\ell j}\varphi(x).
    \end{align*}
    We now analyze the two summands in the above expression. First, since $\nabla^2 \varphi(x) = \left(\nabla^2 \varphi^*(\nabla \varphi(x))\right)^{-1}$ and the two matrices are symmetric,
    \begin{align*}
        \sum_{jk\ell} \partial^{3}_{ik\ell}\varphi^*(\nabla \varphi(x))\partial^2_{\ell j}\varphi(x)\partial^2_{kj}\varphi^*(\nabla\varphi(x)) &= \sum_{k}  \partial^{3}_{ikk}\varphi^*(\nabla \varphi(x)) = \nabla \Delta \varphi^*(\nabla \varphi(x)). 
    \end{align*}
    Hence, \eqref{eq:matrix-div-identity} holds once we establish that
    \begin{align}\label{eq:log-det-ident-summand}
        (\nabla^2 \varphi(x))^{-2}\nabla \log \det \nabla^2 \varphi^*(\nabla \varphi(x)) &= \sum_{jk\ell} \partial^2_{ik}\varphi^*(\nabla \varphi(x))\partial^{3}_{kj\ell}\varphi^*(\nabla \varphi(x))\partial^2_{\ell j}\varphi(x).
    \end{align}
    To see this, we recall that for a matrix valued function, $\frac{d}{dx} \log \det A(x) = \sum_{ij} A_{ij}^{-1}(x) \partial_{x} A_{ij}(x)$. Hence, the $j$th entry of $\nabla \log \det \nabla^2 \varphi^*(\nabla \varphi(x))$ is 
    \begin{align*}
        \left(\nabla \log \det \nabla^2 \varphi(x)\right)_{j} &= \sum_{k\ell u} \partial^2_{k\ell}\varphi(x)\partial^3_{k\ell u}\varphi^*(\nabla \varphi(x))\partial^2_{uj}\varphi(x). 
    \end{align*}
    From this, the $i$th entry of $\left(\nabla^2 \varphi(x)\right)^{-2}\nabla \log \det \nabla^2 \varphi^*(\nabla \varphi(x))$ is equal to
    \begin{align*}
        &\left(\left(\nabla^2 \varphi(x)\right)^{-2}\nabla \log \det \nabla^2 \varphi^*(\nabla \varphi(x))\right)_{i} \\
        &= \sum_{vk\ell u j} \partial_{iv}^2\varphi^*(\nabla \varphi(x))\partial_{vj}^2\varphi^*(\nabla \varphi(x))\cdot \partial^2_{k\ell}\varphi(x)\partial^3_{k\ell u}\varphi^*(\nabla \varphi(x))\partial^2_{uj}\varphi(x) \\
        &= \sum_{vk\ell u}\partial_{iv}^2\varphi^*(\nabla \varphi(x))\cdot \partial^2_{k\ell}\varphi(x)\partial^3_{k\ell u}\varphi^*(\nabla \varphi(x))\cdot \mathbbm{1}(u = v) \\
        &= \sum_{klu} \partial^2_{iu}\varphi^*(\nabla \varphi(x))\partial^2_{k\ell}\varphi(x)\partial^3_{k\ell u}\varphi^*(\nabla \varphi(x))
    \end{align*}
    This establishes \eqref{eq:log-det-ident-summand} and thus the desired matrix divergence identity in \eqref{eq:matrix-div-identity}.

    We now see how \eqref{eq:annoying-cov-identity} follows. By the change of variables formula, it holds that
    \begin{align}\label{eq:cov-formula}
        \beta(\nabla \varphi(x)) = \alpha(x) + \log \det \nabla^2 \varphi(x).
    \end{align}
    Next, take the gradient in $x$ on both sides of \eqref{eq:cov-formula}. This results in
    \begin{align}\label{eq:cov-formula-step2}
        \nabla^2 \varphi(x) \nabla \beta (\nabla \varphi(x)) = \nabla \alpha(x) + \nabla \log \det \nabla^2 \varphi(x).
    \end{align}
    Multiply both sides of \eqref{eq:cov-formula-step2} by $(\nabla^{2} \varphi(x))^{-2}$ to get
    \begin{align}\label{eq:cov-formula-step3}
        \left(\nabla^2 \varphi(x) \right)^{-1}\nabla \beta(\nabla \varphi(x)) &= \left(\nabla^2 \varphi(x)\right)^{-2}\nabla \alpha(x) + \left(\nabla^2 \varphi(x) \right)^{-2}\nabla \log \det \nabla^2 \varphi(x).
    \end{align}
    Now, use the identity in \eqref{eq:matrix-div-identity} to replace the $\left(\nabla^2 \varphi(x) \right)^{-2}\nabla \log \det \nabla^2 \varphi(x)$ term. 
\end{proof}

\begin{proposition}[Step 1 in Proposition \ref{prop:d-dt-matrix-identity}]\label{prop:noisy-mccann-drift-into-divergence}
Retain the setting and notation of Proposition \ref{prop:d-dt-matrix-identity}. The drift in \eqref{eq:pushforward-drift-mccann} can be written in the following divergence form
    \begin{align}\label{eq:simple-transformed-drift-apx}
        \gamma_t^{\eps}(z) = -\frac{\eps}{2}\left(A_{t}(z)\nabla h(z) - \nabla_z \cdot A_{t}(z)\right),
    \end{align}
    where  $A_{t}(z) := \left(\nabla^2 \varphi_{1 \to t}(z)\right)^{-2}$.
\end{proposition}
\begin{proof}[Proof of Proposition \ref{prop:noisy-mccann-drift-into-divergence}]
    We proceed by splitting $\gamma_t^{\eps} = \gamma_t^{(1)}+\gamma_t^{(2)}$, where
    \begin{align}\label{eq:gamma-t-1}
        \gamma_t^{(1)}(z) &:= \left(\nabla^2 \varphi_{1 \to t}(z)\right)^{-1} v_t^0(\nabla \varphi_{1 \to t}(z))+\partial_{t}[\nabla \varphi_{t \to 1}](\nabla \varphi_{1 \to t}(z)),
    \end{align}
    \begin{align}
        \gamma_t^{(2)}(z) &= \frac{\eps}{2}\left(\left(\nabla^2 \varphi_{1 \to t}(z)\right)^{-1}\nabla \log \rho_t^0(\nabla \varphi_{1 \to t}(z))+\Delta \nabla \varphi_{t \to 1}(\nabla \varphi_{1 \to t}(z))\right). 
    \end{align}
    First, we show that $\gamma_t^{(1)}(z) = 0$. To see this, recall that $v_t^0(x) = \nabla \varphi_{t \to 1}(x)-\nabla \varphi_{t \to 0}(x)$ and thus 
    \begin{align}\label{eq:shift-vel}
        v_t^0(\nabla \varphi_{1 \to t}(z)) = z - \nabla \varphi_{1 \to 0}(z).
    \end{align}
    Next, recall that $\nabla \varphi_{1 \to t}(z) = tz + (1-t)\nabla \varphi_{1 \to 0}(z)$. By the chain rule and \eqref{eq:shift-vel}, 
\begin{align*}
    0 &= \partial_{t}\left(\nabla \varphi_{t \to 1} \circ \nabla \varphi_{1 \to t}\right)(z) \\
    &= \nabla^2 \varphi_{t \to 1}(\nabla \varphi_{1 \to t}(z)) \cdot \left(\partial_{t} \nabla \varphi_{1 \to t}\right)(z) + \partial_{t}\left[\nabla \varphi_{t \to 1}\right](\nabla \varphi_{1 \to t}(z)) \\
    &= \left(\nabla^2 \varphi_{1 \to t}(z)\right)^{-1} (z-\nabla \varphi_{1 \to 0}(z)) + \partial_{t}\left[\nabla \varphi_{t \to 1}\right](\nabla \varphi_{1 \to t}(z)) = \gamma_{t}^{(1)}(z). 
\end{align*}
Thus, we must now show that $\gamma_t^{(2)}$ is equal to the right hand side of \eqref{eq:simple-transformed-drift-apx}. But this is exactly the calculation in Proposition \ref{prop:matrix-divergence-identity}, specifically \eqref{eq:annoying-cov-identity} where $e^{-\beta} = \rho_{t}^{0}$, $e^{-\alpha} = e^{-h}$, and the optimal transport map is $\nabla \varphi_{1 \to t}$. 
\end{proof}

\begin{proposition}\label{prop:sqrd-inv-hessian-iden}
Let $\varphi \in C^{2}(\mathbb{R}^{d})$ be a strictly convex function, then
    \begin{align}
        \int_0^{1} \left(\nabla^2 \varphi_{1 \to t}(y)\right)^{-2} dt = \left(\nabla^2 \varphi_{1 \to 0}(y)\right)^{-1}.
    \end{align}
\end{proposition}
\begin{proof}[Proof of Proposition \ref{prop:sqrd-inv-hessian-iden}]
    First, recall that $\nabla^2 \varphi_{1 \to t}(y) = t\Id + (1-t)\nabla^2 \varphi_{1 \to 0}(y)$. Next, diagonalize $\nabla^2 \varphi_{1 \to 0}(y)$ and observe that the computation reduces to the following one-dimensional integral: for $a > 0$
\begin{align*}
    \int_{0}^{1} \frac{1}{((1-t)a+t)^2} dt = \frac{1}{a}. 
\end{align*}
\end{proof}

\subsection{Proof of Proposition \ref{prop:uniform-poincare-constant}}
In this subsection, we prove Proposition \ref{prop:uniform-poincare-constant}. That is, we will establish the well-posedness of the MLD SDE given in \eqref{defn:MLD}, the well-posedness of the $\Schro$ bridge, and various other properties. Importantly, the majority of this section is devoted to establishing a uniform bound on the Poincaré constants for the curve of measures in \eqref{eq:recip-marginal-flow}.

\textbf{Basic Facts and Estimates.} By the uniform positive lower bound on the Hessians of $f$ and $h$, the measures $e^{-f}, e^{-h} \in \cP_2(\mathbb{R}^{d})$ have sub-Gaussian tails \cite[Theorem 5.2.15]{vershynin-hdp}. Recall the reciprocal surface $(U(t,s),t \in [0,1], s \geq 0)$ as defined in \eqref{defn:ust-approx}. As each $U(t,s)$ is the sum of sub-Gaussian random variables and the sub-Gaussian norm $\norm{\cdot}_{\psi_{2}}$ in \cite[Definition 2.5.6]{vershynin-hdp} is a norm, it holds that $\norm{U(t,s)}_{\psi_{2}}$ is bounded over $(t,s) \in [0,1] \times [0,\eps_0]$ for all $\eps_0 > 0$. Thus, for moments of each order there is a uniform bound across the family of measures. Similarly, there is a uniform tail bound.
It then follows that $\Was{p}(\rho_t^{\eps},\rho_t^0) \to 0$ as $\eps \downarrow 0$ for all $t \in [0,1]$. This establishes point (3) in the statement of Proposition \ref{prop:uniform-poincare-constant}.

Next, as $\rho_t^0 = (\nabla \varphi_{0 \to t})_{\#}e^{-f}$, the change of variables formula gives that
\begin{align}\label{eq:mccann-cov}
    -\log \rho_t^{0}(\nabla \varphi_{0 \to t}(x)) = f(x) + \log \det \nabla^2 \varphi_{0 \to t}(x).
\end{align}
By the uniform global bounds on the second, third, and fourth derivatives of $\varphi$, it follows from differentiating \eqref{eq:mccann-cov} twice that $-\log \rho_t^0$ has uniform bounds on its Hessian (although these bounds are not necessarily positive). From these observations, it follows that $\abs{-\log \rho_t^0} \leq C(1+\norm{x}^2)$ for a constant $C >0$ that is uniform in $t \in [0,1]$. Similarly, $\norm{\nabla \log \rho_t^0} \leq K(1+\norm{x})$ for $K > 0$ that is also uniform in $t \in [0,1]$. Altogether, these facts imply that $\Ent(\rho_t^0)$ and  $I(\rho_t^0)$ are finite for all $t \in [0,1]$ and $\int_0^{1} I(\rho_t^0) dt < +\infty$ as well. 

As $\Ent(e^{-f})$ and $\Ent(e^{-h})$ are finite, it follows from \cite[Proposition 2.8]{schroLeonard13} that the $\eps$-$\Schro$ bridge joining $e^{-f}$ to $e^{-h}$ exists for all $\eps > 0$. This establishes point (2) in the statement of Proposition \ref{prop:uniform-poincare-constant}.

As for the well-posedness of the MLD SDE, observe that since $\nabla^2 f(x) \geq \Lambda_1 \Id > 0$ for all $x \in \mathbb{R}^{d}$, it holds for all $x \in \mathbb{R}^{d}$ that
\begin{align*}
    x \cdot \nabla f(x) \geq x \cdot \nabla f(0) + \Lambda_1 \norm{x}^2.
\end{align*}
Thus, we see that $\inf\limits_{x \in \mathbb{R}^{d}}x \cdot \nabla f(x) > -\infty$, so from \cite[Propositions 4.1, 4.5]{klartag-log-concave14} there exists a unique weak solution to \eqref{defn:MLD} with a.s.\ infinite explosion time.

As for the claim that $H(\ell_{\eps}|r_{\eps}) < +\infty$ for each $\eps > 0$, this is established in \cite[Theorem 4]{MP25}. To be more precise, see the application of global heat kernel bounds under $\mathrm{CD}(\kappa,\infty)$ (defined below) in the proof of \cite[Theorem 4]{MP25}. Hence, point (1) in the statement of Proposition \ref{prop:uniform-poincare-constant} is established aside from the uniform Poincaré constant. 

\textbf{Fisher Information Proofs.}
Next, we prove point (4) in  Proposition \ref{prop:uniform-poincare-constant}, that is,
\begin{align}\label{eq:fi-limits-appendix}
    \lim\limits_{\eps \downarrow 0}\int_0^{1} I(\rho_t^{\eps}) dt = \int_0^{1} I(\rho_t^0) dt, \quad \lim\limits_{\eps \downarrow 0} \int_0^{1} I(\rho_t^{\eps}|\rho_t^0)dt = 0. 
\end{align}
As $\rho_t^{\eps}$ converges weakly to $\rho_t^0$ as $\eps \downarrow 0$ for each $t \in [0,1]$, it follows from the lowersemicontinuity of Fisher information with respect to weak convergence \cite[Proposition 13.2]{bobkov-fisher-22} and Fatou's lemma that 
    \begin{align}\label{eq:fi-lsc}
        \liminf\limits_{\eps \downarrow 0} \int_0^{1} I(\rho_t^{\eps}) dt \geq \int_0^{1} I(\rho_t^0)dt.
    \end{align}
Hence, it remains to justify that the complementary $\limsup$ is furnished under the conditions in Assumption \ref{assumption:technical-rd}.

From \eqref{eq:polynomial-dom-to-zero} in Assumption \ref{assumption:technical-rd}, there are $C > 0$ and $k \geq 1$ such that for all $\eps > 0$ small enough and $t \in [0,1]$ it holds that $\norm{\nabla \log \rho_t^{\eps}(x)}^2 \leq K(1+\norm{x}^k)$.
We have shown that $\Was{p}(\rho_t^{\eps},\rho_t^0) \to 0$ as $\eps \downarrow 0$ for all $p \geq 1$, so take $p = k$. Thus, 
\begin{align}\label{eq:moment-limit}
    \lim\limits_{\eps \downarrow 0}\int_{0}^{1} K(1+\norm{x}^k) \rho_t^{\eps}(x)dxdt = \int_0^{1}K(1+\norm{x}^k)\rho_t^0(dx)dt. 
\end{align}
Since the integrand on the left hand side of \eqref{eq:moment-limit} provides a pointwise upper bound to $\norm{\nabla \log \rho_t^{\eps}(x)}^2 \rho_t^{\eps}$, it follows from the generalized Dominated Convergence Theorem that $\int_0^{1}I(\rho_t^{\eps})dt \to \int_0^{1} I(\rho_t^0)dt$ as $\eps \downarrow 0$.

Next, we establish the second limit in \eqref{eq:fi-limits-appendix}. To see this, we note that
    \begin{align}\label{eq:rel-fi-split}
         I(\rho_t^{\eps}|\rho_t^0) &= \Exp{\rho_t^{\eps}}\norm{\nabla \log \rho_t^{\eps}-\nabla \log \rho_t^0}^2 = I(\rho_t^\eps) + 2 \int_{\mathbb{R}^{d}} \Delta \log \rho_{t}^0 d\rho_t^{\eps} + \int_{\mathbb{R}^{d}} \norm{\nabla \log \rho_t^0}^2 \rho_t^{\eps}(dx),
    \end{align}
    where the Laplacian term follows from applying integration by parts to the cross term. Under Assumption \ref{assumption:standard-assumptions}, it holds that $\Delta \log \rho_t^0$ is globally bounded and $\norm{\nabla \log \rho_t^0}^2$ has a global quadratic bound. Moreover, these bounds can be chosen to hold over all $t \in [0,1]$. As $\Was{2}(\rho_t^{\eps},\rho_t^0) \to 0$ as $\eps \downarrow 0$, it follows from generalized Dominated Convergence that 
    \begin{align*}
        \lim\limits_{\eps \downarrow 0}\int_0^{1}\int_{\mathbb{R}^{d}} \Delta \log \rho_{t}^0 d\rho_t^{\eps}dt &= \int_0^{1}  \int_{\mathbb{R}^{d}} \Delta \log \rho_{t}^0 d\rho_t^{0}dt, \\
        \lim\limits_{\eps \downarrow 0}\int_0^{1}\int_{\mathbb{R}^{d}} \norm{\nabla \log \rho_t^0}^2 d\rho_t^{\eps}dt &= \int_0^{1}  \int_{\mathbb{R}^{d}} \norm{\nabla \log \rho_t^0}^2 d\rho_t^{0}dt
    \end{align*}
    Recall that for a measure $\rho$, integration by parts gives $I(\rho) = -\int_{\mathbb{R}^{d}} (\Delta \log \rho) d\rho$. Thus, the above limits combine with \eqref{eq:rel-fi-split} to establish the second limit in \eqref{eq:fi-limits-appendix}. 

\textbf{Uniform Poincaré Constant.}
The rest of this subsection will be concerned with the proof of \eqref{eq:uniform-poincare-constant}, that is, the uniform Poincaré constant under Assumption \ref{assumption:standard-assumptions}. To do so, we will quickly introduce some relevant technical notions. Let $(M,g)$ denote a Riemannian manifold. We use $\langle \cdot,\cdot\rangle_{g}$ to denote the induced inner product on the tangent space $T_{x}M$ for $x \in M$, $\ricci_x$ denote the Ricci tensor at $x$, and $\nabla_{g}f$ to denote the gradient of differentiable functions $f: M \to \mathbb{R}$. We say that $\mu \in \cP(M)$ satisfies a (tight) \textbf{Logarithmic Sobolev Inequality} (LSI) with constant $C > 0$ if for all $f:M \to \mathbb{R}$ in the domain of the Dirichlet form associated with $\mu$,
\begin{align}\label{defn:lsi}
    \Ent_{\mu}(f^2) := \int_{M} f^2 \log f^2 d\mu - \left(\int_{M} f^2 d\mu\right)\log\left(\int_{M} f^2 d\mu\right) \leq 2C\int_{M} \norm{\nabla_g f}_{g}^{2} d\mu.
\end{align}
We will refer to the LSI constant of $\mu$ as the smallest constant $C > 0$ for which \eqref{defn:lsi} holds. 

The key technical notion we need is that of the $\mathrm{CD}(\kappa,\infty)$ condition. Let $\mu \in \cP(M)$ be such that $d\mu/d\vol = e^{-U}$, for some $U \in C^{2}(M)$, where $\vol$ denotes the volume measure on $(M,g)$. 
\begin{definition}[Curvature Dimension Condition]
    For $\kappa \in \mathbb{R}$, the triplet $(M,g,e^{-U}d\vol)$ satisfies
$\text{CD}(\kappa,\infty)$ if the following inequality holds for all $x \in M$ 
\begin{align}\label{eq:cd-defn}\tag{CD($\kappa,\infty$)}
    \ricci + \mathrm{Hess}(U) \geq \kappa g. 
\end{align}
\end{definition}
Let $(P_{t}, t \geq 0)$ denote the semigroup of the diffusion process on $M$ with generator equal to 
\begin{align}
    u \in C_{c}^{\infty}(M) \mapsto - \frac{1}{2} \langle \nabla_g u, \nabla_g U \rangle_{g} + \frac{1}{2}\Delta_{g}u. 
\end{align}
Note that $e^{-U}d\vol$ is the invariant measure of this diffusion process. In this notation, under $\mathrm{CD}(\kappa,\infty)$ a so-called \textbf{local LSI} holds (see \cite[Section 5.7]{bgl-markov}). That is, for all $u: M \to (0,+\infty)$ continuously differentiable and $t > 0$
\begin{align}\label{defn:local-lsi}
    P_{t}(u\log u) - (P_{t}u)\log(P_t u) \leq \frac{1}{2}\left(\int_0^{t}\exp(-\kappa s)ds\right)P_{t} \left(\frac{\norm{\nabla_g u}_{g}^2}{u}\right)
\end{align}

For the rest of this subsection, we will let $(M,g)$ denote the \textbf{Hessian manifold}. That is, the $C^{3}$ manifold obtained by setting $M = \mathbb{R}^{d}$ and $g = \nabla^2 \varphi$. We will identify $M$ with $\mathbb{R}^{d}$ and take care to distinguish manifold quantities from Euclidean quantities. The MLD as defined in \eqref{defn:MLD} can be equivalently written as a diffusion on the Hessian manifold, see \cite[Corollary 1]{MP25}. Let $\mu \in \cP(M)$ denote the stationary measure of the MLD treated as a manifold diffusion. By \cite[Theorem 4.3]{kolesnikov-hessian-metric}, $(M,\nabla^2 \varphi,\mu)$ satisfies $\mathrm{CD}(C,\infty)$ if the following inequality in terms of the \textbf{Euclidean} derivatives of $f,h,\nabla \varphi$ holds for all $x \in \mathbb{R}^{d}$
\begin{align}\label{eq:koles-4-3}
    (\nabla^2 \varphi(x))^{-1/2}\nabla^2 f(x) (\nabla^2 \varphi(x))^{-1/2} + (\nabla^2 \varphi(x))^{1/2}\nabla^2 h(\nabla \varphi(x)) (\nabla^2 \varphi(x))^{1/2} \geq 2C \Id.
\end{align}
From the uniform positive upper and lower bounds on the Euclidean Hessians of $f,h,\varphi$ in Assumption \ref{assumption:standard-assumptions}, \cite[Theorem 4.3]{kolesnikov-hessian-metric} gives the following remark. 
\begin{remark}\label{remark:mld-curv-dim}
    Under Assumption \ref{assumption:standard-assumptions}, the MLD as a diffusion on the Hessian manifold satisfies $\mathrm{CD}(\kappa,\infty)$ for $\kappa > 0$.
\end{remark}

\textbf{Proof of Uniform Poincaré Constant in \eqref{eq:uniform-poincare-constant}.} This proof follows the LSI mixing argument of \cite[Lemma 4.1]{conforti-weak-semi24}. The only modification is in Step 1 of the proof of \cite[Lemma 4.1]{conforti-weak-semi24}, in which the key is to express the local LSI that the MLD satisfies on the Hessian manifold in Euclidean quantities. 

Indeed, fix $u: \mathbb{R}^{d} \to (0,+\infty)$ continuously differentiable. It then follows that $u$ is positive and continuously differentiable as a function on the Hessian manifold. Moreover, as there exists $\alpha, \beta > 0$ such that $\alpha \Id \leq \nabla^2 \varphi(x) \leq \beta \Id$ for all $x \in \mathbb{R}^{d}$, we see that $\norm{\nabla_g u}^2_g \leq \alpha^{-1}\norm{\nabla u}^2$. By Remark \ref{remark:mld-curv-dim} and \eqref{defn:local-lsi}, it holds that
\begin{align}\label{eq:local-lsi-like-mld}
    P_{t}(u\log u) - (P_{t}u)\log(P_t u) \leq \frac{\alpha^{-1}}{2}\left(\int_0^{t}\exp(-\kappa s)ds\right)P_{t} \left(\frac{\norm{\nabla u}^2}{u}\right).
\end{align}
Here, $(P_t, t \geq 0)$ is the semigroup of the primal MLD. 
That is, we establish that the MLD semigroup satisfies a local LSI-type inequality on Euclidean space. 

We now have all the necessary components. Recall the definitions of $(X_t, t \geq 0)$, $(W_t, t \geq 0)$, and $(U(t,s), t \in [0,1], s \geq 0)$ from \eqref{eq:defn-ust}. We will refer to the random variables and their laws satisfying LSI interchangeably, and we are now exclusively working on Euclidean space. First, as $X_0 \sim e^{-f}$ and the Hessian of $f$ is lower bounded by a positive constant, $e^{-f}$ satisfies an LSI inequality \cite[Corollary 5.7.2]{bgl-markov}. From \eqref{eq:local-lsi-like-mld}, it follows that $X_{\eps}|X_{0} = x$ satisfies a local LSI-like inequality. Notice that the constant on the RHS of \eqref{eq:local-lsi-like-mld} is bounded over $\eps \in (0,\eps_0]$. By repeating verbatim the mixing argument in \cite[Lemma 4.1]{conforti-weak-semi24}, $(X_0,X_{\eps})$ satisfies LSI with a uniformly bounded constant over $\eps \in (0,\eps_0]$. By tensorization \cite[Proposition 5.2.7]{bgl-markov}, $(X_0,X_{\eps},W_{\eps})$ has a uniform bound on its LSI constant over $\eps \in (0,\eps_0]$. Recall the notation $z^* =\nabla \varphi(z)$. As $(x,y,z) \mapsto (1-t)x + ty^* + \sqrt{t(1-t)}z$ is Lipschitz (in the Euclidean sense) with bounded Lipschitz constant over $t \in [0,1]$, it follows from \cite[Proposition 5.4.3]{bgl-markov} that $\rho_t^{s} = \mathrm{Law}(U(t,s))$ satisfies LSI with a uniformly bounded constant. As the LSI inequality implies the Poincar\'{e} inequality, this completes the argument.

\subsection{Additional Lemmas}
We start with the following general lemma. We now let $(P_t, t \geq 0)$ denote the standard Brownian motion semigroup and $(p_t(\cdot,\cdot), t > 0)$ its transition densities. 

\begin{proposition}\label{prop:poly-bdd-nabla-log-semigrp}
    Let $f: \mathbb{R}^{d} \to (0,+\infty)$ be such that there is a $\tau > 0$ with $P_{\tau}f(0) < +\infty$. Suppose that there is $C >0$ and integer $k \geq 0$ such that $\norm{\nabla \log f(x)} \leq C(1+\norm{x}^{k})$ for all $x \in \mathbb{R}^{d}$. For each $t_0 \in (0,\tau)$, there exists $K > 0$ such that
    \begin{align}\label{eq:heat-kernel-bdd}
        \norm{\nabla \log P_{t}f(x)} \leq K(1+\norm{x}^k) \text{ for all $x \in \mathbb{R}^{d}$,  $t \in [0,t_0/2)$.}
    \end{align}
\end{proposition}
\begin{proof}
    Fix $x \in \mathbb{R}^{d}$ and $t \in [0,t_0/2)$. Define the probability measure
    \begin{align}\label{eq:heat-measure-bayes}
        Q_{t,x}(dy) := \frac{1}{P_{t}f(x)}p_{t}(x-y)f(y)dy.
    \end{align}
    Under the assumption on $\nabla \log f$ and that $P_{\tau}f(0) <+\infty$, we see that
    \begin{align}\label{eq:nabla-log-semigrp}
        \norm{\nabla \log P_t f(x)} \leq \int_{\mathbb{R}^{d}} \norm{\nabla \log f(y)}Q_{t,x}(dy) \leq C\left(1+\int_{\mathbb{R}^{d}} \norm{y}^{k} Q_{t,x}(dy)\right). 
    \end{align}
    Let $R > 0$ be large enough such that 
    \begin{align}
        m := \int_{B(0,R)} f(y) dy > 0.
    \end{align}
    Then we see that
    \begin{align}\label{eq:semi-lbb}
        P_t f(x) \geq \int_{B(0,R)} p_{t}(x,y)f(y)dy \geq (2\pi t)^{-d/2}\exp\left(-\frac{1}{2t}\left(\norm{x}+R\right)^2\right)m.
    \end{align}
    Next, let $A > 2$ be large enough that
    \begin{align}\label{eqn:a-for-integ}
        a := 1-\frac{2}{A} > \frac{t_0}{\tau}.
    \end{align}
    Suppose that $\norm{y} > A(\norm{x}+R+1)$, then since $A > 1$ it holds that
    \begin{align*}
        \norm{y-x} \geq \norm{y}-(\norm{x} +R) > 0.
    \end{align*}
    Altogether then,
    \begin{align*}
        \norm{x-y}^2 - \left(\norm{x}+R\right)^2 \geq  \left(\norm{y}-\left(\norm{x} +R\right)\right)^2 -  (\norm{x}+R)^2  \geq \left(1-\frac{2}{A}\right)\norm{y}^2 = a\norm{y}^2. 
    \end{align*}
    Let $\Omega_{x} := \left\{\norm{y} > A(\norm{x}+R+1)\right\}$. We then compute that
    \begin{align}\label{eq:omega_x-numer}
        \int_{\Omega_{x}} \norm{y}^k p_{t}(x,y)f(y)dy \leq (2\pi t)^{-d/2}\exp\left(-\frac{1}{2t}\left(\norm{x}+R\right)^2\right)\int_{\mathbb{R}^{d}} \norm{y}^{k}\exp\left(-\frac{a}{2t}\norm{y}^{2}\right)f(y) dy.
    \end{align}
    Observe that the integral rewrites as
    \begin{align}\label{eq:split-exp}
        \int_{\mathbb{R}^{d}} \norm{y}^{k}\exp\left(\frac{1}{2\tau}\left(1-\frac{a\tau}{t}\right)\norm{y}^{2}\right)\exp\left(-\frac{1}{2\tau}\norm{y}^2\right)f(y) dy
    \end{align}
    By the choice of $a$ in \eqref{eqn:a-for-integ} and the fact that $0 < t < t_0/2$, observe that
    \begin{align*}
        1-\frac{a\tau}{t} \leq 1 - \frac{t_0}{t} \leq -1.
    \end{align*}
    Next, let $M = \sup\limits_{y \in \mathbb{R}^{d}} \norm{y}^{k}\exp\left(-\frac{1}{2\tau}\norm{y}^2\right)$. 
    By \eqref{eq:semi-lbb}, \eqref{eq:omega_x-numer}, and \eqref{eq:split-exp}, 
    \begin{align}\label{eq:bdd-1}
        \int_{\Omega_{x}} \norm{y}^{k} Q_{t,x}(dy) \leq \frac{1}{m} \int_{\mathbb{R}^{d}} \norm{y}^{k}\exp\left(-\frac{a}{2t}\norm{y}^{2}\right)f(y) dy \leq \frac{M}{m}P_{\tau}f(0). 
    \end{align}
    Observe that this upper bound is finite and does not depend on the choice of $x \in \mathbb{R}^{d}$. Moreover, it is uniform over all $t \in (0,t_0/2)$. On the other hand, as $Q_{t,x}$ is a probability measure,
    \begin{align}\label{eq:bdd-2}
        \int_{\mathbb{R}^{d} \setminus \Omega_{x}} \norm{y}^{k} Q_{t,x}(dy) \leq A^{k}(\norm{x}+R+1)^{k}
    \end{align}
    Combining \eqref{eq:bdd-1} and \eqref{eq:bdd-2} to control the upper bound in \eqref{eq:nabla-log-semigrp} finally establishes \eqref{eq:heat-kernel-bdd}.  
    
\end{proof}

This implies that the drift of the dynamic $\Schro$ bridge for each fixed $\eps > 0$, $\beta_{t}^{\eps}$, is in the appropriate tangent space to apply the restricted Pythagorean Theorem. We summarize in the following proposition.
\begin{proposition}[Consequence of Proposition \ref{prop:poly-bdd-nabla-log-semigrp}]\label{prop:sb-in-tangent-space}
   Under Assumption \ref{assumption:standard-assumptions}, it follows that the drift of the dynamic $\Schro$ bridge, $\beta_{t}^{\eps}$ defined in \eqref{eq:sb-sde}, satisfies $\beta_{t}^{\eps} \in \Tan{\rho_t^{\eps}}$ for all $t \in [0,1]$ and $\eps > 0$.  Additionally, $\nabla \varphi_{0 \to t}, \nabla \varphi_{t \to 0}, \nabla \varphi_{1 \to t}, \nabla \varphi_{t \to 1}$ as defined in \eqref{eq:mccann-interp} are all in $\Tan{\rho_t^{s}}$ for all $t \in [0,1]$ and $s \geq 0$.
\end{proposition}

\begin{proof}
When there exists $\Lambda_1, \Lambda_2 > 0$ such that $\Lambda_1 \Id \leq \nabla^2 f(x), \nabla^2 h(x) \leq \Lambda_2 \Id$ for all $x \in \mathbb{R}^{d}$, by \cite[Theorem 8]{chewi2022entropic} and the equivalence in \eqref{defn:entropic-potent}, for each $\eps > 0$ there exists $C > 0$ such that $\norm{\nabla \log b^{\eps}} \leq C(1+\norm{x})$. As $\beta_{t}^{\eps} = \eps \nabla \log R^{\eps}_{1-t}b^{\eps}$, it follows from Proposition \ref{prop:poly-bdd-nabla-log-semigrp} that $\norm{\beta_{t}^{\eps}}$ has linear growth for all $t \in [0,1)$. Recall that the $\Schro$ system forces that $R_{\eps}^{1}b^{\eps}$ is finite at each point. As $\rho_{t}^{\eps}$ has finite second moment, $\norm{\beta_{t}^{\eps}} \in L^2(\rho_{t}^{\eps})$. Thus, as $\beta_{t}^{\eps}$ is of gradient type, it follows that $\beta_t^{\eps} \in \Tan{\rho_{t}^{\eps}}$. At $t = 1$, we again use the $\Schro$ system to get that $\beta_{1}^{\eps} = -\eps \nabla f - \eps\nabla \log a^{\eps}$, is again of gradient type and linear growth, establishing the claim at $t = 1$. 

Similar argumentation establishes the second claim. By the Caffarelli Contraction Theorem, there exists $C >0$ such that for all $x \in \mathbb{R}^{d}$, $\norm{\nabla \varphi_{0 \to 1}(x)} \leq C(1+\norm{x})$. By the exact same argument of the previous paragraph, as each $\nabla \varphi_{0 \to t}$ is of gradient type and integrable in the appropriate space, $\nabla \varphi_{0 \to t} \in \Tan{\rho_t^{s}}$ for all $s \geq 0$ and $t \in [0,1]$. Identical argumentation holds for $\nabla \varphi_{1 \to t}$. For the other two terms, note that $\nabla \varphi_{t \to 0} = (\nabla \varphi_{0 \to t})^{-1}$ and $\nabla \varphi_{t \to 1} = (\nabla \varphi_{1\to t})^{-1}$ and then apply the same argument. 
\end{proof}

\begin{proposition}\label{prop:neg-sob-norm-poincare}
Let $\mu \in \cP(\mathbb{R}^{d})$ satisfy a Poincaré inequality with constant $C > 0$, recall the definition in \eqref{defn:poincare-inequality}. Fix $f \in H^{1}(\mu)$ such that $\int_{\mathbb{R}^{d}} f d\mu = 0$, and define a distribution $\sigma$ that acts on $\xi \in C_c^{\infty}(\mathbb{R}^{d})$ by
\begin{align*}
    \sigma : \nabla \xi \mapsto \int \xi fd\mu. 
\end{align*}
Then $\sigma = f\mu$ is an element of $\dot{H}^{-1}(\mu)$ and
\begin{align}\label{eq:pointcare-neg-sobolev-norm}
    \norm{f\mu}_{\dot{H}^{-1}(\mu)}^2 = \norm{\sigma}_{\dot{H}^{-1}(\mu)}^2 \leq C \norm{f}^2_{L^2(\mu)}. 
\end{align}
\end{proposition}
Before proving this proposition, we quickly recall that $\dot{H}^{1}(\mu)$ is a Hilbert space with inner product defined for all $\phi,\psi \in \dot{H}^{1}(\mu)$ by
\begin{align*}
    \langle \phi,\psi\rangle_{\dot{H}^{1}(\mu)} &:= \int_{\mathbb{R}^{d}}\langle \nabla \phi,\nabla \psi \rangle d\mu. 
\end{align*}
By the Riesz Representation Theorem, whenever $\norm{\sigma}_{\dot{H}^{-1}(\mu)} < +\infty$, there exists a $\phi(\sigma) \in \dot{H}^{1}(\mu)$ that agrees with the action of $\sigma$, i.e.\ such that
\begin{align}\label{eq:riesz-rep}
    \langle \sigma, \psi \rangle &= \langle \phi(\sigma),\psi\rangle_{\dot{H}^{1}(\mu)}, \quad \text{for all $\psi \in \dot{H}^{1}(\mu)$}. 
\end{align}
\begin{proof}[Proof of Proposition \ref{prop:neg-sob-norm-poincare}]
This argument is very similar to that of \cite[Theorem 8.3.1]{ambrosio2005gradient} around equations (8.3.10)-(8.3.12) therein. First, note $\sigma$ is well-defined as $f\mu$ integrates to zero. Next, because of the Poincaré inequality it holds for all $\xi \in C_c^{\infty}(\mathbb{R}^{d})$ that
\begin{align*}
     \int_{\mathbb{R}^{d}} \abs{\xi -\int_{\mathbb{R}^{d}}\xi d\mu}^2 d\mu \leq C\norm{\xi}_{\dot{H}^{1}(\mu)}.
\end{align*}
Thus, in the definition of $\sigma$ we can pick $\xi$ with mean zero, giving that $\sigma$ then has a continuous extension to $\xi \in \dot{H}^{1}(\mu)$. By the definition of $\norm{\cdot}_{\dot{H}^{-1}(\mu)}$ in \eqref{defn:neg-sob-norm} and by identifying $\sigma$ with its Riesz Representative in \eqref{eq:riesz-rep}, we see that \eqref{eq:pointcare-neg-sobolev-norm} holds. 
\end{proof}

\subsection{Gaussian Computations}\label{subsec:gaussian-computations}
Let us verify our assumptions for an example involving univariate Gaussian marginals. Fix $\sigma > 0$ and set $e^{-f} = N(0,1)$, $e^{-h} = N(0,\sigma^{2})$. Then the Brenier map is given by $\nabla \varphi(x) = \sigma x$, and $\ell_{\eps}$ is the density given by the bivariate Gaussian
\begin{align*}
    (X_0,X_{\eps}^*) \sim N\left(\begin{pmatrix}0\\0\end{pmatrix},\begin{pmatrix}1 & \sigma e^{-\eps/2\sigma} \\ \sigma e^{-\eps/2\sigma} & \sigma^2 \end{pmatrix}\right).
\end{align*}
Let $Z \sim N(0,1)$ be independent of the above pair and recall the notation in \eqref{eq:recip-marginal-flow}. For $t \in [0,1]$ it holds that $\rho_{t}^{\eps}$ is the law of the random variable $(1-t)X_0 + tX_{\eps}^* + \sqrt{\eps t(1-t)}Z$. Define $\sigma_t := (1-t)+t\sigma$, we compute that
\begin{align*}
    \Var_{Q^{\eps}}(\omega_t) &= (1-t)^{2}+t^2\sigma^2+2t(1-t)\sigma e^{-\eps/2\sigma}+\eps t(1-t)\\
    &= \sigma_t^2 + 2t(1-t)\sigma\left(e^{-\eps/2\sigma}-1+\frac{\eps}{2\sigma}\right) \\
    &= \sigma_t^2 + \frac{t(1-t)}{4\sigma}\eps^2 + o(\eps^2). 
\end{align*}
That is, in the notation of \eqref{eq:recip-marginal-flow} for all $t \in [0,1]$
\begin{align}\label{eq:gp-uncond}
   \rho_t^{\eps} \sim N\left(0,\sigma_t^2 + \frac{t(1-t)}{4\sigma}\eps^2 + o(\eps^2)\right).
\end{align}
Let $(\sigma_{t}^{\eps})^2$ denote the variance of $\rho_{t}^{\eps}$, and observe that this is smooth as a function of $(t,\eps) \in [0,1] \times [0,+\infty)$ and has a positive global lower bound. Thus, we see that Assumption \ref{assumption:technical-rd} holds in the case of univariate Gaussian marginals. 

We now manually compute the Fisher information and kinetic energy quantities that were essential in Theorems \ref{thm:rel-ent-tangent-markov-proj} and \ref{thm:mp-mld-approx-of-sb}.

\textbf{Fisher information.} First, recall for $\nu_{i} \sim N(\mu_i,\sigma_i^2)$ it holds that $I(\nu_i) = \sigma_i^{-2}$ and $I(\nu_1|\nu_2) = (\sigma_1^{-2}-\sigma_2^{-2})^{2}\sigma_{1}^2 + (\mu_1-\mu_2)^2\sigma_2^{-4}$. Thus, as the variance of $(\rho_t^{\eps}, t \in [0,1], \eps \geq 0)$ has a positive lower bound and is continuous in $t$ and $\eps$, it holds that $\int_0^1 I(\rho_t^{\eps})dt \to \int_0^{1} I(\rho_t^0)dt$ and $\int_0^{1} I(\rho_t^{\eps}|\rho_t^0)dt \to 0$ as $\eps \downarrow 0$. 
Moreover, by \eqref{eq:gp-uncond} it holds for each $t \in [0,1]$ that
\begin{align*}
     I(\rho_t^{\eps}) &= \frac{1}{\sigma_t^2 + \frac{t(1-t)}{4\sigma}\eps^2 + o(\eps^2)} = I(\rho_t^0)+O(\eps^2).
\end{align*}
Integrating in $t$, $\int_{0}^{1}I(\rho_t^\eps) dt = \int_0^1 I(\rho_t^0)dt + O(\eps^2)$.

\textbf{Kinetic energy.}  
Using the Taylor expansion about $x= 0$, $\sqrt{1+ax^2} = 1 + \frac{a}{2}x^2 +O(x^4)$, compute from \eqref{eq:gp-uncond} that
\begin{align*}
    \sqrt{\Var(\rho_t^{\eps})} &= \sigma_{t}\sqrt{1+\frac{t(1-t)}{4\sigma \sigma_t^2}\eps^2+o(\eps^2)} = \sigma_t\left(1+\frac{t(1-t)}{8\sigma \sigma_t^2}\eps^2 +O(\eps^4)\right).
\end{align*}
Thus, it holds for each $t \in [0,1]$ and each $h$ such that $(t+h) \in [0,1]$
\begin{align*}
    \Was{2}(\rho_{t+h}^{\eps},\rho_t^{\eps}) &=\abs{\sqrt{\Var(\rho_{t+h}^{\eps})}-\sqrt{\Var(\rho_t^{\eps})} }\\
    &= \abs{(\sigma_{t+h}-\sigma_t) + \left(\frac{(t+h)(1-(t+h))}{\sigma_{t+h}^2}-\frac{t(1-t)}{ \sigma_{t}^2}\right)\frac{\eps^2}{8\sigma} +O(\eps^4)}. 
\end{align*}
Dividing by $h$ and sending $h \to 0$ then gives for $\mathrm{Leb}$-a.e.\ $t \in [0,1]$
\begin{align*}
    |(\rho_t^{\eps})'| &= \lim\limits_{h \to 0}\frac{1}{\abs{h}}\Was{2}(\rho_{t+h}^{\eps},\rho_t^{\eps}) = \abs{(\sigma-1) + \frac{\eps^2}{8\sigma} \frac{d}{dt}\left(\frac{t(1-t)}{\sigma_t^2}\right) + O(\eps^4)}. 
\end{align*}
Altogether then, for small enough $\eps > 0$ it holds for $\mathrm{Leb}$-a.e.\ $t \in [0,1]$
\begin{align}\label{eq:ke-time-t}
    \norm{v_t^{\eps}}_{L^{2}(\rho_t^{\eps})}^2 &= (\sigma-1)^2 + \frac{d}{dt}\left(\frac{t(1-t)}{\sigma_t^2}\right) \frac{(\sigma-1)}{4\sigma}\eps^2 + O(\eps^4).
\end{align}
Now, integrating $t$ over $[0,1]$ finally gives
\begin{align}
    \frac{1}{2}\int_0^1 \norm{v_t^{\eps}}_{L^{2}(\rho_t^{\eps})}^2 dt &= \frac{1}{2}(\sigma-1)^2 + O(\eps^4) = \frac{1}{2}\Was{2}^{2}(e^{-f},e^{-h})+O(\eps^4). 
\end{align}

\subsection{A Nontrivial Class of Examples Satisfying Assumption \ref{assumption:technical-rd}.}

A nontrivial class of $f,h,\varphi': \mathbb{R} \to \mathbb{R}$ satisfying Assumption \ref{assumption:technical-rd}, specifically \eqref{eq:polynomial-dom-to-zero}, exists. The construction is as follows. Let $\eta \in C_c^{\infty}((-1,1))$ be nonzero and define the function
    \begin{align}\label{eqn:G-fnc}
        G(x) = x + \delta \eta(x),
    \end{align}
    where $\delta \neq 0$ is small enough that $G'(x) \geq c > 0$ for all $x \in \mathbb{R}$. Set $F = G^{-1}$.  Define a map $T: \mathbb{R} \to \mathbb{R}$ by
    \begin{align}\label{eq:T-fnc}
        T(x) = \int_0^{x} (G'(u))^2 du.
    \end{align}
    Note that since $T' > 0$, $T$ is the derivative of a convex function.

\begin{proposition}\label{prop:example-class}
    Let $ e^{-f(x)} := \frac{1}{\sqrt{2\pi}}\exp\left(-\frac{1}{2}G^2(x)\right)G'(x)$ and let $e^{-h} = T_{\#}e^{-f}$.
     Then Assumption \ref{assumption:technical-rd} is satisfied.
\end{proposition}
\begin{proof}
    By the definitions above and the change of variables formula, it holds that
    \begin{align}\label{eq:f}
        f(x) &= \frac{1}{2}G^2(x)-\log G'(x) + \frac{1}{2}\log(2\pi),
    \end{align}
    \begin{align}\label{eq:ht}
        h(T(x)) &= f(x) + \log T'(x). 
    \end{align}
    Since $G'$ has a positive global lower bound and $\eta \in C_c^{\infty}((-1,1))$, it holds that $f,h,T$ are all $C^{\infty}(\mathbb{R})$. Since $G^{(k)} \in C_c^{\infty}((-1,1))$ for all $k \geq 2$, it follows that the second derivatives of $f$ and $h$ are bounded above and below by positive constants for $\abs{\delta}$ small enough, and the second and third derivatives of $T$ are bounded too. That is, Assumption \ref{assumption:standard-assumptions} holds.

    It remains now to show the regularity in Assumption \ref{assumption:technical-rd}. The way that we do this is by recognizing the primal MLD as the image of the Ornstein-Uhlenbeck (OU) process under the function $F = G^{-1}$. Let $(Z_t, t \geq 0)$ denote the stationary OU process
    \begin{align}\label{eq:ou}
        dZ_t = -\frac{1}{2}Z_t dt + dB_t, \quad Z_0 \sim N(0,1). 
    \end{align}
    By the definition of $e^{-f}$, it holds that $F(Z_s) \sim e^{-f}$. Set $X_s := F(Z_s)$, then by $\Ito$'s formula and the Inverse Function Theorem
    \begin{align}\label{eq:xs-ito}
        dX_t = \left(-\frac{G(X_t)}{2G'(X_t)}-\frac{G''(X_t)}{2G'(X_t)^{3}}\right)dt + \frac{1}{G'(X_t)}dB_t, \quad X_0 \sim e^{-f}. 
    \end{align}
    Observe that \eqref{eq:xs-ito} is actually the primal MLD in our setting as $\sqrt{\left(T'(x)\right)^{-1}} = 1/G'(x)$ and $-\frac{1}{2}h'(T(x))= -\frac{G(x)}{2G'(x)}-\frac{G''(x)}{2G'(x)^{3}}$ by differentiating the formulas in \eqref{eq:ht} and \eqref{eq:f}. 

    Via the identification of the primal MLD as the image of the OU process, we can now construct the reciprocal surface as in \eqref{eq:defn-ust} in terms of the image of Gaussians under smooth functions. Define the function
    \begin{align}\label{eq:defn-j}
        J(x) = T (F(x)) \implies J'(x) = T'(F(x))F'(x) = G'(F(x)).
    \end{align}
    Recall from \eqref{eq:defn-ust} that for all $s \geq 0$ and $t \in [0,1]$,
    \begin{align}\label{eq:ust-appendix}
        U(t,s) = (1-t)X_0 + tT(X_s) + \sqrt{t(1-t)}W_{s}, \quad \rho_t^{s} = \mathrm{Law}(U(t,s)), 
    \end{align}
    where $(W_t, t \geq 0)$ is an independent standard Brownian motion. Since $(X_0,X_{s}) = (F(Z_0),F(Z_s))$, it holds that
    \begin{align}
        (X_0,X_s) \overset{(d)}{=} (F(U),F(e^{-s/2}U+\sqrt{1-e^{-s}}V))
    \end{align}
    for $U,V$ independent standard Gaussians. Thus, letting $W$ be yet another independent standard Gaussian, the reciprocal surface writes as
    \begin{align}\label{eq:sheet-as-gaussian-images}
        U(t,s) = (1-t)F(U) + tJ(e^{-s/2}U+\sqrt{1-e^{-s}}V) + \sqrt{st(1-t)}W.
    \end{align}

    From the identification in \eqref{eq:sheet-as-gaussian-images}, we will now establish the regularity in \eqref{eq:polynomial-dom-to-zero}. Indeed, we will show an even stronger statement: there exists $s_0, C, N > 0$ such that
    \begin{align}\label{eq:poly-bdds-apdx}
        \abs{\partial_t \partial_s \log \rho_t^s(x)}+\abs{\partial_x \partial_s \log \rho_t^s(x)} \leq Cs(1+\abs{x}^{N}), \quad \text{for all $(t,s,x) \in [0,1]\times [0,s_0] \times \mathbb{R}$.}
    \end{align}
    To smooth out the square roots appearing in \eqref{eq:sheet-as-gaussian-images}, we define the following smooth even function $c: \mathbb{R} \to [0,+\infty)$ by
    \begin{align}
        c(r) = \begin{cases}
            \sqrt{r^{-2}(1-e^{-r^2})}, & \text{$r \neq 0$,} \\
            1, & \text{$r = 0$.}
        \end{cases}
    \end{align}
    and the function 
    \begin{align}\label{eq:psi}
        \Psi_{t,r}(z_0,z_1) = (1-t)F(z_0) + t J(e^{-r^2/2}z_0+rc(r)z_1). 
    \end{align}
    We emphasize that $\Psi_{t,r}$ is defined for all $r \in \mathbb{R}$ and that since $V \overset{(d)}{=} -V$,
    \begin{align}\label{eq:evenness}
        \Psi_{t,r}(U,V) \overset{(d)}{=} \Psi_{t,-r}(U,V)
    \end{align}

    \textbf{Step 1: Remove Noise and Square Root.} We first prove regularity without adding the independent Gaussian noise in $U(t,s)$. Define the map
    \begin{align}
        (t,r,x) \mapsto n_{t,r}(x) = \mathrm{Law}(\Psi_{t,r}(U,V)).
    \end{align}
    From \eqref{eq:evenness}, $(t,r,x) \mapsto n_{t,r}(x)$ is even in $r$. From this symmetry and the formula in \eqref{eq:density-without-noise} below, it holds that $n_{t,r}(x)$ is \textbf{smooth} on $(t,r,x) \in [0,1] \times \mathbb{R} \times \mathbb{R}$. We now show that $n_{t,r}$ has global bounds similar to \eqref{eq:poly-bdds-apdx}.
    This follows from the Implicit Function Theorem. 

    First, since $F'$ and $J'$ are globally bounded from above and below by positive constants, there exist $c_1,C_1 > 0$ such that
    \begin{align}\label{eq:post-deriv-bdds}
        0 < c_1 \leq \partial_{z_0} \Psi_{t,r}(z_0,z_1) \leq C_1
    \end{align}
    for all $\abs{r}$ small enough and all $t \in [0,1]$, $z_0,z_1 \in \mathbb{R}$. Hence, for fixed $z_1$, the map $\Psi_{t,r}(\cdot,z_1)$ is a bijection on $\mathbb{R}$ for $\abs{r}$ small enough. That is, for each $x$ there is a unique $z_0$ such that $\Psi_{t,r}(z_0,z_1) = x$. Use this fact to define a function $\chi_{t,r}(\cdot,\cdot)$ such that
    \begin{align}\label{eq:chi-t-r}
        \chi_{t,r}(x,z_1) := z_0 \text{ such that } \Psi_{t,r}(z_0,z_1) = x.
    \end{align}
    Recall that all derivatives (first order and higher) of $F$ and $J$ are bounded. From implicit differentiation and \eqref{eq:post-deriv-bdds}, it follows that all derivatives of $\chi_{t,r}$ in $x$, $z_1$, $r$, and $t$ have polynomial growth. We note that there then exists $C >0$ such that for all $x,z_1 \in \mathbb{R}$, $t \in [0,1]$, and $\abs{r}$ small enough
    \begin{align}\label{eq:z0-linear-bdd}
        \abs{z_0} = \abs{\chi_{t,r}(x,z_1)} \leq C(1+\abs{x}+\abs{z_1}). 
    \end{align}
    
    Using \eqref{eq:chi-t-r}, the density of $\Psi_{t,r}(U,V)$ is equal to
    \begin{align}\label{eq:density-without-noise}
        n_{t,r}(x) = \int_{\mathbb{R}} \frac{\gamma(\chi_{t,r}(x,z_1))\gamma(z_1)}{\partial_{z_0}\Psi_{t,r}(\chi_{t,r}(x,z_1),z_1)} dz_1,
    \end{align}
    where $\gamma$ is the $N(0,1)$ density. By passing the derivatives in $t$, $r$, and $x$ in \eqref{eq:density-without-noise} to the integrand, it follows from the aforementioned polynomial bounds on the $\chi_{t,r}(\cdot,\cdot)$ and its derivatives that derivatives of $n_{t,r}(x)$ in $t$, $r$, and $x$ have polynomial bounds. 
    
    It now remains to argue why derivatives in $t$, $r$, and $x$ of $\log n_{t,r}$ have polynomial bounds. Define $a(z_0,z_1) = \partial_{z_0}\Psi_{t,r}(z_0,z_1)$. It then follows that
    \begin{align}\label{eq:score-to-exp}
        \partial_x \log n_{t,r}(x) = \frac{\partial_x n_{t,r}(x)}{n_{t,r}(x)} = \Exp{}\left[-\frac{U}{a(U,V)}-\frac{\partial_{z_0}a(U,V)}{a(U,V)^2} \mid \Psi_{t,r}(U,V)=x\right].
    \end{align}
    Analogous but more complicated expressions will hold for \eqref{eq:score-to-exp} when we apply further derivatives in $r$, $t$, and $x$ to the left hand side. Importantly, combining \eqref{eq:z0-linear-bdd}, \eqref{eq:post-deriv-bdds}, and the polynomial growth for the derivatives of $a$, the polynomial growth of $\abs{\partial_x \log n_{t,r}(x)}$ follows from obtaining a moment bound of $V|\Psi_{t,r}(U,V) =x $ that is polynomial in $x$. 
    
    Thus, we proceed by producing polynomial bounds for the moments of $V|\Psi_{t,r}(U,V) =x $. We note from \eqref{eq:post-deriv-bdds} and \eqref{eq:density-without-noise} that there exists $K_1, K_2 >0$ such that
    \begin{align}\label{eq:n-t-r-dens-bdds}
        n_{t,r}(x) \leq K_1, \quad n_{t,r}(x) \geq K_2 \exp(-K_1(1+x^2)).
    \end{align}
    The latter bound is attained by restricting the integral in \eqref{eq:density-without-noise} to the ball $\abs{z_1} \leq 1$. It then holds that
    \begin{align}
        P(\abs{V} > R| \Psi_{t,r}(U,V) =x)\leq \min\left(1,C\exp(C(1+x^2)-cR^{2})\right).
    \end{align}
    Integrating the above tail bound gives that $V|\Psi_{t,r}(U,V) =x$ has moments that grow polynomially in $x$. From \eqref{eq:score-to-exp}, it follows that $\abs{\partial_x \log n_{t,r}(x)}$ grows polynomially in $x$. By applying this same logic to higher derivatives, all expressions of the form $\abs{\partial_{t}^{i}\partial_{r}^{j}\partial_{x}^{\ell}n_{t,r}(x)}/n_{t,r}(x)$ have uniform polynomial bounds.

    \textbf{Step 2: Add Back in Noise.} With $n_{t,r}$ as defined in \eqref{eq:density-without-noise}, we now add back in the independent Gaussian noise provided by the random variable $W$ in \eqref{eq:sheet-as-gaussian-images}. Let $(P_{t}, t \geq 0)$ denote the standard Euclidean heat semigroup. We now define for $t \in [0,1]$ and $r \in \mathbb{R}$
    \begin{align}\label{eq:mtr}
        m_{t,r}(x) = P_{r^2 t(1-t)}n_{t,r}(x).
    \end{align}
    We note that $m_{t,r} = \rho_t^{r^2}$. From the heat equation, observe that
    \begin{align}\label{eq:r-t-derivs}
        \partial_{r}m_{t,r}(x) &= P_{r^2t(1-t)}\left(\partial_{r}n_{t,r}\right)(x)+rt(1-t)m_{t,r}''(x),
    \end{align}
    \begin{align}\label{eq:r-t-derivs2}
        \partial_{t}m_{t,r}(x) &= P_{r^2 t(1-t)}\left(\partial_{t}n_{t,r}\right)(x)+\frac{r^2(1-2t)}{2}m_{t,r}''(x).
    \end{align}
    Importantly, from the regularity we have established on $n_{t,r}$, it follows that the derivatives in \eqref{eq:r-t-derivs} and \eqref{eq:r-t-derivs2} are regular at $r = 0$ and $t = 0,1$. To emphasize, the map $(t,r,x) \mapsto m_{t,r}(x)$ is smooth in $(t,r,x) \in [0,1] \times \mathbb{R} \times \mathbb{R}$.
    
    Since we are now just adding Gaussian noise, it follows from the same conditional expectation argument above (this is also very similar to Proposition \ref{prop:poly-bdd-nabla-log-semigrp}) that the required mixed derivatives in $r$, $t$, and $x$ of $\log m_{t,r}$ have polynomial growth. In particular,
    \begin{align}\label{eq:almost-there}
        \abs{\partial_r^4 \partial_t \log m_{t,r}(x)}+\abs{\partial_r^4 \partial_x \log m_{t,r}(x)} \leq C(1+\abs{x}^{N})
    \end{align}
    for some $C, N > 0$. 

    \textbf{Step 3: Use symmetry to smooth out at zero.} Recall again that $m_{t,r} = \rho_t^{r^2}$. Since $m_{t,r}$ is smooth and even in $r$, it holds that $\left.\partial_{r} m_{t,r}\right|_{r = 0} = 0$. Similarly, by the evenness in $r$ of $m_{t,r}$ it holds in the sense of distributions that
    \begin{align}\label{eq:deriv-at-zero}
        \frac{1}{2}\left.\partial_{r}^2 m_{t,r}\right|_{r=0} = \left.\partial_{s}\rho_t^{s}\right|_{s = 0}.
    \end{align}
    Under \textbf{only} Assumption \ref{assumption:standard-assumptions}, Lemma \ref{lem:zero-initial-velocity} implies that $\left.\partial_{s} \rho_t^{s}\right|_{s = 0}$ is zero in the sense of distributions. Thus, since the left hand side of \eqref{eq:deriv-at-zero} exists classically, it must hold that $\left.\partial_{r}^2 m_{t,r}\right|_{r=0} = 0$ classically. Since $m_{t,r}(x) > 0$ for all $(t,r,x) \in [0,1] \times \mathbb{R} \times \mathbb{R}$, it holds classically that
    \begin{align}
        \left.\partial_{r}^2 \log m_{t,r}\right|_{r = 0} = \left.\left(\frac{\partial_r^2 m_{t,r}}{m_{t,r}}-\left(\frac{\partial_r m_{t,r}}{m_{t,r}}\right)^2\right)\right|_{r = 0}= 0.  
    \end{align}
    By the smoothness of $\log m_{t,r}$, we can apply $\partial_t$ and $\partial_x$ to $\partial_r^2 \log m_{t,r}$ and the resulting quantity still vanishes identically over $(t,x) \in [0,1] \times \mathbb{R}$ when $r = 0$.

    Finally then, let $B_{t,r}(x) = \partial_t \log m_{t,r}$ or $\partial_x \log m_{t,r}$. We have argued that $\partial_{r}^2B_{t,r}$ vanishes at $r = 0$, and since $m_{t,r}$ is even in $r$ it holds that $\partial_r B_{t,r} = 0$ and $\partial_{r}^{3} B_{t,r} =0$ at $r = 0$ as well. Thus, Taylor's theorem and \eqref{eq:almost-there} imply that for all $x \in \mathbb{R}$, $t \in [0,1]$, and $\abs{r}$ small enough there exist $C, N > 0$ such that
    \begin{align}\label{eq:deriv-inr}
        \abs{\partial_r B_{t,r}(x)} \leq C\abs{r}^3(1+\abs{x}^{N}).
    \end{align}
    Recall finally then that $\rho_t^{r^2} = m_{t,r}$. From the change of variables $s = r^2$, it holds that $\partial_s = (2r)^{-1}\partial_{r}$ for $r > 0$. Employing this change of variables and using \eqref{eq:deriv-inr} to extend $\partial_{s}B_{t,\sqrt{s}}$ continuously to $s = 0$, it follows that 
    \begin{align}
        \abs{\partial_s B_{t,\sqrt{s}}(x)} \leq Cs(1+\abs{x}^{N}).
    \end{align}
    Thus, the bound in \eqref{eq:poly-bdds-apdx} holds. 
\end{proof}

\bibliographystyle{alpha}
\bibliography{sample}

@article {brunick-mimicking13,
    AUTHOR = {Brunick, Gerard and Shreve, Steven},
     TITLE = {Mimicking an {I}t\^o{} process by a solution of a stochastic
              differential equation},
   JOURNAL = {Ann. Appl. Probab.},
  FJOURNAL = {The Annals of Applied Probability},
    VOLUME = {23},
      YEAR = {2013},
    NUMBER = {4},
     PAGES = {1584--1628},
      ISSN = {1050-5164,2168-8737},
   MRCLASS = {60G99 (60H10 91G20)},
  MRNUMBER = {3098443},
MRREVIEWER = {Hannah\ Geiss},
       DOI = {10.1214/12-aap881},
       URL = {https://doi-org.offcampus.lib.washington.edu/10.1214/12-aap881},
}

@article{WongYang,
title={Pseudo-{R}iemannian geometry encodes information geometry in optimal transport},
author={Wong, T.-K. L. and Yang, J.},
journal={Inf. Geom.},
volume={5},
number={1},
pages={131--159},
year={2022}
}

@article{KimMcCann,
    author ={Kim, Y.-H. and Mc{C}ann, R.},
    title = {Continuity,
curvature, and the general covariance of optimal transportation},
    journal = {J. Eur. Math. Soc.},
    year = {2007},
    volume={12}
}

@article {gyongy-mimicking-86,
    AUTHOR = {Gy{\"{o}}ngy, I.},
     TITLE = {Mimicking the one-dimensional marginal distributions of
              processes having an {I}t\^o{} differential},
   JOURNAL = {Probab. Theory Relat. Fields},
  FJOURNAL = {Probability Theory and Related Fields},
    VOLUME = {71},
      YEAR = {1986},
    NUMBER = {4},
     PAGES = {501--516},
      ISSN = {0178-8051,1432-2064},
   MRCLASS = {60H10},
  MRNUMBER = {833267},
MRREVIEWER = {R.\ Sh.\ Liptser},
       DOI = {10.1007/BF00699039},
       URL = {https://doi.org/10.1007/BF00699039},
}

@unpublished{CKP26,
    title={Mirror {L}angevin diffusions: Convergence rates and {M}arkov chain approximations}, 
    author={Benjamin Capdeville and Young-Heon Kim and Soumik Pal},
    year={2026},
    note={ar{X}iv preprint [math.PR]}
}

@article {AHMP25,
    AUTHOR = {Agarwal, Medha and Harchaoui, Zaid and Mulcahy, Garrett and
              Pal, Soumik},
     TITLE = {{L}angevin diffusion approximation to same marginal
              {S}chr\"odinger bridge},
   JOURNAL = {J. Funct. Anal.},
  FJOURNAL = {Journal of Functional Analysis},
    VOLUME = {291},
      YEAR = {2026},
    NUMBER = {1},
     PAGES = {Paper No. 111494, 27},
      ISSN = {0022-1236,1096-0783},
   MRCLASS = {49Q22 (60J60)},
  MRNUMBER = {5054456},
       DOI = {10.1016/j.jfa.2026.111494},
       URL = {https://doi-org.offcampus.lib.washington.edu/10.1016/j.jfa.2026.111494},
}

@article{ahn2021efficient,
  title={Efficient constrained sampling via the mirror-{L}angevin algorithm},
  author={Ahn, Kwangjun and Chewi, Sinho},
  journal={Advances in Neural Information Processing Systems},
  volume={34},
  pages={28405--28418},
  year={2021}
}

@article{DGW,
title={Transportation cost-information inequalities and applications to
random dynamical systems and diffusions},
author={Djellout, H. and Guillin, A. and Wu, H.},
journal={Ann. Probab.},
volume={32},
number={3B},
pages={2702--2732},
year={2004}
}

@book{revuz2004continuous,
  title={Continuous Martingales and Brownian Motion},
  author={Revuz, D. and Yor, M.},
  isbn={9783540643258},
  lccn={98053189},
  series={Grundlehren der mathematischen Wissenschaften},
  url={https://books.google.com/books?id=1ml95FLM5koC},
  year={2004},
  publisher={Springer Berlin Heidelberg}
}

@article {schroLeonard13,
    AUTHOR = {L{\'{e}}onard, Christian},
     TITLE = {A survey of the {S}chr\"{o}dinger problem and some of its
              connections with optimal transport},
   JOURNAL = {Discrete Contin. Dyn. Syst.},
  FJOURNAL = {Discrete and Continuous Dynamical Systems. Series A},
    VOLUME = {34},
      YEAR = {2014},
    NUMBER = {4},
     PAGES = {1533--1574},
      ISSN = {1078-0947},
   MRCLASS = {60J25 (46N10 60F10)},
  MRNUMBER = {3121631},
MRREVIEWER = {Nicolas Juillet},
       DOI = {10.3934/dcds.2014.34.1533},
       URL = {https://doi-org.offcampus.lib.washington.edu/10.3934/dcds.2014.34.1533},
}

@article {chiarini2022gradient,
    AUTHOR = {Chiarini, Alberto and Conforti, Giovanni and Greco, Giacomo
              and Tamanini, Luca},
     TITLE = {Gradient estimates for the {S}chr\"odinger potentials:
              convergence to the {B}renier map and quantitative stability},
   JOURNAL = {Comm. Partial Differential Equations},
  FJOURNAL = {Communications in Partial Differential Equations},
    VOLUME = {48},
      YEAR = {2023},
    NUMBER = {6},
     PAGES = {895--943},
      ISSN = {0360-5302,1532-4133},
   MRCLASS = {47D07 (49Q22 53C21 60E15)},
  MRNUMBER = {4643676},
MRREVIEWER = {Yuguo\ Lin},
       DOI = {10.1080/03605302.2023.2215527},
       URL = {https://doi.org/10.1080/03605302.2023.2215527},
}

@book{ambrosio2005gradient,
      AUTHOR = {Ambrosio, Luigi and Gigli, Nicola and Savar\'{e}, Giuseppe},
     TITLE = {Gradient flows in metric spaces and in the space of
              probability measures},
    SERIES = {Lectures in Mathematics ETH Z\"{u}rich},
   EDITION = {Second},
 PUBLISHER = {Birkh\"{a}user Verlag, Basel},
      YEAR = {2008},
     PAGES = {x+334},
      ISBN = {978-3-7643-8721-1},
   MRCLASS = {49-02 (28A33 35K55 35K90 49Q20 60B05)},
  MRNUMBER = {2401600},
MRREVIEWER = {Pietro\ Celada},
}

@article {conforti21deriv,
    AUTHOR = {Conforti, Giovanni and Tamanini, Luca},
     TITLE = {A formula for the time derivative of the entropic cost and
              applications},
   JOURNAL = {J. Funct. Anal.},
  FJOURNAL = {Journal of Functional Analysis},
    VOLUME = {280},
      YEAR = {2021},
    NUMBER = {11},
     PAGES = {Paper No. 108964, 48},
      ISSN = {0022-1236},
   MRCLASS = {49Q22 (28A33 49J45 49N80 60J60)},
  MRNUMBER = {4232667},
       DOI = {10.1016/j.jfa.2021.108964},
       URL = {https://doi.org/10.1016/j.jfa.2021.108964},
}

@article {csiszar75idiv,
    AUTHOR = {Csisz\'{a}r, I.},
     TITLE = {{$I$}-divergence geometry of probability distributions and
              minimization problems},
   JOURNAL = {Ann. Probability},
  FJOURNAL = {The Annals of Probability},
    VOLUME = {3},
      YEAR = {1975},
     PAGES = {146--158},
      ISSN = {0091-1798},
   MRCLASS = {62B10 (60E05)},
  MRNUMBER = {365798},
MRREVIEWER = {I. J. Good},
       DOI = {10.1214/aop/1176996454},
       URL = {https://doi-org.offcampus.lib.washington.edu/10.1214/aop/1176996454},
}

@article{pal2019difference,
      title={On the difference between entropic cost and the optimal transport cost}, 
      author={Soumik Pal},
      year={2024},
      journal={The Annals of Applied Probability},
      volume={34},
      number={1B},
      pages={1003--1028}
}

@article {nutz-weisel-22,
    AUTHOR = {Nutz, Marcel and Wiesel, Johannes},
     TITLE = {Entropic optimal transport: convergence of potentials},
   JOURNAL = {Probab. Theory Related Fields},
  FJOURNAL = {Probability Theory and Related Fields},
    VOLUME = {184},
      YEAR = {2022},
    NUMBER = {1-2},
     PAGES = {401--424},
      ISSN = {0178-8051},
   MRCLASS = {49Q22 (49J45)},
  MRNUMBER = {4498514},
MRREVIEWER = {Yuxin Ge},
       DOI = {10.1007/s00440-021-01096-8},
       URL = {https://doi.org/10.1007/s00440-021-01096-8},
}

@misc{pooladian2022entropic,
      title={Entropic estimation of optimal transport maps}, 
      author={Aram-Alexandre Pooladian and Jonathan Niles-Weed},
      year={2022},
      eprint={2109.12004},
      archivePrefix={ar{X}iv},
      primaryClass={math.ST}
}

@article {chewi2022entropic,
    AUTHOR = {Chewi, Sinho and Pooladian, Aram-Alexandre},
     TITLE = {An entropic generalization of {C}affarelli's contraction
              theorem via covariance inequalities},
   JOURNAL = {C. R. Math. Acad. Sci. Paris},
  FJOURNAL = {Comptes Rendus Math\'ematique. Acad\'emie des Sciences. Paris},
    VOLUME = {361},
      YEAR = {2023},
     PAGES = {1471--1482},
      ISSN = {1631-073X,1778-3569},
   MRCLASS = {60E15 (49N60 49Q22)},
  MRNUMBER = {4683324},
       DOI = {10.5802/crmath.486},
       URL = {https://doi.org/10.5802/crmath.486},
}

@article {fathigozlan20,
    AUTHOR = {Fathi, Max and Gozlan, Nathael and Prod'homme, Maxime},
     TITLE = {A proof of the {C}affarelli contraction theorem via entropic
              regularization},
   JOURNAL = {Calc. Var. Partial Differential Equations},
  FJOURNAL = {Calculus of Variations and Partial Differential Equations},
    VOLUME = {59},
      YEAR = {2020},
    NUMBER = {3},
     PAGES = {Paper No. 96, 18},
      ISSN = {0944-2669,1432-0835},
   MRCLASS = {60A10 (28A33 49J55 49Q22)},
  MRNUMBER = {4098037},
       DOI = {10.1007/s00526-020-01754-0},
       URL = {https://doi.org/10.1007/s00526-020-01754-0},
}

@book {bgl-markov,
    AUTHOR = {Bakry, Dominique and Gentil, Ivan and Ledoux, Michel},
     TITLE = {Analysis and geometry of {M}arkov diffusion operators},
    SERIES = {Grundlehren der mathematischen Wissenschaften [Fundamental
              Principles of Mathematical Sciences]},
    VOLUME = {348},
 PUBLISHER = {Springer, Cham},
      YEAR = {2014},
     PAGES = {xx+552},
      ISBN = {978-3-319-00226-2; 978-3-319-00227-9},
   MRCLASS = {60J25 (58J65 60J35 60J60)},
  MRNUMBER = {3155209},
MRREVIEWER = {Ming\ Liao},
       DOI = {10.1007/978-3-319-00227-9},
       URL = {https://doi.org/10.1007/978-3-319-00227-9},
}

@article {bobkov-fisher-22,
    AUTHOR = {Bobkov, Sergey G.},
     TITLE = {Upper bounds for {F}isher information},
   JOURNAL = {Electron. J. Probab.},
  FJOURNAL = {Electronic Journal of Probability},
    VOLUME = {27},
      YEAR = {2022},
     PAGES = {Paper No. 115, 44},
      ISSN = {1083-6489},
   MRCLASS = {60E15 (60E05 60G50)},
  MRNUMBER = {4477011},
MRREVIEWER = {Fraser\ Alexander\ Daly},
       DOI = {10.1214/22-ejp834},
       URL = {https://doi.org/10.1214/22-ejp834},
}

@article {leo-sb-to-kp12,
    AUTHOR = {L{\'{e}}onard, Christian},
     TITLE = {From the {S}chr\"{o}dinger problem to the
              {M}onge-{K}antorovich problem},
   JOURNAL = {J. Funct. Anal.},
  FJOURNAL = {Journal of Functional Analysis},
    VOLUME = {262},
      YEAR = {2012},
    NUMBER = {4},
     PAGES = {1879--1920},
      ISSN = {0022-1236,1096-0783},
   MRCLASS = {49Q20 (46N10)},
  MRNUMBER = {2873864},
MRREVIEWER = {Luca\ Granieri},
       DOI = {10.1016/j.jfa.2011.11.026},
       URL = {https://doi.org/10.1016/j.jfa.2011.11.026},
}

@article {vonrenesse-conf18,
    AUTHOR = {Conforti, Giovanni and Von Renesse, Max},
     TITLE = {Couplings, gradient estimates and logarithmic {S}obolev
              inequality for {L}angevin bridges},
   JOURNAL = {Probab. Theory Related Fields},
  FJOURNAL = {Probability Theory and Related Fields},
    VOLUME = {172},
      YEAR = {2018},
    NUMBER = {1-2},
     PAGES = {493--524},
      ISSN = {0178-8051,1432-2064},
   MRCLASS = {28D20 (47A63 47D07 47D08 60J60)},
  MRNUMBER = {3851837},
MRREVIEWER = {Hac\`ene\ Djellout},
       DOI = {10.1007/s00440-017-0814-9},
       URL = {https://doi.org/10.1007/s00440-017-0814-9},
}

@book {vershynin-hdp,
    AUTHOR = {Vershynin, Roman},
     TITLE = {High-dimensional probability},
    SERIES = {Cambridge Series in Statistical and Probabilistic Mathematics},
    VOLUME = {47},
      NOTE = {An introduction with applications in data science,
              With a foreword by Sara van de Geer},
 PUBLISHER = {Cambridge University Press, Cambridge},
      YEAR = {2018},
     PAGES = {xiv+284},
      ISBN = {978-1-108-41519-4},
   MRCLASS = {60-01 (60B05 60B20 60E15 60Fxx 62H25)},
  MRNUMBER = {3837109},
MRREVIEWER = {Sasha\ Sodin},
       DOI = {10.1017/9781108231596},
       URL = {https://doi.org/10.1017/9781108231596},
}

@article {bgn-eot-gld,
    AUTHOR = {Bernton, Espen and Ghosal, Promit and Nutz, Marcel},
     TITLE = {Entropic optimal transport: geometry and large deviations},
   JOURNAL = {Duke Math. J.},
  FJOURNAL = {Duke Mathematical Journal},
    VOLUME = {171},
      YEAR = {2022},
    NUMBER = {16},
     PAGES = {3363--3400},
      ISSN = {0012-7094,1547-7398},
   MRCLASS = {49Q22 (60F10)},
  MRNUMBER = {4505361},
MRREVIEWER = {Marc\ Sedjro},
       DOI = {10.1215/00127094-2022-0035},
       URL = {https://doi.org/10.1215/00127094-2022-0035},
}

@unpublished{leonard2011stochastic,
      title={Stochastic derivatives and generalized h-transforms of {M}arkov processes}, 
      author={Christian L{\'{e}}onard},
      year={2011},
      note={ar{X}iv preprint [math.PR]}
}

@inproceedings{gentil2017analogy,
  title={About the analogy between optimal transport and minimal entropy},
  author={Gentil, Ivan and L{\'e}onard, Christian and Ripani, Luigia},
  booktitle={Annales de la Facult{\'e} des sciences de Toulouse: Math{\'e}matiques},
  volume={26},
  number={3},
  pages={569--600},
  year={2017}
}

@article {sinkhorn-OG,
    AUTHOR = {Sinkhorn, Richard and Knopp, Paul},
     TITLE = {Concerning nonnegative matrices and doubly stochastic
              matrices},
   JOURNAL = {Pacific J. Math.},
  FJOURNAL = {Pacific Journal of Mathematics},
    VOLUME = {21},
      YEAR = {1967},
     PAGES = {343--348},
      ISSN = {0030-8730,1945-5844},
   MRCLASS = {15.65},
  MRNUMBER = {210731},
MRREVIEWER = {J.\ G.\ Mauldon},
       URL = {http://projecteuclid.org/euclid.pjm/1102992505},
}

@book {fpk-bogachev,
    AUTHOR = {Bogachev, Vladimir I. and Krylov, Nicolai V. and R\"{o}ckner,
              Michael and Shaposhnikov, Stanislav V.},
     TITLE = {Fokker-{P}lanck-{K}olmogorov equations},
    SERIES = {Mathematical Surveys and Monographs},
    VOLUME = {207},
 PUBLISHER = {American Mathematical Society, Providence, RI},
      YEAR = {2015},
     PAGES = {xii+479},
      ISBN = {978-1-4704-2558-6},
   MRCLASS = {35-02 (60J35 60J60)},
  MRNUMBER = {3443169},
MRREVIEWER = {Zhen-Qing\ Chen},
       DOI = {10.1090/surv/207},
       URL = {https://doi.org/10.1090/surv/207},
}

@article{cuturi2013sinkhorn,
  title={Sinkhorn distances: Lightspeed computation of optimal transport},
  author={Cuturi, Marco},
  journal={Advances in Neural Information Processing Systems},
  volume={26},
  year={2013}
}

@incollection {azencott84,
    AUTHOR = {Azencott, Robert},
     TITLE = {Densit\'e{} des diffusions en temps petit: d\'eveloppements
              asymptotiques. {I}},
 BOOKTITLE = {Seminar on probability, {XVIII}},
    SERIES = {Lecture Notes in Math.},
    VOLUME = {1059},
     PAGES = {402--498},
 PUBLISHER = {Springer, Berlin},
      YEAR = {1984},
      ISBN = {3-540-13332-1},
   MRCLASS = {60J60 (58G11 58G32)},
  MRNUMBER = {770974},
MRREVIEWER = {Kazuaki\ Taira},
       DOI = {10.1007/BFb0100057},
       URL = {https://doi.org/10.1007/BFb0100057},
}

@article {leo-reciprocal14,
    AUTHOR = {L\'eonard, Christian and R{\oe}lly, Sylvie and Zambrini,
              Jean-Claude},
     TITLE = {Reciprocal processes. {A} measure-theoretical point of view},
   JOURNAL = {Probab. Surv.},
  FJOURNAL = {Probability Surveys},
    VOLUME = {11},
      YEAR = {2014},
     PAGES = {237--269},
      ISSN = {1549-5787},
   MRCLASS = {60J25 (60A10 60G07 60G60)},
  MRNUMBER = {3269228},
MRREVIEWER = {Martynas\ Manstavi\v cius},
       DOI = {10.1214/13-PS220},
       URL = {https://doi.org/10.1214/13-PS220},
}

@article {conforti-weak-semi24,
    AUTHOR = {Conforti, Giovanni},
     TITLE = {Weak semiconvexity estimates for {S}chr\"odinger potentials
              and logarithmic {S}obolev inequality for {S}chr\"odinger
              bridges},
   JOURNAL = {Probab. Theory Related Fields},
  FJOURNAL = {Probability Theory and Related Fields},
    VOLUME = {189},
      YEAR = {2024},
    NUMBER = {3-4},
     PAGES = {1045--1071},
      ISSN = {0178-8051,1432-2064},
   MRCLASS = {49Q22 (35G50 39B62 49L12 60J60)},
  MRNUMBER = {4771110},
       DOI = {10.1007/s00440-024-01264-6},
       URL = {https://doi.org/10.1007/s00440-024-01264-6},
}

@article {THEbrenier,
    AUTHOR = {Brenier, Yann},
     TITLE = {Polar factorization and monotone rearrangement of
              vector-valued functions},
   JOURNAL = {Comm. Pure Appl. Math.},
  FJOURNAL = {Communications on Pure and Applied Mathematics},
    VOLUME = {44},
      YEAR = {1991},
    NUMBER = {4},
     PAGES = {375--417},
      ISSN = {0010-3640,1097-0312},
   MRCLASS = {46E40 (35Q99 46E99 49Q99)},
  MRNUMBER = {1100809},
MRREVIEWER = {Robert\ McOwen},
       DOI = {10.1002/cpa.3160440402},
       URL = {https://doi.org/10.1002/cpa.3160440402},
}

@article {divol2024tightstabilityboundsentropic,
    AUTHOR = {Divol, Vincent and Niles-Weed, Jonathan and Pooladian,
              Aram-Alexandre},
     TITLE = {Tight stability bounds for entropic {B}renier maps},
   JOURNAL = {Int. Math. Res. Not. IMRN},
  FJOURNAL = {International Mathematics Research Notices. IMRN},
      YEAR = {2025},
    NUMBER = {7},
     PAGES = {Paper No. rnaf078, 17},
      ISSN = {1073-7928,1687-0247},
   MRCLASS = {49Q22 (49K40)},
  MRNUMBER = {4888259},
MRREVIEWER = {Sinho\ Chewi},
       DOI = {10.1093/imrn/rnaf078},
       URL = {https://doi-org.offcampus.lib.washington.edu/10.1093/imrn/rnaf078},
}

@InProceedings{zhang-mld-20a,
  title = 	 {Wasserstein Control of {M}irror {L}angevin {M}onte {C}arlo},
  author =       {Zhang, Kelvin Shuangjian and Peyr\'e, Gabriel and Fadili, Jalal and Pereyra, Marcelo},
  booktitle = 	 {Proceedings of Thirty Third Conference on Learning Theory},
  pages = 	 {3814--3841},
  year = 	 {2020},
  editor = 	 {Abernethy, Jacob and Agarwal, Shivani},
  volume = 	 {125},
  series = 	 {Proceedings of Machine Learning Research},
  month = 	 {09--12 Jul},
  publisher =    {PMLR},
}

@article {varadhan-diff-ldp67,
    AUTHOR = {Varadhan, S. R. S.},
     TITLE = {Diffusion processes in a small time interval},
   JOURNAL = {Comm. Pure Appl. Math.},
  FJOURNAL = {Communications on Pure and Applied Mathematics},
    VOLUME = {20},
      YEAR = {1967},
     PAGES = {659--685},
      ISSN = {0010-3640,1097-0312},
   MRCLASS = {60.62 (35.00)},
  MRNUMBER = {217881},
MRREVIEWER = {H.\ P.\ McKean, Jr.},
       DOI = {10.1002/cpa.3160200404},
       URL = {https://doi.org/10.1002/cpa.3160200404},
}

@article{benamou2000computational,
  title={A computational fluid mechanics solution to the {M}onge-{K}antorovich mass transfer problem},
  author={Benamou, Jean-David and Brenier, Yann},
  journal={Numerische Mathematik},
  volume={84},
  number={3},
  pages={375--393},
  year={2000},
  publisher={Springer-Verlag Berlin/Heidelberg}
}

@book{shima2007geometry,
  title={The Geometry Of Hessian Structures},
  author={Shima, H.},
  isbn={9789814477024},
  url={https://books.google.com/books?id=wLPICgAAQBAJ},
  year={2007},
  publisher={World Scientific Publishing Company}
}

@book {riemannian-manifolds-lee,
    AUTHOR = {Lee, John M.},
     TITLE = {Introduction to {R}iemannian manifolds},
    SERIES = {Graduate Texts in Mathematics},
    VOLUME = {176},
   EDITION = {Second},
 PUBLISHER = {Springer, Cham},
      YEAR = {2018},
     PAGES = {xiii+437},
      ISBN = {978-3-319-91754-2; 978-3-319-91755-9},
   MRCLASS = {53-01 (53B20 53B30 53C20 53C21)},
  MRNUMBER = {3887684},
MRREVIEWER = {Robert\ J.\ Low},
}

@unpublished{gozlan2025globalregularityestimatesoptimal,
      title={Global Regularity Estimates for Optimal Transport via Entropic Regularisation}, 
      author={Nathael Gozlan and Maxime Sylvestre},
      year={2025},
      note={ar{X}iv preprint [math.FA]} 
}

@article {kolesnikov-hessian-metric,
    AUTHOR = {Kolesnikov, Alexander V.},
     TITLE = {Hessian metrics, {$CD(K,N)$}-spaces, and optimal
              transportation of log-concave measures},
   JOURNAL = {Discrete Contin. Dyn. Syst.},
  FJOURNAL = {Discrete and Continuous Dynamical Systems. Series A},
    VOLUME = {34},
      YEAR = {2014},
    NUMBER = {4},
     PAGES = {1511--1532},
      ISSN = {1078-0947,1553-5231},
   MRCLASS = {35J60 (35J96 46E35 58J60)},
  MRNUMBER = {3121630},
MRREVIEWER = {Ahmed\ Mohammed},
       DOI = {10.3934/dcds.2014.34.1511},
       URL = {https://doi.org/10.3934/dcds.2014.34.1511},
}

@article {benarous-expan-88,
    AUTHOR = {Ben Arous, G.},
     TITLE = {D\'eveloppement asymptotique du noyau de la chaleur
              hypoelliptique hors du cut-locus},
   JOURNAL = {Ann. Sci. \'Ecole Norm. Sup. (4)},
  FJOURNAL = {Annales Scientifiques de l'\'Ecole Normale Sup\'erieure.
              Quatri\`eme S\'erie},
    VOLUME = {21},
      YEAR = {1988},
    NUMBER = {3},
     PAGES = {307--331},
      ISSN = {0012-9593},
   MRCLASS = {60H99 (35H05 58G11 58G32 62F10)},
  MRNUMBER = {974408},
MRREVIEWER = {R\'emi\ L\'eandre},
       URL = {http://www.numdam.org/item?id=ASENS_1988_4_21_3_307_0},
}

@article {gigli-ot-manifold,
    AUTHOR = {Gigli, Nicola},
     TITLE = {Second order analysis on {$(\mathscr{P}_2(M),W_2)$}},
   JOURNAL = {Mem. Amer. Math. Soc.},
  FJOURNAL = {Memoirs of the American Mathematical Society},
    VOLUME = {216},
      YEAR = {2012},
    NUMBER = {1018},
     PAGES = {xii+154},
      ISSN = {0065-9266,1947-6221},
      ISBN = {978-0-8218-5309-2},
   MRCLASS = {58B20 (49Q20)},
  MRNUMBER = {2920736},
MRREVIEWER = {Beno\^it\ Kloeckner},
       DOI = {10.1090/S0065-9266-2011-00619-2},
       URL = {https://doi.org/10.1090/S0065-9266-2011-00619-2},
}

@article {mccann-ot-maifold01,
    AUTHOR = {McCann, Robert J.},
     TITLE = {Polar factorization of maps on {R}iemannian manifolds},
   JOURNAL = {Geom. Funct. Anal.},
  FJOURNAL = {Geometric and Functional Analysis},
    VOLUME = {11},
      YEAR = {2001},
    NUMBER = {3},
     PAGES = {589--608},
      ISSN = {1016-443X,1420-8970},
   MRCLASS = {58E15 (46N10 49Q20 53C20)},
  MRNUMBER = {1844080},
MRREVIEWER = {Lucio\ Renato\ Berrone},
       DOI = {10.1007/PL00001679},
       URL = {https://doi.org/10.1007/PL00001679},
}

@article{shi2023diffusion,
  title={Diffusion {S}chr{\"o}dinger bridge matching},
  author={Shi, Yuyang and De Bortoli, Valentin and Campbell, Andrew and Doucet, Arnaud},
  journal={Advances in Neural Information Processing Systems},
  volume={36},
  pages={62183--62223},
  year={2023}
}

@book {heat-kernel-manifold-grig09,
    AUTHOR = {Grigor'yan, Alexander},
     TITLE = {Heat kernel and analysis on manifolds},
    SERIES = {AMS/IP Studies in Advanced Mathematics},
    VOLUME = {47},
 PUBLISHER = {American Mathematical Society, Providence, RI; International
              Press, Boston, MA},
      YEAR = {2009},
     PAGES = {xviii+482},
      ISBN = {978-0-8218-4935-4},
   MRCLASS = {58J35 (31B05 31C12 35K08 35P15 35R01 47D07 58J50)},
  MRNUMBER = {2569498},
MRREVIEWER = {Thierry\ Coulhon},
       DOI = {10.1090/amsip/047},
       URL = {https://doi.org/10.1090/amsip/047},
}

@unpublished{mordant24selfEOT,
    title={The entropic optimal (self-)transport problem: Limit distributions for decreasing regularization with application to score function estimation}, 
      author={Gilles Mordant},
      year={2024},
      note={ar{X}iv preprint [math.ST], 2412.12007v2}
}

@article {conforti-recip-characteristics18,
    AUTHOR = {Conforti, Giovanni},
     TITLE = {Fluctuations of bridges, reciprocal characteristics and
              concentration of measure},
   JOURNAL = {Ann. Inst. Henri Poincar\'e{} Probab. Stat.},
  FJOURNAL = {Annales de l'Institut Henri Poincar\'e{} Probabilit\'es et
              Statistiques},
    VOLUME = {54},
      YEAR = {2018},
    NUMBER = {3},
     PAGES = {1432--1463},
      ISSN = {0246-0203,1778-7017},
   MRCLASS = {60J27 (60J75)},
  MRNUMBER = {3825887},
       DOI = {10.1214/17-AIHP844},
       URL = {https://doi.org/10.1214/17-AIHP844},
}

@article {mccann-interp97,
    AUTHOR = {McCann, Robert J.},
     TITLE = {A convexity principle for interacting gases},
   JOURNAL = {Adv. Math.},
  FJOURNAL = {Advances in Mathematics},
    VOLUME = {128},
      YEAR = {1997},
    NUMBER = {1},
     PAGES = {153--179},
      ISSN = {0001-8708,1090-2082},
   MRCLASS = {82B05 (26B25 90C08)},
  MRNUMBER = {1451422},
MRREVIEWER = {Carlos\ Matr\'an},
       DOI = {10.1006/aima.1997.1634},
       URL = {https://doi-org.offcampus.lib.washington.edu/10.1006/aima.1997.1634},
}

@unpublished{MP25,
    title={Diffusion Approximations to {S}chr\"odinger Bridges on Manifolds},
    author={Mulcahy, G. and Pal, S.},
    year={2025},
    note={ar{X}iv preprint [math.PR]}
}

@unpublished{nutz2026entropicregularizationmongesproblem,
      title={Entropic regularization of {M}onge's problem}, 
      author={Marcel Nutz and Chenyang Zhong},
      year={2026},
      note={ar{X}iv preprint [math.OC]}, 
}

@unpublisched{aryan2025entropicselectionprinciplemonges,
      title={Entropic Selection Principle for {M}onge's Optimal Transport}, 
      author={Shrey Aryan and Promit Ghosal},
      year={2025},
      note={ar{X}iv preprint [math.PR]}, 
}

@incollection {leo-gis12,
    AUTHOR = {L{\'e}onard, Christian},
     TITLE = {Girsanov theory under a finite entropy condition},
 BOOKTITLE = {S\'eminaire de {P}robabilit\'es {XLIV}},
    SERIES = {Lecture Notes in Math.},
    VOLUME = {2046},
     PAGES = {429--465},
 PUBLISHER = {Springer, Heidelberg},
      YEAR = {2012},
      ISBN = {978-3-642-27460-2},
   MRCLASS = {60G07 (60G44 60J60 60J75)},
  MRNUMBER = {2953359},
MRREVIEWER = {Vilmos\ Prokaj},
       DOI = {10.1007/978-3-642-27461-9\_20},
       URL = {https://doi.org/10.1007/978-3-642-27461-9_20},
}

@article {THEcaffarelli,
    AUTHOR = {Caffarelli, Luis A.},
     TITLE = {The regularity of mappings with a convex potential},
   JOURNAL = {J. Amer. Math. Soc.},
  FJOURNAL = {Journal of the American Mathematical Society},
    VOLUME = {5},
      YEAR = {1992},
    NUMBER = {1},
     PAGES = {99--104},
      ISSN = {0894-0347,1088-6834},
   MRCLASS = {35B65 (35A30 35J60)},
  MRNUMBER = {1124980},
       DOI = {10.2307/2152752},
       URL = {https://doi.org/10.2307/2152752},}

@article{peluchetti-23,
  author  = {Stefano Peluchetti},
  title   = {Diffusion Bridge Mixture Transports, {S}chr{\"{o}}dinger Bridge Problems and Generative Modeling},
  journal = {Journal of Machine Learning Research},
  year    = {2023},
  volume  = {24},
  number  = {374},
  pages   = {1--51},
  url     = {http://jmlr.org/papers/v24/23-0527.html}
}

@inproceedings{SCD-dsb-25,
  title={Exponential Convergence Guarantees for {I}terative {M}arkovian {F}itting},
  author={Silveri, Marta Gentiloni and Conforti, Giovanni and Durmus, Alain Oliviero},
  booktitle={The Thirty-ninth Annual Conference on Neural Information Processing Systems},
  year={2025},
}

@article{gentiloni2024theoretical,
  title={Theoretical guarantees in {KL} for diffusion flow matching},
  author={Gentiloni Silveri, Marta and Durmus, Alain and Conforti, Giovanni},
  journal={Advances in Neural Information Processing Systems},
  volume={37},
  pages={138432--138473},
  year={2024}
}

@incollection {klartag-log-concave14,
    AUTHOR = {Klartag, Bo'az},
     TITLE = {Logarithmically-concave moment measures {I}},
 BOOKTITLE = {Geometric aspects of functional analysis},
    SERIES = {Lecture Notes in Math.},
    VOLUME = {2116},
     PAGES = {231--260},
 PUBLISHER = {Springer, Cham},
      YEAR = {2014},
      ISBN = {978-3-319-09476-2; 978-3-319-09477-9},
   MRCLASS = {58J65 (53C55 58J50 60J60)},
  MRNUMBER = {3364690},
MRREVIEWER = {Ming\ Liao},
       DOI = {10.1007/978-3-319-09477-9\_16},
       URL = {https://doi.org/10.1007/978-3-319-09477-9_16},
}

@article {mikami04,
    AUTHOR = {Mikami, Toshio},
     TITLE = {{M}onge's problem with a quadratic cost by the zero-noise limit
              of {$h$}-path processes},
   JOURNAL = {Probab. Theory Related Fields},
  FJOURNAL = {Probability Theory and Related Fields},
    VOLUME = {129},
      YEAR = {2004},
    NUMBER = {2},
     PAGES = {245--260},
      ISSN = {0178-8051,1432-2064},
   MRCLASS = {60J25},
  MRNUMBER = {2063377},
MRREVIEWER = {Vadim\ A.\ Ka\u imanovich},
       DOI = {10.1007/s00440-004-0340-4},
       URL = {https://doi.org/10.1007/s00440-004-0340-4},
}

@article {carlierEOTgeneralcost23,
    AUTHOR = {Carlier, Guillaume and Pegon, Paul and Tamanini, Luca},
     TITLE = {Convergence rate of general entropic optimal transport costs},
   JOURNAL = {Calc. Var. Partial Differential Equations},
  FJOURNAL = {Calculus of Variations and Partial Differential Equations},
    VOLUME = {62},
      YEAR = {2023},
    NUMBER = {4},
     PAGES = {Paper No. 116, 28},
      ISSN = {0944-2669,1432-0835},
   MRCLASS = {49Q22 (49K40 49N15 94A17)},
  MRNUMBER = {4565039},
MRREVIEWER = {Hao\ Wu},
       DOI = {10.1007/s00526-023-02455-0},
       URL = {https://doi.org/10.1007/s00526-023-02455-0},
}

@article {malamut-syl25,
    AUTHOR = {Malamut, Hugo and Sylvestre, Maxime},
     TITLE = {Convergence rates of the regularized optimal transport:
              disentangling suboptimality and entropy},
   JOURNAL = {SIAM J. Math. Anal.},
  FJOURNAL = {SIAM Journal on Mathematical Analysis},
    VOLUME = {57},
      YEAR = {2025},
    NUMBER = {3},
     PAGES = {2533--2558},
      ISSN = {0036-1410,1095-7154},
   MRCLASS = {49Q22 (49K40 94A17)},
  MRNUMBER = {4907179},
MRREVIEWER = {Lukas\ Koch},
       DOI = {10.1137/23M1591554},
       URL = {https://doi.org/10.1137/23M1591554},
}

@article {nutzweisel-stab23,
    AUTHOR = {Nutz, Marcel and Wiesel, Johannes},
     TITLE = {Stability of {S}chr\"odinger potentials and convergence of
              {S}inkhorn's algorithm},
   JOURNAL = {Ann. Probab.},
  FJOURNAL = {The Annals of Probability},
    VOLUME = {51},
      YEAR = {2023},
    NUMBER = {2},
     PAGES = {699--722},
      ISSN = {0091-1798,2168-894X},
   MRCLASS = {60F15 (49Q22 60F05 90C05)},
  MRNUMBER = {4546630},
MRREVIEWER = {Oliver\ Johnson},
       DOI = {10.1214/22-aop1611},
       URL = {https://doi.org/10.1214/22-aop1611},
}

@book {statOTbook,
    AUTHOR = {Chewi, Sinho and Niles-Weed, Jonathan and Rigollet, Philippe},
     TITLE = {Statistical optimal transport},
    SERIES = {Lecture Notes in Mathematics},
    VOLUME = {2364},
      NOTE = {\'Ecole d'\'Et\'e{} de Probabilit\'es de Saint-Flour
              XLIX---2019,
              \'Ecole d'\'Et\'e{} de Probabilit\'es de Saint-Flour.
              [Saint-Flour Probability Summer School]},
 PUBLISHER = {Springer, Cham},
      YEAR = {[2025] \copyright 2025},
     PAGES = {xiv+258},
      ISBN = {978-3-031-85159-9; 978-3-031-85160-5},
   MRCLASS = {49-01 (49Q22 60D05 62Gxx)},
  MRNUMBER = {4901218},
       DOI = {10.1007/978-3-031-85160-5},
       URL = {https://doi.org/10.1007/978-3-031-85160-5},
}

@article {follmergantert97,
    AUTHOR = {F\"ollmer, Hans and Gantert, Nina},
     TITLE = {Entropy minimization and {S}chr\"odinger processes in infinite
              dimensions},
   JOURNAL = {Ann. Probab.},
  FJOURNAL = {The Annals of Probability},
    VOLUME = {25},
      YEAR = {1997},
    NUMBER = {2},
     PAGES = {901--926},
      ISSN = {0091-1798,2168-894X},
   MRCLASS = {60J60 (60F10 60J45)},
  MRNUMBER = {1434130},
MRREVIEWER = {Luis\ G.\ Gorostiza},
       DOI = {10.1214/aop/1024404423},
       URL = {https://doi.org/10.1214/aop/1024404423},
}

@article {jamison-rp-74,
    AUTHOR = {Jamison, Benton},
     TITLE = {Reciprocal processes},
   JOURNAL = {Z. Wahrscheinlichkeitstheorie und Verw. Gebiete},
  FJOURNAL = {Zeitschrift f\"ur Wahrscheinlichkeitstheorie und Verwandte
              Gebiete},
    VOLUME = {30},
      YEAR = {1974},
     PAGES = {65--86},
   MRCLASS = {60J35},
  MRNUMBER = {359016},
MRREVIEWER = {Mamoru\ Kanda},
       DOI = {10.1007/BF00532864},
       URL = {https://doi.org/10.1007/BF00532864},
}

@article {legervialard23,
    AUTHOR = {L{\'{e}}ger, Flavien and Vialard, Fran{\c{c}}ois-Xavier},
     TITLE = {A geometric {L}aplace method},
   JOURNAL = {Pure Appl. Anal.},
  FJOURNAL = {Pure and Applied Analysis},
    VOLUME = {5},
      YEAR = {2023},
    NUMBER = {4},
     PAGES = {1041--1080},
      ISSN = {2578-5893,2578-5885},
   MRCLASS = {53C15 (49Q22 53B12)},
  MRNUMBER = {4680531},
MRREVIEWER = {Gabriel\ Khan},
       DOI = {10.2140/paa.2023.5.1041},
       URL = {https://doi.org/10.2140/paa.2023.5.1041},
}

@article{GalichonSalanie09,
  author  = {Galichon, Alfred and Salani{\'e}, Bernard},
  title   = {Matching with Trade-Offs: Revealed Preferences Over Competing Characteristics},
  journal = {SSRN Electronic Journal},
  year    = {2009},
  doi     = {10.2139/ssrn.1487307}
}

@article {trevisan16,
    AUTHOR = {Trevisan, Dario},
     TITLE = {Well-posedness of multidimensional diffusion processes with
              weakly differentiable coefficients},
   JOURNAL = {Electron. J. Probab.},
  FJOURNAL = {Electronic Journal of Probability},
    VOLUME = {21},
      YEAR = {2016},
     PAGES = {Paper No. 22, 41},
      ISSN = {1083-6489},
   MRCLASS = {60J60 (35B30 35Q84 35R60)},
  MRNUMBER = {3485364},
MRREVIEWER = {Gabriela\ Marinoschi},
       DOI = {10.1214/16-EJP4453},
       URL = {https://doi.org/10.1214/16-EJP4453},
}

@article {kadota70,
    AUTHOR = {Kadota, T. T. and Shepp, L. A.},
     TITLE = {Conditions for absolute continuity between a certain pair of
              probability measures},
   JOURNAL = {Z. Wahrscheinlichkeitstheorie und Verw. Gebiete},
  FJOURNAL = {Zeitschrift f\"ur Wahrscheinlichkeitstheorie und Verwandte
              Gebiete},
    VOLUME = {16},
      YEAR = {1970},
     PAGES = {250--260},
   MRCLASS = {60.40},
  MRNUMBER = {278344},
MRREVIEWER = {T.\ Pitcher},
       DOI = {10.1007/BF00534599},
       URL = {https://doi.org/10.1007/BF00534599},
}

\end{document}